\documentclass[11pt,reqno]{amsart}
\usepackage[dvipsnames,table]{xcolor}
\usepackage{enumitem}
\usepackage{manyfoot}
\usepackage{amsmath,amssymb,amsthm,mathtools}
\usepackage{calligra} 
\usepackage{siunitx}
\usepackage{graphicx}
\usepackage{caption}
\usepackage{subcaption}
\usepackage{wrapfig}
\usepackage{float}
\usepackage[percent]{overpic}
\usepackage{rotating}
\usepackage{multicol}
\usepackage{booktabs}
\usepackage[left=2cm, right=2cm, top=3cm, bottom=2cm]{geometry}

\renewcommand{\arraystretch}{1.25} 
\usepackage{tikz}
\usepackage{tikz-cd}
\usetikzlibrary{shapes, arrows, positioning, fit, backgrounds, matrix, arrows.meta, bending}
\usepackage{pgfplots}
\pgfplotsset{compat=1.18}
\usepackage{comment}
\usepackage{orcidlink}
\usepackage{cite}
\usepackage{hyperref} 
\newcommand{\R}{\mathbb{R}}

\newcommand{\Z}{\mathbb{Z}}
\newcommand{\N}{\mathbb{N}}
\newcommand{\es}{\mathbb{S}}
\DeclareMathOperator{\Ker}{Ker}
\DeclareMathOperator{\Ima}{Im}
\newcommand{\restr}[2]{{%
  \left.\kern-\nulldelimiterspace 
  #1 
  \vphantom{\big|} 
  \right|_{#2} 
}}
\theoremstyle{plain}
\newtheorem{theorem}{Theorem}[section]
\newtheorem{lemma}[theorem]{Lemma}
\newtheorem{proposition}[theorem]{Proposition}
\newtheorem{corollary}[theorem]{Corollary}

\theoremstyle{definition}
\newtheorem{definition}[theorem]{Definition}
\newtheorem{example}[theorem]{Example}

\theoremstyle{remark}
\newtheorem{remark}[theorem]{Remark}

\def\pe#1{{\bar #1}}

\def\0jp{\pe{0}}

\begin{document}

\title[Singular Morse-Smale Flows on Pseudomanifolds]{Singular Morse-Smale Flows on Pseudomanifolds with Spherical-Cone Singularities: Conley Theory and Intersection Homology}

\author{Jean-Paul Brasselet}
\address{The Institut de Mathématiques de Marseille, Aix-Marseille Université, Marseille, France}
\email{jean-paul.brasselet@univ-amu.fr}
\thanks{Jean-Paul Brasselet was supported by FAPESP under grant No. 2025/06178-3, CNRS (I2M, Marseille) and  Aix-Marseille University.}

\author{Dahisy Lima}
\address{Center for Mathematics, Computing and Cognition, Federal University of ABC, Santo André, São Paulo, Brazil}
\email{dahisy.lima@ufabc.edu.br}
\thanks{Dahisy  Lima was supported by the São Paulo Research Foundation (FAPESP) under grants 2025/12435-9, 2024/00923-6 and 2022/16455-6.}

\author{Denilson Tenório}
\address{Center for Mathematics, Computing and Cognition, Federal University of ABC, Santo André, São Paulo, Brazil}
\email{denilson.tenorio@ufabc.edu.br}
\thanks{Denilson Tenório was supported by the Coordenação de Aperfeiçoamento de Pessoal de Nível Superior - Brasil (CAPES) - Finance Code 001. The author acknowledges the support and hospitality of Aix-Marseille Universit\'e and the Institut de Math\'ematiques de Marseille (I2M), where part of this research was conducted.}


\date{June 24, 2026} 

\begin{abstract}
Classical Morse-Conley theory provides powerful tools for relating dynamical and topological invariants of smooth manifolds. In this paper, we extend this perspective to pseudomanifolds with spherical-cone singularities. By introducing and investigating singular Morse–Smale flows on pseudomanifolds with isolated singularities whose links are homeomorphic to finite disjoint unions of spheres. We establish formulas for the Conley indices of spherical-cone singularities in terms of their local dynamics,  prove the existence of global Lyapunov functions, and investigate the structure of the associated Lyapunov graphs. These results yield alternative formulas for the Euler-Poincar\'{e} characteristic expressed in terms of Conley-theoretic invariants. To relate the singular and smooth settings, we introduce a global morsification procedure that associates a smooth manifold $\widetilde{X}$ to a singular pseudomanifold $X$. This construction allows us to compare the topology of $X$ and $\widetilde{X}$ and, in particular, to derive formulas relating their Euler–Poincaré characteristics. Finally, we study the intersection homology of pseudomanifolds with spherical-cone singularities. We establish connections between intersection homology, singular homology, and the Morse homology of the morsification, thereby providing a dynamical approach to the computation of intersection homology.
\end{abstract}

\maketitle

\tableofcontents

\section{Introduction}
\label{section_introduction}

The study of Morse-Smale flows is one of the classical topics in dynamical systems, originating in the pioneering work of Smale~\cite{smale1960morse}. 
These flows are characterized by a particularly simple form of recurrent dynamics: their chain recurrent sets consist of finitely many hyperbolic singularities and periodic orbits, and the stable and unstable manifolds of these invariant sets intersect transversely. When the phase space has nonempty boundary, one also requires the flow to be transverse to the boundary.

A fundamental result due to Peixoto~\cite{PEIXOTO1962101} states that the class of  Morse-Smale flows on compact surfaces is an open and dense subset of the space of $C^1$-vector fields, thereby characterizing structurally stable dynamics in dimension two. Although this density property fails in higher dimensions~\cite{smale1966structurally}, Morse-Smale flows remain structurally stable in arbitrary dimension~\cite{palis1970structural}. Moreover, the existence of Lyapunov functions for such systems was established in~\cite{meyer1968energy}.

The subclass of Morse-Smale flows without periodic orbits is often referred to in the literature as a \textit{Morse flow} or a \textit{gradient-like flow}. A classical theorem of Smale~\cite{Smale1961} shows that every such flow arises as the gradient flow of a Morse function with respect to a suitable Riemannian metric. Further characterizations and results concerning gradient-like flows can be found in~\cite{bertolim2006isolating, BERTOLIM2006minimal, cruz1999, ketty1993lyapunov, fleitas1975classification}. Despite its classical origins, the theory of Morse-Smale flows remains an active area of research, with recent developments across different contexts such as ~\cite{bernstein2026existence, bertolim2024minimal, breen2021morse, carvalho2015hyperbolicity, dang2019spectral, lima2024cancellations,  pochinka2022nonsingular}.

Besides their importance in dynamical systems, Morse-Smale flows are deeply connected with topology through Morse theory and Conley index theory. One of the fundamental principles underlying these theories is that dynamical information can be used to recover topological invariants of the underlying space. Classical examples include Morse inequalities, Morse homology, the Poincaré-Hopf theorem, and Conley-theoretic formulas for the Euler-Poincaré characteristic. These results reveal a rich interplay between dynamics and topology in the smooth setting.

Much less is known when the phase space possesses singularities. Singular spaces arise naturally in geometry and topology, and several classical tools from smooth dynamics are no longer directly available. In particular, notions such as Morse index, stable and unstable manifolds, and gradient flows require suitable adaptations. While some progress has been achieved in dimension two \cite{lima2024cancellations, Lima2025, LimaTenorio_2, Montufar}, a Conley-theoretic framework for higher-dimensional singular spaces has remained largely unexplored.

The purpose of this paper is to develop such a framework for a class of pseudomanifolds with isolated spherical-cone singularities. More precisely, we consider $n$-pseudomanifolds $X$ with isolated singularities whose links are finite disjoint unions of $(n-1)$-dimensional spheres and introduce the notion of a \textit{singular Morse-Smale flow} on these spaces. This setting generalizes the singular surfaces studied in~\cite{Lima2021, Lima2025} and, to the best of our knowledge, provides the first extension of Conley theory to higher-dimensional singular spaces of this type.

Throughout this paper, cone singularities whose links are homeomorphic to finite disjoint unions of spheres will be called \textit{spherical-cone singularities}.

The local description of the flows around a spherical-cone singularity can be done in terms of the local sheets whose dynamics correspond to a nondegenerate critical point of Morse index $\lambda$. Using this description, we derive explicit formulas for the homotopy and numerical Conley indices of spherical-cone singularities. In particular, we show that the homotopy Conley index is always a wedge sum of spheres whose dimensions and multiplicities are determined by the local dynamics.

We then investigate global aspects of the dynamics. We prove the existence of Lyapunov functions for singular Morse-Smale flows and study the associated Lyapunov graphs. In particular, we are able to locally characterize the indegree and outdegree of a vertex  associated with an arbitrary spherical-cone singularity. 

A central tool introduced in this work, to study the dynamics globally, is a global morsification procedure of a pseudomanifold with spherical-cone singularities, generalizing the concept of morsification introduced in~\cite{Lima2021}. This construction associates to a pseudomanifold $X$ equipped with a singular Morse-Smale flow a smooth manifold $\widetilde X$ endowed with a Morse-Smale flow whose dynamics reflect those of the original system. 

The morsification allows us to compare the topology and dynamics of the singular and smooth settings and leads to several applications.

Among these applications, we obtain dynamical formulas for the Euler-Poincaré characteristic of a pseudomanifold expressed in terms of Conley-theoretic invariants and local data associated with the spherical-cone singularities.

These formulas may be viewed as singular counterparts of classical results relating topology and dynamics on smooth manifolds.
These formulas recover the classical theorem of Smale ~\cite{smale1960morse} in the smooth case.

We also establish explicit relationships between the Euler-Poincaré characteristics of a pseudomanifold and its morsification.

Finally, we investigate the relationships between intersection homology, singular homology, and Morse homology. We first obtain a topological description of the intersection homology groups of these spaces in terms of singular homology; this is purely topological and independent of the underlying dynamics. Moreover, we also highlight that the construction of the morsification yields a map $\mathfrak{p}\colon \widetilde{X}\to X$, which acts as a normalization for the pseudomanifold $X$. As a consequence, the dynamics of a singular Morse-Smale flow can be used to compute intersection homology, providing a new connection between Morse-Conley theory and singular topology. Recent developments linking intersection homology with dynamics are discussed in~\cite{LimaTenorio_2, Ludwig2017}. 

The paper is organized as follows. In Section~\ref{section_background}, we review the necessary background on Morse-Smale flows on smooth manifolds, basic notions of  Conley theory,  and intersection homology. Section~\ref{section_conley} develops the Conley-theoretic framework for singular Morse-Smale flows $\varphi$ on $n$-pseudomanifolds $X$ with spherical-cone singularities. In Subsection \ref{subsection_definitions}, we formally introduce these notions and establish formulas for the Conley index of these singularities. In  Subsection \ref{subsection_lyapunov_function}, we prove the existence of Lyapunov functions in this setting. In Subsection \ref{subsection_morsification}, we establish the existence of a global morsification $(\widetilde{X},\widetilde{\varphi})$  for each pair $(X,\varphi)$. In Subsection \ref{subsection_graph}, we compare the data associated with the Lyapunov functions   and Lyapunov graphs associated to  $X$ and $\widetilde{X}$.
Section \ref{section_characteristic}, we focus on the Euler-Poincaré characteristic of pseudomanifolds and compare the topology of $X$ and $\widetilde{X}$. In particular, we derive formulas relating their Euler–Poincaré characteristics. 
Finally, Section~\ref{section_IH} is devoted to intersection homology. We establish connections between intersection homology, singular homology, and the Morse homology of the morsification, thereby providing a dynamical approach to the computation of intersection homology.

\section{Background}
\label{section_background}

This paper lies at the intersection of two a priori unrelated theories: Conley theory and intersection homology. To make our results accessible to specialists in either area, we find it beneficial to review the fundamental framework of both theories.

\subsection{Morse-Smale Flows on Smooth Manifolds}
\label{subsection_MS_background}

Morse-Smale flows constitute a fundamental class of dynamical systems, serving as an ideal model to formulate hypotheses regarding continuous flows on pseudomanifolds with spherical-cone singularities see Subsection~\ref{subsection_definitions}, in particular Definition~\ref{definicao_fluxo}. For completeness, we briefly review the basic definitions of flows and Morse-Smale flows below. We refer the reader to~\cite{robinson1998dynamical, akin1993general, palis1982geometric} for general definitions and main results, and to~\cite{Banyaga2004} for foundational results in classical Morse theory.

Let $M$ be a compact smooth $n$-manifold, and let $\varphi$ be a \textit{continuous flow} on $M$, that is, a continuous map
\[
    \varphi\colon \mathbb{R}\times M \to M,
\]
satisfying
\begin{enumerate}[label=(\roman*), leftmargin=*, font=\itshape, align=left]
    \item $\varphi(0,x)=x$ for all $x\in M$;
    \item $\varphi(s+t, x)=\varphi(s, \varphi(t,x))$ for all $s, t \in \mathbb{R}$ and $x \in M$.
\end{enumerate}

Given $x \in M$, the \textit{$\omega$-limit set} of $x$ is defined by
\[
    \omega(x) \coloneqq \left\{ y \in M \;\middle|\; \exists (t_i)_{i\in\N} \subset \mathbb{R} \text{ such that } t_i \to +\infty \text{ and } \varphi(t_i, x) \to y \right\}.
\]
Similarly, the \textit{$\alpha$-limit set} of $x\in M$ is defined as
\[
    \alpha(x) \coloneqq  \left\{ y \in M \;\middle|\; \exists (t_i)_{i\in\N} \subset \mathbb{R} \text{ such that } t_i \to -\infty \text{ and } \varphi(t_i, x) \to y \right\}.
\]
The \textit{orbit} of a point $x \in M$ with respect to  $\varphi$ is  the set
\(
    \mathcal{O}(x) \coloneqq \{ \varphi(t,x) \mid t \in \mathbb{R} \}.
\)
Orbits of a continuous flow can be classified into three types:
\begin{enumerate}[label=(\arabic*), leftmargin=*, align=left]
    \item \textit{Singular orbit}: an orbit consisting of a single point, that is,  $\varphi(t,x)=x$ for all $t\in \mathbb{R}$. In this case, $\mathcal{O}(x)=\{x\}$, and $x$ is called a \textit{singularity} of $\varphi$. The set of all singularities of $\varphi$ is denoted by $\operatorname{Sing}(\varphi)$.
    
    \item \textit{Periodic  (or closed) orbit}: a nonsingular orbit for which there exists $T > 0$ such that
    \( 
        \varphi(t+T,x) = \varphi(t,x) \) for all  \(t \in \mathbb{R}.
    \)
    The smallest such $T>0$ is called the \textit{period} of the orbit. The set of all periodic orbits is denoted by $\operatorname{Per}(\varphi)$.

    \item \textit{Regular orbit}: an orbit that is neither a singular orbit nor a periodic orbit.
\end{enumerate}

A point $x\in M$ is called a \textit{wandering point} of $\varphi$ if there exist a neighborhood $U$ of $x$ and $t_0>0$ such that
\[
    \varphi(t,U)\cap U = \varnothing , \qquad \text{for all} \ |t|>t_0.
\]
Otherwise,  $x$ is called a \textit{nonwandering point}. Denote by $\Omega(\varphi)$ the set of all nonwandering points of $\varphi$, called the \textit{nonwandering set}. Clearly
\[
    \operatorname{Sing}(\varphi)\cup \operatorname{Per}(\varphi)\subset \Omega(\varphi).
\]

Let $d$ be a metric on $M$. Given $x, y\in M$ and $\epsilon>0$, an \textit{$\epsilon$-chain} from $x$ to $y$ is a finite sequence of points $x_1=x, x_2, \ldots, x_{m_0}=y$ in $M$ together with times $t_1, t_2, \ldots, t_{m_0-1} \ge 1$ such that
\[
    d( \varphi(t_i,x_{i}), x_{i+1} )<\epsilon, \quad \text{for all}\  1\leq i < m_0.
\]
 A point $x\in M$ is called \textit{chain recurrent} if for every $\epsilon > 0$ there exists an $\epsilon$-chain from $x$ to itself. The set of all chain recurrent points, denoted by $\mathcal{R}(\varphi)$, is called the \textit{chain recurrent set}. It follows immediately from the definitions that
\[
    \operatorname{Sing}(\varphi)\cup \operatorname{Per}(\varphi)\subset \Omega(\varphi) \subset \mathcal{R}(\varphi).
\]

The elements of $\operatorname{Sing}(\varphi) \cup \operatorname{Per}(\varphi)$ are called \textit{critical elements}. Given a  critical element $\sigma$,  the \textit{stable manifold} of $\sigma$ is defined as
\[
    W^s(\sigma) \coloneqq \{ x \in M \mid \omega(x) = \sigma \},
\]
and the \textit{unstable manifold} of $\sigma$  is defined as
\[
    W^u(\sigma) \coloneqq \{ x \in M \mid \alpha(x) = \sigma \}.
\]

We are ready to introduce the notion of a Morse–Smale flow.

\begin{definition}
\label{def_Morse_Smale_suave}
A continuous flow $\varphi$ on $M$ is said to be a \textit{Morse-Smale flow} if it satisfies the following conditions:
\begin{enumerate}[label=(\roman*), leftmargin=*, font=\itshape, align=left]
    \item  $\Omega(\varphi)=\operatorname{Sing}(\varphi)\cup \operatorname{Per}(\varphi)$ and it consists of finitely many hyperbolic\footnote{See~\cite{palis1982geometric} for the definition of hyperbolicity.} critical elements;
    \item   $W^s(\sigma_1)$ and $W^u(\sigma_2)$ intersect transversely, for any two distinct critical elements $\sigma_1$ and $\sigma_2$;
    \item $\varphi$ is transverse to the boundary $\partial M$.
\end{enumerate}    
\end{definition}

By condition (i) above and the results of~\cite{Smale1961}, hyperbolic singularities are locally conjugate to nondegenerate critical points of a smooth function.  Also, if $p$ is a singularity of Morse index $\lambda$, then its stable and unstable manifolds are smooth submanifolds of $M$ satisfying
\[
    \dim(W^u(p)) = \lambda \quad \text{and} \quad \dim(W^s(p)) = n - \lambda, 
\] where $n$ is the dimension of $M$.
Moreover, if $\gamma$ is a periodic orbit of Morse index $\lambda$, its stable and unstable manifolds are also submanifolds of $M$, with dimensions given by
\[
    \dim(W^u(\gamma)) = \lambda + 1 \quad \text{and} \quad \dim(W^s(\gamma)) = n - \lambda. 
\]

A fundamental consequence for Morse–Smale flow is that $\mathcal{R}(\varphi) = \Omega(\varphi) =\operatorname{Sing}(\varphi)\cup \operatorname{Per}(\varphi)$. The first equality  follows from the transversality condition in Definition~\ref{def_Morse_Smale_suave}, which implies that Morse-Smale flows admit no homoclinic cycles\footnote{Given $p\in \operatorname{Sing}(\varphi)$ and $x\in M $  with $x \neq p$, the orbit $\mathcal{O}(x)$ is called \textit{homoclinic} to $p$ if $\varphi (t,x) \to  p$ as  $t\to \pm\infty$. In this case, the set $ \mathcal{O}(x) \cup \{p\}$ is called a \textit{homoclinic cycle}.}  nor heteroclinic cycles\footnote{A \textit{heteroclinic cycle} consists of a finite sequence of singularities $p_1, p_2, \ldots, p_k \in \operatorname{Sing}(\varphi)$ with $k \geq 2$, where $p_1, \ldots, p_{k-1}$ are distinct and $p_k = p_1$, together with points $x_i\in X$ such that $\alpha(x_i) = p_i$ and $\omega(x_i) = p_{i+1}$ for all $1 \leq i \leq k-1$. See \cite{palis1982geometric} for more details.}. 

We now introduce a natural relation on the set of critical elements which is in fact a partial order on the set $\mathcal{R}(\varphi)$. 

\begin{definition}
\label{definition_ordem_singularidades}
For critical elements $\sigma_1, \sigma_2\in \mathcal{R}(\varphi)$, define $\sigma_1\leq \sigma_2$ if $W^u(\sigma_1 )\cap W^s(\sigma_2 ) \neq \varnothing$.
\end{definition}

Morse-Smale flows are structurally stable. Classical results of Palis~\cite{palis1970structural} show that any sufficiently small $C^1$-perturbation of the vector field generating $\varphi$ yields a new flow that is topologically equivalent to $\varphi$; that is, there exists a homeomorphism mapping the orbits of $\varphi$ onto the orbits of the perturbed flow, preserving their orientation. 

Moreover, Morse-Smale flows exhibit gradient-like behavior outside their recurrent set. Results of Smale~\cite{smale1960morse} and Meyer~\cite{meyer1968energy} guarantee the existence of a global strict Lyapunov (or energy) function, which is smooth, constant on each critical element, and strictly decreasing along regular orbits.

Finally, the dynamical  structure associated with a Morse-Smale flow imposes strong constraints on the topology of the underlying manifold $M$. Let $c_k$ denote the number of critical elements of Morse index $k$, and let $\beta_k$ denote the $k$-th Betti number of $M$. As shown in~\cite{smale1960morse}, the strong Morse inequalities hold
\[
    \sum_{i=0}^{k} (-1)^{k-i} c_i \geq \sum_{i=0}^{k} (-1)^{k-i} \beta_i  \quad \text{for all } 0 \leq k \leq n = \dim(M),
\]
with equality for $k=n$, yielding the Euler–Poincaré formula.

\subsection{Conley Index}
\label{subsection_conley_theory_background}

Originally introduced as a generalization of the Morse index, the Conley index has the notable advantage of being well-defined for any isolated invariant set associated to a continuous dynamical system. This generality makes it particularly well suited for the study of dynamical behaviors in topological spaces that may exhibit singularities or lack smooth structure.

In what follows, we introduce only the notation and definitions needed for our purposes; for a comprehensive treatment, we refer the reader to~\cite{Conley} and~\cite{livro_coloquio}.

Let $X$ be a Hausdorff topological space, and let $\varphi$ be a continuous flow on $X$. A compact subset $N\subset X$ is an \textit{isolating neighborhood} if
\[
    \operatorname{Inv}(N) \coloneqq \{x \in N \mid \varphi(t,x) \in N \text{ for all } t \in \mathbb{R}\} \subset \operatorname{int}(N),
\]
where $\operatorname{int}(N)$ denotes the interior of $N$. A subset $S \subset X$ is called an \textit{isolated invariant set} with respect to $\varphi$ if $S=\operatorname{Inv}(N)$ for some isolating neighborhood $N$.

Let $S\subset X$ be an isolated invariant set. A pair of compact sets $(L_1,L_2)$ is an \textit{index pair} of $S$ if:
\begin{enumerate}[label=(\roman*), leftmargin=*, font=\itshape, align=left]
    \item $\overline{L_1 \setminus L_2}$ is an isolating neighborhood for $S$ in $X$ and $S\subset \operatorname{int} (L_1\setminus L_2)$; 
    \item $L_2$ is \textit{positively invariant} in $L_1$, that is, if $x\in L_2$ and $\varphi([0,T],x)\subset L_1$ then $\varphi([0,T],x)\subset L_2$; 
    \item $L_2$ is the \textit{exit set} of the flow in $L_1$, that is, if $x\in L_1$ and $\varphi([0,\infty], x)\not\subset L_1$ then there exists $T> 0$ such that $\varphi([0,T],x)\subset L_1$ and $\varphi(T,x)\in L_2$. 
\end{enumerate}

Figure~\ref{figura_conley_index_pair} provides a schematic illustration of the notion of index pairs. In all three examples, the compact subsets $L_2 \subset L_1$ and the invariant set $S$ are fixed, while the underlying dynamics vary. From left to right: in the first, $L_2$ is not positively invariant; in the second, $L_2$ fails to be an exit set; and in the third, the pair $(L_1, L_2)$ constitutes a valid index pair for $S$.

\begin{figure}[!ht]
\centering
\begin{overpic}[width=9cm]{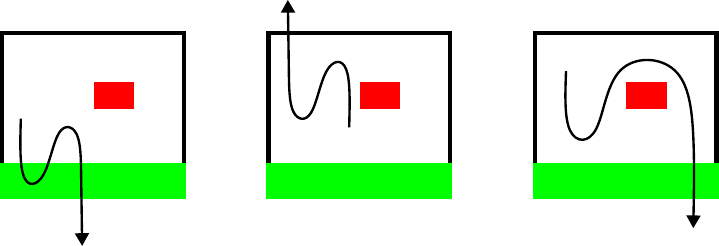}
\put(0,32){$L_1$}
\put(34,32){$L_1$}
\put(74, 32){$L_1$}
\put(0,2){$L_2$}
\put(35, 2){$L_2$}
\put(74, 2){$L_2$}
\put(15,15){$S$}
\put(52, 15){$S$}
\put(89, 15){$S$}
\end{overpic}
\caption{Examples of configurations allowed and not allowed for index pairs.}
\label{figura_conley_index_pair}
\end{figure}
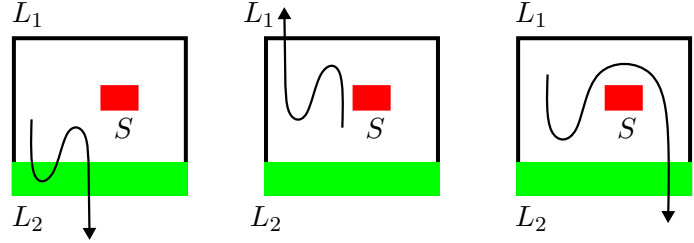

Figure~\ref{figura_plano_index_pair} depicts index pairs $(L_1, L_2)$ 
for nondegenerate critical points in the plane. In each case, $L_1$ is a 
compact neighborhood of the respective critical point. For the critical 
point $p_1$, we have $L_2 = \varnothing$; for $p_2$, the set $L_2$ consists of two 
connected components (shown in green); and  for $p_3$, $L_2$ is  the boundary of $L_1$.

\begin{figure}[!ht]
\centering
\vspace{0.5cm}
\begin{overpic}[width=10cm]{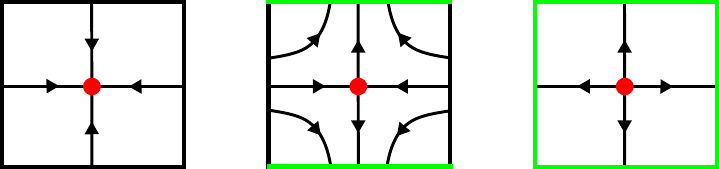}
\put(7, 25){$L_2=\varnothing$}
\put(48, 25){$L_2$}
\put(80, 25){$L_2 = \partial L_1$}
\put(8.5,8.5){$p_1$}
\put(45.5, 8.5){$p_2$}
\put(82.5, 8.5){$p_3$}
\end{overpic}
\caption{Index pairs for nondegenerate critical points in the plane: $p_1$ is an attracting, $p_2$ a saddle, and $p_3$ a repelling singularities.}
\label{figura_plano_index_pair}
\end{figure}
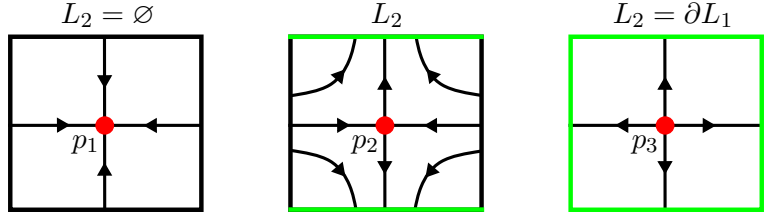

Before introducing the concept of Conley index, we establish some necessary terminology. A \textit{pointed space} $(X, x)$ consists of a topological space $X$ equipped with a distinguished basepoint $x \in X$. For a given pointed space $(X, x)$, we denote by $[X, x]$ its \textit{homotopy type}; that is, the equivalence class of all pointed spaces that are homotopy equivalent\footnote{Homotopy equivalence is taken in the sense of basepoint-preserving maps.} to $(X, x)$.

In this setting, the \textit{homotopy Conley index} of an isolated invariant set $S\subset X$ is defined as
\[
    h(S) \coloneqq [L_1 / L_2, [L_2]],
\]
which is the homotopy type of the pointed space $(L_1 / L_2, [L_2])$, where $(L_1, L_2)$ is an index pair for $S$, and $[L_2]$ is the equivalence class of $L_2$ in the quotient space $L_1 / L_2$. If $L_2=\varnothing$, we adopt the convention of replacing the pair $(L_1,\varnothing)$ by the pointed pair $(L_1\sqcup{\ast},{\ast})$, where $\ast$ denotes an arbitrary point.

The \textit{homology Conley index} of $S\subset X$ is defined as
\[
    CH_{\ast}(S) \coloneqq \widetilde{H}_{\ast}(h(S)),
\]
where $\widetilde{H}_{\ast}$ denotes the reduced singular homology over $\mathbb{Z}$. Finally, the \textit{numerical Conley index} is given by
\[
    h_{\ast}(S) \coloneqq \operatorname{rank}(CH_{\ast}(S)).
\]

As established in~\cite{Conley}, every isolated invariant set admits an index pair. Moreover, the homotopy type of the pointed space $(L_1/L_2, [L_2])$ is independent of the choice of index pair $(L_1, L_2)$ for $S$. Consequently, the Conley index is well-defined.

The next example shows the relation between Conley index and Morse index when both are defined. 

\begin{example}
\label{example_conley_smooth}
Let $\varphi$ be a Morse-Smale flow on a smooth manifold $M$. The singularities and periodic orbits comprising the chain recurrent set $\mathcal{R}(\varphi)$ are isolated invariant sets with respect to $\varphi$, and their homotopy Conley indices are computed as follows:
\begin{enumerate}[label=(\alph*), font=\itshape, leftmargin=*, align=left]
    \item If $x$ is a nondegenerate critical point of Morse index $\lambda$, its homotopy Conley index is $h(x) = \es^{\lambda}$. See \cite[Example 1.22]{livro_coloquio}.
    \item If $\gamma$ is an orientable hyperbolic periodic orbit of Morse index $\lambda$, its homotopy Conley index is $h(\gamma)=\es^{\lambda}\vee \es^{\lambda+1}$. See \cite[Proposition 1.12]{livro_coloquio}.
\end{enumerate}    
\end{example}

\subsection{Intersection Homology}
\label{subsection_IH_background}

In this subsection, we recall the definition of intersection homology, a theory originally introduced by Goresky and MacPherson in their seminal work~\cite{goresky1980}. The primary motivation for this theory was to restore duality for singular varieties, a setting where classical results such as Poincaré and Poincaré-Lefschetz duality fail. For comprehensive introductions to the subject, we refer the reader to~\cite{kirwan2006, maxim2019, Brasselet2021}, alongside the foundational paper~\cite{goresky1980}. We begin by reviewing the necessary preliminary concepts.

A topological  space $X$ is an \textit{$n$-pseudomanifold} if is nonempty, paracompact, Hausdorff topological space and there exists a closed subspace $\Sigma \subset X$ satisfying the following conditions:
\begin{enumerate}[label=(\roman*), font=\itshape, leftmargin=*, align=left]
    \item $X\setminus \Sigma$ is an $n$-dimensional manifold that is dense in $X$;
    \item  $\dim(\Sigma)\leq n - 2$.
\end{enumerate}
The subspace $\Sigma$ associated to a  pseudomanifold $X$ contains the set of possible \textit{singularities}; that is, points that do not admit an open neighborhood in $X$ homeomorphic to an open subset of a Euclidean space (of any finite dimension), or whose boundary is homeomorphic to a sphere.

A topological space $X$ is said to be endowed with a \textit{piecewise linear structure}, or simply a \textit{PL structure}, if there exists a class of locally finite simplicial triangulations of $X$, called \textit{admissible triangulations}, satisfying the following conditions:
\begin{enumerate}[label=(\roman*), font=\itshape, leftmargin=*, align=left]
    \item Every subdivision of an admissible triangulation is admissible.
    \item Any two admissible triangulations admit a common subdivision.
\end{enumerate}
A space equipped with such a structure is called a \textit{PL space}.

A triangulation of a PL space $X$ is a triangulation belonging to its corresponding admissible class; that is, a simplicial complex $K$ whose geometric realization, denoted by $\lvert K \rvert$, is homeomorphic to $X$. In this case, the space $X$ is said to be \textit{triangulated}, and we write $X = \lvert K \rvert$.

A \textit{PL pseudomanifold} of dimension $n$ is a PL space $X$ containing a closed PL subspace $\Sigma$ of codimension at least $2$ such that $X \setminus \Sigma$ is an $n$-dimensional PL manifold that is dense in $X$. Equivalently, given a triangulation $X = \lvert K \rvert$, the space $X$ is the union of $n$-simplices, and every $(n-1)$-simplex is a face of exactly two $n$-simplices.

A PL pseudomanifold $X$ is \textit{orientable} if it admits a compatible orientation for all its $n$-simplices. Once such an orientation is chosen, $X$ is said to be \textit{oriented}.

\begin{example}
The pinched torus and the suspension of the torus are examples of orientable PL-pseudomanifolds. 
\end{example}

A \textit{topological stratification} $\mathcal{S}$ of an $n$-pseudomanifold $X$ is a filtration
\[
    X=X_n\supset X_{n-1}=X_{n-2}\supset \cdots \supset X_1\supset X_0 \supset X_{-1}=\varnothing ,
\]
by closed subspaces satisfying the following conditions:
\begin{enumerate}[label=(\roman*), font=\itshape, leftmargin=*, align=left]
    \item The \textit{stratum} $S_i \coloneqq X_i\setminus X_{i-1}$ is either empty or a finite union of $i$-dimensional smooth submanifolds of $X$.
    
    \item Each point $x\in S_i$ admits a distinguished neighborhood $U_x\subset X$ together with a homeomorphism
    \[
        \phi_x\colon U_x \to \mathbb{B}^i\times c(L),
    \]
    satisfying the \textit{local triviality property}. Here, the components are defined as follows:
    \begin{itemize}
        \item $\mathbb{B}^i$ is an open ball in $\mathbb{R}^i$.
        \item The \textit{link} $L$ of the stratum $S_i$ is a compact $(n-i-1)$-dimensional pseudomanifold, independent (up to homeomorphism) of the choice of $x\in S_i$, and filtered by
        \[
            L=L_{n-i-1}\supset L_{n-i-3} \supset \cdots \supset L_0 \supset L_{-1}=\varnothing.
        \]
        \item The space $c(L)$ is the \textit{open cone} over $L$, defined as
        \[
            c(L)\coloneqq \left( L \times [0,1) \right) / {\sim},
        \]
        where $(x,0) \sim (x',0)$ for all $x, x' \in L$. It is filtered by setting $(c(L))_0=\{ \text{vertex} \}$ and $(c(L))_k=c(L_{k-1})$ for $k> 0$. By convention, $c(\varnothing)$ is a single point.
    \end{itemize}
    Furthermore, the homeomorphism $\phi_x$ preserves the stratifications; that is, it induces restriction homeomorphisms:
    \[
        \restr{\phi_x}{U_x \cap X_j} \colon U_x \cap X_j \to \mathbb{B}^i\times c(L_{j-i-1}), \quad \text{for all } j\geq i.
    \]
\end{enumerate}

Stratifications satisfying the Whitney conditions are known to admit local topological triviality, see~\cite{borel1984intersection}.

A \textit{PL stratification} $\mathcal{S}$ of an $n$-dimensional PL pseudomanifold $X$ is a stratification in which all involved subspaces are PL subspaces, and the local triviality property holds in the PL category.

Given a triangulation $X = \lvert K \rvert$ of an oriented and compact PL space $X$, the chain complex of simplicial chains of $K$ with coefficients in $\mathbb{Z}$ is denoted by $C_{\ast}(K; \mathbb{Z})$. A chain $\xi \in C_i(K; \mathbb{Z})$ is written as $\xi = \sum \xi_\sigma \sigma$, where the sum runs over all the oriented $i$-simplices $\sigma$ of $K$ and $\xi_\sigma \in \mathbb{Z}$.

Every chain $\xi \in C_i(K; \mathbb{Z})$ has a canonical image in $C_i(K'; \mathbb{Z})$ for any subdivision $K'$ of $K$. Two chains in $C_i(K_1; \mathbb{Z})$ and $C_i(K_2; \mathbb{Z})$ are identified if their images in a common subdivision coincide. The group $C_i(X; \mathbb{Z})$ of PL geometric chains with closed supports on $X$ is the direct limit, under refinement, of the groups $C_i(K; \mathbb{Z})$ over all triangulations of $X$.

The support of $\xi \in C_i(K; \mathbb{Z})$, denoted by $|\xi|$, is the union of the closed simplices for which $\xi_\sigma \neq 0$. This support is independent of the chosen subdivision; thus, the support of a chain $\xi \in C_i(X; \mathbb{Z})$ is well-defined.

Using the standard boundary map, the chain complex $C_{\ast}(X; \mathbb{Z})$ is well-defined, and its homology, denoted by $H_{\ast}(X; \mathbb{Z})$, is called the homology with closed supports, or the Borel-Moore homology of $X$. It is a well-known result that, for compact spaces, this homology is isomorphic to singular homology. Throughout this paper, we denote the singular homology of $X$ simply by $H_{\ast}(X; \mathbb{Z})$. 

Let $X$ be a PL stratified pseudomanifold. If an $i$-chain $\xi$ is transverse to the stratum $X_{n-\alpha}$ of a PL filtration, then transversality implies the equality
\[
    \dim( |\xi| \cap X_{n-\alpha} ) = i - \alpha. 
\]

The allowable chains and cycles are those that intersect each stratum $X_{n-\alpha}$ with a controlled transversality defect, which is bounded by a perversity function.

A \textit{perversity}, also referred to as a \textit{GM perversity} in honor of Goresky and MacPherson, is an integer-valued function
\[
    \overline{p}\colon  [0, \dim (X)]\cap \mathbb{Z} \to \mathbb{N}
\]
satisfying the following conditions:
\begin{enumerate}[label=(\roman*), font=\itshape, leftmargin=*, align=left]
    \item $\overline{p}(0)=\overline{p}(1)=\overline{p}(2)=0$. 
    \item $\overline{p}(\alpha)\leq \overline{p}(\alpha + 1)\leq \overline{p}(\alpha)+1$ for all $\alpha \geq 2$.
\end{enumerate}    

We sometimes denote $\overline{p}(\alpha)$ by $p_{\alpha}$. These conditions correspond to the original definition given by Goresky and MacPherson to ensure the fundamental properties of the theory. Over the years, however, more general perversities have been investigated by various authors, leading to new results and further developments in the theory. In this paper, we restrict our attention to GM perversities.

Standard examples of GM perversities include the following:
\begin{enumerate}[label=(\arabic*), leftmargin=*, align=left]
    \item The zero perversity $\overline{0}=(0, \ldots, 0)$.
    \item The top (or maximum) perversity $\overline{t}=(0,0,0,1,2,\ldots, n-2)$; that is, $t_{\alpha}=\alpha-2$ for all $\alpha \geq 2$.
    \item The upper middle perversity $\overline{n} (\alpha) = \left\lceil \frac{\alpha -2}{2} \right\rceil$ for all $\alpha \geq 2$.
    \item The lower middle perversity $\overline{m} (\alpha) = \left\lfloor \frac{\alpha -2}{2} \right\rfloor$ for all $\alpha \geq 2$.
\end{enumerate}

An $i$-chain $\xi\in C_i(X; \mathbb{Z})$ is said to be \textit{$\overline{p}$-allowed} if
\[
    \dim \bigl( |\xi| \cap X_{n-\alpha} \bigr) \leq i - \alpha + p_{\alpha}, \quad \text{for all } \alpha\geq 2.
\]

We denote by $IC^{\overline{p}}_i(X; \mathbb{Z})$ the subgroup of $C_i(X; \mathbb{Z})$ consisting of chains $\xi$ such that both $\xi$ and $\partial \xi$ are $\overline{p}$-allowed; that is, $\xi\in IC^{\overline{p}}_i(X; \mathbb{Z})$ if
\begin{align*}
    \dim \bigl( |\xi| \cap X_{n-\alpha} \bigr) &\leq i - \alpha + p_{\alpha}, \\
    \dim \bigl( |\partial\xi| \cap X_{n-\alpha} \bigr) &\leq (i-1) - \alpha + p_{\alpha},
\end{align*}
for all $\alpha \geq 2$. Using the standard boundary map of chains $\partial_{\ast}$, we obtain the chain complex denoted by $(IC_{\ast}^{\overline{p}}(X; \mathbb{Z}), \partial_{\ast})$.

The \textit{intersection homology groups with perversity $\overline{p}$}, or simply the \textit{intersection homology groups}, denoted by $IH^{\overline{p}}_{\ast}(X; \mathbb{Z})$, are the homology groups of the complex $(IC_{\ast}^{\overline{p}}(X; \mathbb{Z}), \partial_{\ast})$.

Although we have introduced intersection homology with coefficients in $\mathbb{Z}$, this construction naturally extends to any ring $R$. 

Finally, we discuss the notions of normal pseudomanifolds and normalization.
An oriented $n$-dimensional pseudomanifold $X$ is said to be \textit{normal} if:
\[
    H_n(X, X \setminus \{x\}; \mathbb{Z}) = \mathbb{Z},
\]
for all $x \in X$. A \textit{normalization} of an $n$-dimensional pseudomanifold $X$ is a normal pseudomanifold $\widetilde{X}$ together with a finite-to-one projection $\pi\colon \widetilde{X} \to X$ such that, for any $p \in X$, the induced map:
\[
    \pi_{\ast}\colon \bigoplus_{q \in \pi^{-1}(p)} H_n(\widetilde{X}, \widetilde{X}\setminus \{q\}; \mathbb{Z}) \to H_n(X, X\setminus \{p\}; \mathbb{Z})
\]
is an isomorphism.

A classical result established in~\cite{goresky1980} states that intersection homology is invariant under normalization. Specifically, if $X$ is a pseudomanifold with normalization $\pi\colon \widetilde{X} \to X$, then the induced map $\pi_{\ast}$ yields isomorphisms $IH^{\overline{p}}_{\ast}(\widetilde{X}; \mathbb{Z}) \cong IH^{\overline{p}}_{\ast}(X; \mathbb{Z})$ for any perversity $\overline{p}$.

\section{Conley Theory for Singular Morse-Smale Flows on Pseudomanifolds with Spherical-Cone Singularities}
\label{section_conley}

In this section, we formally introduce both the class of pseudomanifolds and the class of continuous flows that constitute the central objects of investigation in this work. We then develop the Conley-theoretic framework adapted to this singular setting and present our first main results. Specifically, we establish a formula for computing the Conley index of spherical-cone singularities, construct a global Lyapunov function compatible with the singular flow, and introduce a suitable morsification procedure that can be used to describe the associated Lyapunov graphs and their relation with the underlying dynamics.

\subsection{Flows on Pseudomanifolds with Spherical-Cone Singularities and their Conley Indices}
\label{subsection_definitions}

We begin this subsection by introducing the singular space (see Definition \ref{definicao_pseudovariedade}) and the continuous flow (see Definition~\ref{definicao_fluxo}) that constitute the framework of our study.

\begin{definition}
\label{definicao_regioes_cone}
Let $k \in \mathbb{N}$ with $k \geq 2$. A \textit{$k$-sheet cone} of dimension $n$, denoted by $\mathcal{C}^n_k$, is defined as the wedge sum of $k$ copies of the $n$-dimensional open unit ball $\mathbb{B}^n$, identified at their centers. More precisely, we define 
\[
    \mathcal{C}^n_k \coloneqq \bigvee_{i=1}^{k} \mathbb{B}_i^n,
\]
where each $\mathbb{B}_i^n$ is a copy of $\mathbb{B}^n = \{x \in \mathbb{R}^n \mid \|x\| < 1\}$, and all basepoints $0 \in \mathbb{B}^n_i$  are identified in the wedge construction.
\end{definition}

Throughout this paper, we fix the dimension $n \in \mathbb{N}$ with $n \geq 2$. When no ambiguity arises, we simply write $\mathcal{C}_k$ instead of $\mathcal{C}^n_k$.

In this work, we adopt the notion of $C^r$ maps as in~\cite{milnor1965topology}: that is, a map $f\colon K\to \mathbb{R}^n$, where $K \subset \mathbb{R}^m$, is of \textit{class $C^r$} for $1\leq r\leq \infty$ if $f$ admits an extension $\hat{f}$ of class $C^r$ to an open neighborhood of $K$ in $\mathbb{R}^m$. In this context, a map $f\colon K_1\to K_2$ between subsets $K_1\subset \mathbb{R}^m$ and $K_2\subset \mathbb{R}^n$ is called a \textit{$C^r$-diffeomorphism} if it is a bijection and both $f$ and $f^{-1}$ are of class $C^r$.

\begin{definition}
\label{definicao_pseudovariedade}
A subset $X \subset \mathbb{R}^\ell$  is said to be  an \textit{$n$-pseudomanifold with spherical-cone singularities} if, for every $p \in X$, there exists an open neighborhood $U_p$ of $p$ in $X$ and a diffeomorphism 
\[
    \psi \colon U_p \to \mathcal{P}
\]
such that $\mathcal{P}$ is either an open subset of $\mathbb{R}^n$ or the space $\mathcal{C}^n_k$ for some $k \geq 2$. In this case, $\psi$ is called a \textit{local chart} around $p$.
\end{definition}

\begin{remark}
It is straightforward to check that any set $X \subset \R^\ell$ satisfying Definition~\ref{definicao_pseudovariedade} is in fact an $n$-pseudomanifold as defined in Subsection~\ref{subsection_IH_background}.
\end{remark}

Let $X$ be a pseudomanifold with spherical-cone singularities. For each $k\geq 2$, let $\operatorname{M}(\mathcal{C}_k)$  be the set of all points $p\in X$  admitting a local chart of the form $\psi\colon U_p\to \mathcal{C}_k$, where $\mathcal{C}_k=\bigvee_{i=1}^{k} \mathbb{B}_i$.
Moreover, for each  $p \in \operatorname{M}(\mathcal{C}_k)$,  we define the \textit{sheets} of $p$ as the subsets $\psi^{-1}(\mathbb{B}_i)\subset X$ for each $i=1,\ldots, k$.

The \textit{singular set} of $X$  is defined as
\[
    \operatorname{Sing}(X) \coloneqq \bigcup_{k\geq 2} \operatorname{M}(\mathcal{C}_k).
\]
The points in $\operatorname{Sing}(X)$ are called \textit{spherical-cone singularities}, or simply \textit{cone singularities}. In this paper, we consider pseudomanifolds whose spherical-cone singularities have finite multiplicity, meaning that there exists some $k_0\in \mathbb{N}$ with $k_0\geq 2$ such that $\operatorname{M}(\mathcal{C}_{k})=\varnothing$ for all $k\geq k_0$.

\begin{remark}
\label{remark_link}
Let $p\in \operatorname{M}(\mathcal{C}_k)$ and consider a local chart $\psi\colon U_p \to \mathcal{C}_k$, where $U_p$ is an open neighborhood of $p$. Observe that the \textit{link} of $p$, denoted by $L_p$, can be regarded as the boundary of $\overline{U}_p$ and is homeomorphic to the disjoint union of $k$ copies of the $(n-1)$-dimensional sphere. We will use this fact extensively throughout the paper. This fact also motivates the nomenclature ``spherical-cone singularities''.
\end{remark}

Given a pseudomanifold with spherical-cone singularities $X$, we say that $X$ is \textit{closed} if it is a compact topological space without boundary, and it is \textit{orientable} if its \textit{regular part} $X \setminus \operatorname{Sing}(X)$ is an orientable smooth $n$-manifold.

\begin{example}
The pinched torus and, more generally, any topological space obtained by a finite sequence of wedge sums of $n$-dimensional manifolds (not necessarily at the same basepoint) are examples of $n$-pseudomanifolds with spherical-cone singularities. On the other hand, the suspension of the torus is not an example, since the links of its singularities are homeomorphic to a torus.
\end{example}

The primary goal of this paper is to recover topological properties of an $n$-dimensional pseudomanifold $X$ with spherical-cone singularities through the study of dynamical systems on $X$. To this end, we consider a continuous flow $\varphi$ on $X$. However, to obtain concrete results, we impose further hypotheses on $\varphi$, which are detailed below.

\begin{definition}
\label{definicao_fluxo}
Let $X$ be a compact $n$-pseudomanifold with spherical-cone singularities. A continuous flow $\varphi$ on $X$ is called a \textit{singular Morse-Smale flow} if it satisfies the following conditions:
\begin{enumerate}[label=(\roman*), font=\itshape, leftmargin=*, align=left]
    \item The chain recurrent set satisfies $\mathcal{R}(\varphi) = \operatorname{Sing}(\varphi)\cup \operatorname{Per}(\varphi)$.
    \item The restriction of $\varphi$ to $X \setminus \left( \bigcup_{p\in \operatorname{Sing}(X)} U_p \right)$ is a Morse-Smale flow, where $U_p$ is an open neighborhood of $p\in \operatorname{Sing}(X)$ as given in Definition~\ref{definicao_pseudovariedade}.
    \item For every $p \in \operatorname{Sing}(X)$ with local chart $\psi\colon U_p\to \bigvee_{i=1}^k \mathbb{B}_i$, the restriction of $\varphi$ to the closure of each sheet $\overline{\psi^{-1}(\mathbb{B}_i)}$ is a Morse-Smale flow for which $p$  is the unique critical element; moreover, $p$ is a hyperbolic singularity of the restricted flow.
    \item The flow $\varphi$ is transverse to the boundary $\partial X$.
\end{enumerate}
\end{definition}

Figure~\ref{figura_exemplo_01} illustrates some examples of closed $2$-pseudomanifolds with spherical-cone singularities endowed with singular Morse-Smale flows.

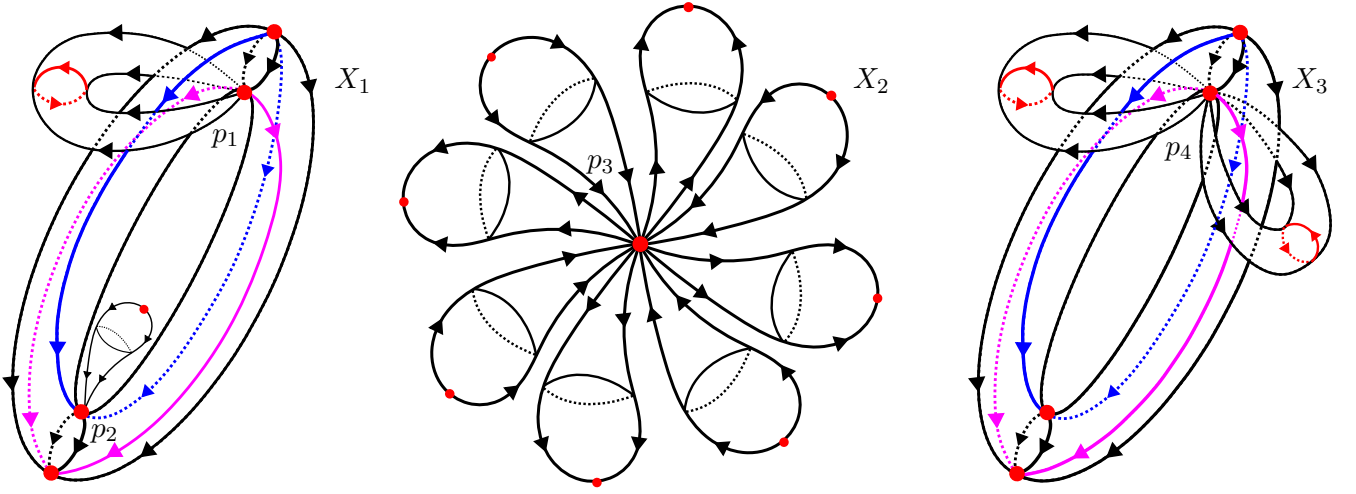
\begin{figure}[!ht]
\centering
\begin{overpic}[width=\textwidth]{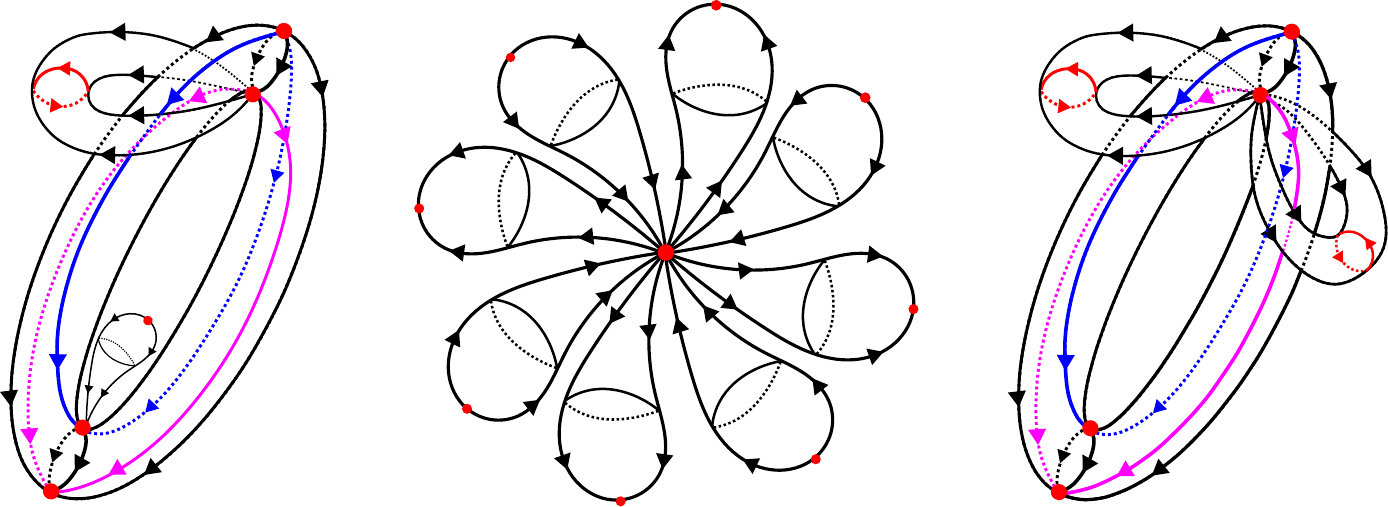}
\put(15.8, 26){$p_1$}
\put(6.7, 3.8){$p_2$}
\put(44,24){$p_3$}
\put(87.5,25){$p_4$}
\put(25,30){$X_1$}
\put(64,30){$X_2$}
\put(97,30){$X_3$}
\end{overpic}
\caption{Examples of $2$-pseudomanifolds endowed with  singular Morse-Smale flows.}
\label{figura_exemplo_01}
\end{figure}

By conditions $(i), (ii)$ of Definition~\ref{definicao_fluxo}, and since  $X \setminus \left( \bigcup_{p\in \operatorname{Sing}(X)} U_p \right)$ is a compact $n$-dimensional manifold, the set of singularities of Morse index $i$, denoted by $\operatorname{Crit}_i(\varphi)$, is finite.\footnote{By hypothesis, the singularities of $\varphi$ on $X \setminus \left( \bigcup_{p\in \operatorname{Sing}(X)} U_p \right)$ are hyperbolic, hence  they are locally given by nondegenerate critical points of a Morse function, this justifies the terminology.} By condition $(iii)$, all topological cone singularities of $X$ are also dynamical singularities.  Thus, the set of all singularities of the flow $\varphi$ is given by
\[
    \operatorname{Sing}(\varphi) = \operatorname{Sing}(X)\cup \left( \bigcup_{i=0}^n \operatorname{Crit}_i(\varphi) \right).
\]
For notational convenience, we introduce the following definition.

\begin{definition}
\label{definicao_natureza}
Let $X$ be an $n$-pseudomanifold with spherical-cone singularities, and let $\varphi$ be a singular Morse-Smale flow on $X$. Given $p \in \operatorname{M}(\mathcal{C}_k)$, the \textit{nature} of $p$ is defined by the word
\[
    a^{\eta_0(p)} s_1^{\eta_1(p)} s_2^{\eta_2(p)}\cdots s_{n-1}^{\eta_{n-1}(p)} r^{\eta_n(p)},
\]
where, for  each $\lambda \in \{0, \ldots, n\}$, the integer $\eta_{\lambda}(p)$  counts the number of sheets of  $p$  on which the restriction of $\varphi$ has a hyperbolic singularity of Morse index $\lambda$. The number $\eta_i(p)$ is called the \textit{$i$-th nature number} of $p$. For a singularity $x$ of $\varphi$ in the regular part of $X$ with Morse index $\lambda$, the \textit{$i$-th nature number} of $x$ is given by $\eta_i(x)=0$ for $i\neq \lambda$ and $\eta_i(x)=1$ for $i=\lambda$.
\end{definition}

When the dependence on $p$ is clear from the context, we omit it and simply write $a^{\eta_0} s_1^{\eta_1} \cdots s_{n-1}^{\eta_{n-1}} r^{\eta_n}$. Note that if $p\in \operatorname{M}(\mathcal{C}_k)$ has nature $a^{\eta_0} s_1^{\eta_1} \cdots s_{n-1}^{\eta_{n-1}} r^{\eta_n}$, then 
\[
    \sum_{i=0}^n \eta_i = k \quad \text{and} \quad 0 \leq \eta_i \leq k, \quad \text{for all } i \in \{0, \ldots, n\}.
\]

For example, in Figure~\ref{figura_exemplo_01}, the spherical-cone singularity $p_1\in \operatorname{M}(\mathcal{C}_3)$ has nature $s_1^1r^2=s_1r^2$, and $p_2\in \operatorname{M}(\mathcal{C}_2)$ has nature $as_1$. Also in Figure~\ref{figura_exemplo_01}, we have $\operatorname{Sing}(X_2)=\{ p_3\}$, where $p_3\in \operatorname{M}(\mathcal{C}_8)$ has nature $a^4r^4$. Finally, $p_4\in \operatorname{M}(\mathcal{C}_5)$ has nature $s_1r^4$.

Furthermore, we point out that if $X$ is an $n$-pseudomanifold with spherical-cone singularities endowed with a singular Morse-Smale flow, the standard definitions of $\alpha$- and $\omega$-limit sets, as well as stable and unstable manifolds, adapt naturally to this setting. Although the stable and unstable manifolds are no longer smooth submanifolds globally, they still exhibit a well-behaved structure, as described in the following remark.

\begin{remark}
\label{remark_stable_unstable_cone}
Let $p \in \operatorname{M}(\mathcal{C}_k)$ be a singularity of nature $a^{\eta_0}s_1^{\eta_1}\ldots s_{n-1}^{\eta_{n-1}}r^{\eta_n}$. Then, $W^u(p)\setminus \{p\}$ is the disjoint union of $k-\eta_0$ submanifolds, with exactly $\eta_{\lambda}$ submanifolds of dimension $\lambda$ for $\lambda=1, 2, \ldots, n$. Similarly, $W^s(p)\setminus \{p\}$ is the disjoint union of $k-\eta_n$ submanifolds, with exactly $\eta_{\lambda}$ submanifolds of dimension $n-\lambda$ for $\lambda=0, 1, \ldots, n-1$. 
\end{remark}

We now present our first result. Specifically, we determine the homotopy Conley index for each spherical-cone singularity of a singular Morse-Smale flow. As a consequence, we obtain a formula for the numerical Conley index.  Since the Conley index is well know for singularities and periodic orbits of Morse-Smale flows in the smooth case (see Example~\ref{example_conley_smooth}), the following results focus exclusively on the cone singularities $p\in \operatorname{Sing}(X)$.

\begin{proposition}
\label{proposition_indice}
Let $X$ be an $n$-pseudomanifold with spherical-cone singularities endowed with  a singular Morse-Smale flow  $\varphi$. Let $p \in \operatorname{M}(\mathcal{C}_k)$, the homotopy Conley index of $p$ is given as follows:
\begin{enumerate}[label=(\alph*), font=\itshape, leftmargin=*, align=left]
    \item If $p$ has nature $a^k$, then $h(p) = \es^0$.
    \item If $p$ has nature $a^{\eta_0} s_1^{\eta_1} s_2^{\eta_2}\ldots s_{n-1}^{\eta_{n-1}} r^{\eta_n}$ with $\eta_0 < k$, then
    \[
        h(p) = \left( \bigvee_{i=1}^{\eta_1 + (k-\eta_0-1)} \es^1 \right) \vee \left( \bigvee_{i=1}^{\eta_2} \es^2 \right) \vee \cdots \vee \left( \bigvee_{i=1}^{\eta_{n-1}} \es^{n-1} \right) \vee \left( \bigvee_{i=1}^{\eta_n} \es^n \right).
    \]
\end{enumerate}
\end{proposition}

\begin{proof}
    Given $p\in \operatorname{M}(\mathcal{C}_k)$, consider an isolating neighborhood $N$ for $\{p\}$ which, up to homeomorphism, is given by
    \[
    N = \bigvee_{i=1}^{k} \overline{\mathbb{B}}_i,
    \]
    where $\overline{\mathbb{B}}_i$ denotes a copy of the closed unit ball $\overline{\mathbb{B}}=\{x \in \mathbb{R}^n \mid \|x\| \leq 1\}$. The compact set $N$ can be constructed using condition $(iii)$ of Definition~\ref{definicao_fluxo}, together with the transversality of the flow to the boundary. 

    If $p$ is an attracting singularity, the exit set of the flow on $N$ is empty. Since $N$ is contractible, the result in item $(a)$ follows.
    
    Now, suppose that $p \in \operatorname{M}(\mathcal{C}_k)$ has nature $a^{\eta_0} s_1^{\eta_1} s_2^{\eta_2} \ldots s_{n-1}^{\eta_{n-1}} r^{\eta_n}$ with $\eta_0 < k$. Observe that the restriction of $\varphi$ to each $\overline{\mathbb{B}}_i$ is a Morse-Smale flow without periodic orbits. Let $L_i$ denote the exit set on the boundary of each $\overline{\mathbb{B}}_i$, and let $L \subset N$ be the disjoint union of these exit sets $L_i$. Since exactly $\eta_0$ sheets are attracting, exactly $k-\eta_0$ of these exit sets are nonempty.
    Thus, $(N,L)$ is an index pair for $p$, since each pair $(\overline{\mathbb{B}}_i, L_i)$ is an index pair for the singularity restricted to the sheet $\overline{\mathbb{B}}_i$.

    We now determine the homotopy type of the pointed space $(N/L, [L] )$. To this end, we analyze the structure of the quotient space $N/L$.

    First, by collapsing each of the $k-\eta_0$ nonempty exit sets $L_i$ to a distinct point, we obtain the following homotopy type: 
    \begin{equation}
    \label{eq_prop_conley}
        \left( \bigvee_{i=1}^{\eta_1} \es^1 \right) \vee \left( \bigvee_{i=1}^{\eta_2} \es^2 \right) \vee \cdots \vee \left( \bigvee_{i=1}^{\eta_{n-1}} \es^{n-1} \right) \vee \left( \bigvee_{i=1}^{\eta_n} \es^n \right),
    \end{equation}
    where $p$ is the basepoint of the wedge sum. Each sphere arises from the fact that the restriction of $\varphi$ to each $\overline{\mathbb{B}}_i$ admits a unique nondegenerate critical point of Morse index $\lambda$; consequently, its Conley index is precisely $\es^{\lambda}$ (see Example~\ref{example_conley_smooth}).

    After this first step, the subset $L$ is transformed into a set of $k-\eta_0$ distinct points, located on the spheres within the wedge sum described in~\eqref{eq_prop_conley}. Identifying these $k-\eta_0$ points to a single point yields the quotient space $N/L$. Up to homotopy, this identification generates an additional $(k-\eta_0)-1$ spheres of dimension $1$ in the wedge sum. Therefore, the homotopy type of the pointed space is given by 
    \[
        h(p) = \left( \bigvee_{i=1}^{\eta_1+(k-\eta_0 -1)} \es^1 \right) \vee \left( \bigvee_{i=1}^{\eta_2} \es^2 \right) \vee \cdots \vee \left( \bigvee_{i=1}^{\eta_{n-1}} \es^{n-1} \right) \vee \left( \bigvee_{i=1}^{\eta_n} \es^n \right).
    \]
\end{proof}

\begin{corollary}
\label{corollary_indice}
Let $X$ be an $n$-pseudomanifold with spherical-cone singularities endowed with  a singular Morse-Smale flow $\varphi$. For any $p \in \operatorname{M}(\mathcal{C}_k)$, the numerical Conley index of $p$ is given as follows.
\begin{enumerate}[label=(\alph*), font=\itshape, leftmargin=*, align=left]
    \item If $p$ has nature $a^k$, then $h_{\ast}(p) = (1, 0, \ldots, 0)$. 
    \item If $p$ has nature $a^{\eta_0} s_1^{\eta_1} s_2^{\eta_2}\ldots s_{n-1}^{\eta_{n-1}} r^{\eta_n}$ with $\eta_0 < k$, then
    \[
        h_{\ast}(p) = \left(0, \eta_1 + (k - \eta_0 - 1), \eta_2, \ldots, \eta_{n-1}, \eta_n\right).
    \]
\end{enumerate}  
\end{corollary}

\begin{proof}
These equalities follow directly from Proposition~\ref{proposition_indice}, the definition of the numerical Conley index (see Subsection~\ref{subsection_conley_theory_background}), and the reduced homology of spheres.
\end{proof}

In Examples~\ref{example_spheres} and~\ref{example_torus}, we apply Corollary~\ref{corollary_indice} to compute the numerical Conley index for specific spherical-cone singularities.  

\begin{example}
\label{example_spheres}
Consider the disjoint union of $k$ $n$-dimensional spheres $\mathbb{S}^n_i$, each equipped with a North-South flow $\varphi_i\colon \mathbb{R}\times \mathbb{S}^n_i\to \mathbb{S}^n_i$, which is a Morse-Smale flow; that is, each sphere admits exactly two nondegenerate critical points: a repeller (the North Pole) and an attractor (the South Pole).

Let $X = \bigvee_{i=1}^k \mathbb{S}^n_i$, where the base point $p$ of the wedge sum is formed by identifying $0 < \eta_0 < k$ South Poles and $0 < \eta_n < k$ North Poles, such that $\eta_0+\eta_n=k$. Thus, $X$ is a closed $n$-pseudomanifold with a spherical-cone singularity, where $p \in \operatorname{M}(\mathcal{C}_k)$.

Now consider the singular Morse-Smale flow $\varphi\colon \mathbb{R}\times X\to X$ on $X$  defined by 
\[
\varphi(t,x)\coloneqq
\begin{cases}
    p, & \textit{if }x=p,\\
    \varphi_i(t,x), & \textit{if }x\neq p,
\end{cases}
\]
where $i$, in the previously equation, is the unique index $i$ such that $x\in \mathbb{S}^n_i\setminus\{ p \}$. By construction, the singular point $p\in \operatorname{M}(\mathcal{C}_{\eta_0+\eta_n})$ has nature $a^{\eta_0}r^{\eta_n}$. By Corollary~\ref{corollary_indice}, the numerical Conley index of $p$ is
\[
    h_{\ast}(p) = (0, \eta_n -1, 0, \dots, 0 , \eta_n ).
\]

A particular case of this example is given by the $2$-pseudomanifold $X_2$ in Figure~\ref{figura_exemplo_01}, where $\operatorname{Sing}(X_2)=\{p_3\}$ and $p_3\in \operatorname{M}(\mathcal{C}_8)$ has nature $a^4r^4$, yielding $h_{\ast}=(0,3,4)$. 
\hfill $\diamond$
\end{example}

\begin{example}
\label{example_torus}
For this example, let $n \in \mathbb{N}$ be an even number with $n\geq 2$ and consider the $n$-dimensional torus $T^n = \mathbb{S}^1 \times \cdots \times \mathbb{S}^1$. Points of $\mathbb{T}^n$ may be represented by standard angular coordinates $(\theta_1, \dots, \theta_n)$ in fundamental domain 
\[
    \left\{ (\theta_1, \dots, \theta_n)\in \mathbb{R}^n \mid \theta_i \in [0,2\pi ) \quad \text{for all} \quad i =1,\dots, n \right\}.
\]
Define the smooth function $f \colon T^n \to \mathbb{R}$  by 
\[
    f(\theta_1, \dots , \theta_n   )\coloneqq \sum_{j=1}^n \cos(\theta_j).
\]
Endow $T^n$ with the standard flat metric induced  by the metric on $\mathbb{R}^n$, the gradient of $f$ is given by
\[
    \nabla f (\theta_1, \dots, \theta_n) = \left( -\sin(\theta_1) , \dots , -\sin(\theta_n) \right).
\]
Thus, $(\theta_1, \dots, \theta_n)$ is a critical point if and only if $\theta_j \in \{0, \pi\}$, for all $j=1, \dots, n$. Hence, there are $2^n$ critical points. Moreover, the Hessian matrix of $f$ in these coordinates is 
\[
H_f(\theta_1, \dots , \theta_n)=
\begin{bmatrix}
 -\cos(\theta_1) & 0 & \dots & 0 \\
 0 & -\cos(\theta_2) & \dots & 0 \\
 \vdots & \vdots & \ddots & \vdots \\
 0 & 0 & \dots & -\cos(\theta_n)
\end{bmatrix}.
\]
Therefore, $f$ is a Morse function. Observe that, for each $j=0, 1, \dots, n$, we have exactly:
\[
c_j= \binom{n}{j} = \frac{n!}{j!(n-j)!}, 
\]
critical points of Morse index $j$. After a sufficiently small perturbation, the vector field $-\nabla f$ induces a Morse-Smale flow $\varphi$ with one repelling singularity and one attracting singularity, and for each $\lambda=1, \dots, n-1$, we have exactly $c_\lambda$ saddles of index $\lambda$.

Now, consider $T^n$ endowed with this flow $\varphi$. Let $M_1$ be the disjoint union of $m_1$ copies of the $n$-dimensional sphere, each equipped with a flow consisting of two isolated repelling singularities and an attracting closed orbit at the equator. This construction relies fundamentally on the hypothesis that $n$ is even, since the $(n-1)$-dimensional equatorial sphere admits a closed orbit only when $n-1$ is odd. Let $M_2$ be the disjoint union of $m_2$ copies of the $n$-dimensional sphere, each exhibiting North-South dynamics. Finally,  define the singular space
\[
 X \coloneqq \left( T^n\sqcup M_1 \sqcup M_2\right) / \sim ,
\]
where $\sim$ is the equivalence relation obtained by collapsing the $2m_1$ repellers of $M_1$ into a saddle of index $\lambda$ on $T^n$, resulting in a point $p_1 \in \operatorname{M}(\mathcal{C}_{2m_1+1})$, and collapsing the $m_2$ attractors of $M_2$ into a distinct saddle of index $\lambda$ on $T^n$, resulting in a point $p_2 \in \operatorname{M}(\mathcal{C}_{m_2+1})$. This is possible because, for every even $n$ and $1 \leq \lambda \leq n-1$, we have $c_{\lambda}\geq 2$.  On this $n$-pseudomanifold with spherical-cone singularities, we can define a singular Morse-Smale flow through a construction analogous to that of Example~\ref{example_spheres}. 

In this configuration, the singular point $p_1\in \operatorname{M}(\mathcal{C}_{2m_1+1})$ has nature $s_{\lambda}r^{2m_1}$ and $p_2\in \operatorname{M}(\mathcal{C}_{m_2+1})$ has nature $a^{m_2}s_{\lambda}$. By Corollary~\ref{corollary_indice}, we obtain
\[
h_{\ast}(p_1)= 
\begin{cases}
    \left( 0, 2m_1+1, 0, \dots , 0, 2m_1  \right), &\quad \text{if } \lambda=1;\\
    \left( 0, 2m_1, 0, \dots , 0, 1, 0, \dots , 0, 2m_1 \right), &\quad \text{if } \lambda \neq 1;
\end{cases}
\]
and
\begin{align*}
h_{\ast}(p_2)&= 
\begin{cases}
    \left( 0, 1 +((m_2+1)-\eta_0-1), 0, \dots , 0  \right), &\quad \text{if } \lambda=1;\\
    \left( 0, (m_2 +1 )-\eta_0 -1, 0, \dots , 0, 1, 0, \dots , 0\right), &\quad \text{if } \lambda \neq 1.
\end{cases}
\\
&= 
\begin{cases}
    \left( 0, 1, 0, \dots , 0  \right), &\quad \text{if } \lambda=1;\\
    \left( 0, 0, \dots , 0, 1, 0, \dots , 0 \right), &\quad \text{if } \lambda \neq 1.
\end{cases}
\end{align*}

Two particular cases of this example are illustrated in Figure~\ref{figura_exemplo_01} by the pseudomanifolds $X_1$ and $X_3$. Using the previous formulas, the numerical Conley indices of the singularities in $X_1$ and $X_3$ are given by:
\[
    h_{\ast}(p_1) = (0,3,2),\quad  h_{\ast}(p_2) = (0,1,0), \quad \text{and} \quad h_{\ast}(p_4) = (0,5,4).
\]\hfill $\diamond$
\end{example}

\subsection{Lyapunov Function for Singular Morse-Smale Flows}
\label{subsection_lyapunov_function}

In this subsection, we establish the local and global existence of Lyapunov functions for singular Morse-Smale flows on $n$-pseudomanifolds with spherical-cone singularities. The use of these functions has proven to be effective for obtaining results related to the Conley index, as investigated in~\cite{cruz1999, ketty1993lyapunov, Lima2025, Montufar}. We begin by formalizing the definition of a Lyapunov function in this setting.

\begin{definition}
\label{definition_lyapunov_function}
Let $X$ be an $n$-pseudomanifold with spherical-cone singularities endowed with a singular Morse-Smale flow $\varphi$. A continuous function $f\colon X \to \mathbb{R}$ is called a \textit{Lyapunov function} for $\varphi$ if the following conditions hold:
    \begin{enumerate}[label=(\roman*), leftmargin=*, font=\itshape, align=left]
        \item $\restr{f}{ X \setminus \operatorname{Sing}(X) }$ is a smooth function;  
        \item $\frac{d}{dt} \restr{f}{X \setminus \operatorname{Sing}(X)} (\varphi(t,x)) < 0$ for all $x \notin \mathcal{R}(\varphi)$, where $\mathcal{R}(\varphi)$ denotes the chain recurrent set of $\varphi$;
        \item $f$ is constant on each connected component of $\mathcal{R}(\varphi)$ and takes distinct values on distinct components.
    \end{enumerate}
\end{definition}

The values of $f$ on the critical elements $\mathcal{R}(\varphi)$ are called \textit{critical values}; all other values in the image of $f$ are called \textit{regular values}. Combining results from~\cite{Conley} and~\cite{wilson1969smoothing}, it follows that any smooth flow on a compact smooth manifold admits a Lyapunov function.

In the context of pseudomanifolds with spherical-cone singularities, the first step in constructing a global Lyapunov function is to establish the existence of such functions locally, specifically within the isolating neighborhood of each singularity. The following result guarantees the local existence.

\begin{lemma}
\label{lemma_lyapunov_local}
Let $X$ be an $n$-pseudomanifold with spherical-cone singularities endowed with a singular Morse-Smale flow $\varphi$. For each $p \in \operatorname{Sing}(X)$, there exist a neighborhood $N$ of $p$ in $X$ and a Lyapunov function $f\colon N \to \mathbb{R}$ for $\varphi$. 
\end{lemma}

\begin{proof}
Let $p \in \operatorname{M}(\mathcal{C}_k)$ be a spherical-cone singularity whose nature is described by the word 
\[
    a^{\eta_0} s_1^{\eta_1} s_2^{\eta_2}\ldots s_{n-1}^{\eta_{n-1}} r^{\eta_n}.
\]
Let $N$ be a compact neighborhood of $p$ in $X$ that, up to homeomorphism, is given by
\[
    N = \bigvee_{i=1}^k \overline{\mathbb{B}}_i,    
\]
where each $\overline{\mathbb{B}}_i$ denotes a copy of the closed unit ball $\overline{\mathbb{B}}=\{x \in \mathbb{R}^n \mid \|x\| \leq 1\}$. Since the restriction of $\varphi$ to each $\overline{\mathbb{B}}_i$ is a Morse-Smale flow without periodic orbits, we can characterize $\varphi$ locally on $N$. On each $\overline{\mathbb{B}}_i$, there exists a local coordinate system $(x_1,\dots , x_n)$ in which the flow $\varphi$ takes one of the following normal forms:
\begin{itemize}
\item the flow $\varphi^a \colon \mathbb{R}\times \overline{\mathbb{B}}^a_i \to \overline{\mathbb{B}}_i$ is given by
\[
    \varphi^a (t, x) = (e^{-t}x_1, \dots , e^{-t}x_n),
\]
yielding an attracting singularity;
\item the flow $\varphi^{s_\lambda} \colon \mathbb{R}\times \overline{\mathbb{B}}^{s_\lambda}_i \to \overline{\mathbb{B}}_i$ is given by
\[
    \varphi^{s_\lambda} (t, x) = (e^{t}x_1, \dots, e^{t}x_{\lambda}, e^{-t}x_{\lambda+1}, \dots , e^{-t}x_n),
\]
yielding a saddle singularity;
\item the flow $\varphi^r \colon \mathbb{R}\times \overline{\mathbb{B}}^r_i \to \overline{\mathbb{B}}_i$ is given by
\[
    \varphi^r (t, x) = (e^{t}x_1, \dots , e^{t}x_n),
\]
yielding a repelling singularity.
\end{itemize}

Here, $\varphi^a$ and $\varphi^r$ denote the restrictions of $\varphi$ to the sheets containing an attracting and a repelling singularity, respectively, while $\varphi^{s_\lambda}$ represents the restriction of $\varphi$ to a sheet containing a saddle singularity of Morse index $\lambda$, for each $\lambda \in \{ 1, \dots, n-1 \}$.  Accordingly, we denote the sheet by $\overline{\mathbb{B}}^{a}_i$, $\overline{\mathbb{B}}^{r}_i$, or $\overline{\mathbb{B}}^{s_\lambda}_i$ depending on whether the restriction of $\varphi$ to $\overline{\mathbb{B}}_i$ exhibits an attracting, repelling, or saddle singularity of Morse index $\lambda$. Thus, in these coordinates, we have
    \begin{align*}
        \varphi (t,x) =
        \begin{cases}
            p, & \text{if } x=p;\\
            \varphi^a(t,x), & \text{if } x\in \overline{\mathbb{B}}^a_i\setminus \{ p\};\\
            \varphi^{s_\lambda}(t,x), & \text{if } x\in \overline{\mathbb{B}}^{s_\lambda}_i\setminus \{ p\};\\
            \varphi^r(t,x), & \text{if } x\in \overline{\mathbb{B}}^r_i \setminus \{p\}.
        \end{cases}
    \end{align*}
Associated with each of these local models, we define the following local functions:
\begin{align*}
f^a \colon \overline{\mathbb{B}}^a_i  \to \mathbb{R},  & \quad  f^a(x_1, \dots, x_n)  \coloneqq   x_1^2 + \cdots + x_n^2; \\[1em]
f^{s_\lambda} \colon  \overline{\mathbb{B}}^{s_\lambda}_i  \to \mathbb{R}, & \quad f^{s_\lambda}(x_1, \dots, x_n)  \coloneqq   -x_1^2 - \cdots - x_{\lambda}^2 + x_{\lambda+1}^2 + \cdots + x_n^2; \\[1em]    
f^r \colon \overline{\mathbb{B}}^r_i \to \mathbb{R},  & \quad f^r(x_1, \dots, x_n)  \coloneqq   -(x_1^2 + \cdots + x_n^2).
\end{align*}

Note that $f^a(0)=f^{s_\lambda}(0)=f^r(0)=0$ at the singularity $p$ in these coordinates for all $\lambda \in \{ 1, \dots, n-1 \}$. This allows us to continuously glue these functions at $p$. Finally, we define $f \colon N \to \mathbb{R}$ by
\[
    f(x) \coloneqq 
    \begin{cases}
        f^a(x), & \text{if } x \in \overline{\mathbb{B}}^a_i; \\
        f^{s_\lambda}(x), & \text{if } x \in \overline{\mathbb{B}}^{s_\lambda}_i; \\
        f^r(x), & \text{if } x \in \overline{\mathbb{B}}^r_i.
    \end{cases}
\]

We claim that $f$ is a Lyapunov function for $\varphi$ on $N$. Indeed, the restriction
\[
\restr{f}{N \setminus  \operatorname{Sing}(N)} = \restr{f}{N \setminus \{p\}}
\]
is smooth and has no critical points. Furthermore, we can verify the strict decrease of $f$ along regular flow lines by evaluating $\frac{d}{dt} \restr{f}{ N\setminus \{p\}} (\varphi(t,x))$ in each sheet:
   
\begin{itemize}
\item If $x\in \overline{\mathbb{B}}^a_i \setminus \{p \}$, then
\begin{align*}
    \frac{d}{dt}\left(  \restr{f}{ N\setminus \{p\}}\circ \varphi(t,x) \right) &= \frac{d}{dt}\left( f^a(\varphi^a(t,x))  \right)\\
    &= \frac{d}{dt}\left( e^{-2t}x^2_1 + \cdots + e^{-2t}x^2_n \right)\\
    &=-2e^{-2t} \left( x^2_1 + \cdots + x^2_n \right)<0.    
\end{align*}

\item If $x\in \overline{\mathbb{B}}^{s_\lambda}_i \setminus \{p \}$, then
\begin{align*}
    \frac{d}{dt}\left(  \restr{f}{ N\setminus \{p\}}\circ \varphi(t,x) \right) &= \frac{d}{dt}\left( f^{s_\lambda}(\varphi^{s_\lambda}(t,x))  \right)\\
    &= \frac{d}{dt}\left( -e^{2t}x^2_1 - \cdots - e^{2t}x^2_\lambda + e^{-2t}x^2_{\lambda +1} + \cdots + e^{-2t}x^2_{n} \right)\\
    &=-2\left( e^{2t}x^2_1 + \cdots + e^{2t}x^2_\lambda +e^{-2t}x^2_{\lambda +1} + \cdots + e^{-2t}x^2_n \right) < 0.   
\end{align*}

    \item If $x\in \overline{\mathbb{B}}^r_i \setminus \{p \}$, then
    \begin{align*}
        \frac{d}{dt}\left(  \restr{f}{ N\setminus \{p\}}\circ \varphi(t,x) \right) &= \frac{d}{dt}\left( f^r(\varphi^r(t,x))  \right)\\
        &= \frac{d}{dt}\left( -e^{2t}x^2_1 - \cdots - e^{2t}x^2_n \right)\\
        &=-2 e^{2t} \left(x^2_1 + \cdots + x^2_n \right)<0.    
    \end{align*}
    \end{itemize}
Hence, in all cases, $\frac{d}{dt} \restr{f}{N \setminus \{p\}} (\varphi(t,x)) < 0$, and $f(p)$ is the unique critical value of $f$. This concludes the proof that $f$ is a Lyapunov function for $\varphi$ on the neighborhood $N$. 
\end{proof}

In our setting, Lemma~\ref{lemma_lyapunov_local} constitutes the primary step toward the construction of global Lyapunov functions. The subsequent lemmas follow as direct consequences of the local functions provided by Lemma~\ref{lemma_lyapunov_local}. These results adapt classical techniques from~\cite{Smale1961}, and we explicitly detail modifications to their proofs where appropriate.

Let $X$ be a compact $n$-pseudomanifold with spherical-cone singularities with possibly nonempty boundary $\partial X$, and let  $\varphi$ be a singular Morse-Smale flow on $X$. Throughout the paper, we assume that $\varphi$ is transverse to the boundary, which allows us to decompose $\partial X$ into exactly two disjoint sets: the exit set $\partial X^-$, and the entrance set $\partial X^+ \coloneqq \partial X \setminus \partial X^-$.

The sequence $(G_i)_i$ established in the following lemma constitutes what is known as a \textit{filtration} for a Morse-Smale flow. For further details regarding the smooth case, we refer the reader to~\cite[Chapter 4]{palis1982geometric}.

\begin{lemma}
\label{Lema_sequencia}
Let $X$ be a compact $n$-pseudomanifold with spherical-cone singularities and possibly nonempty boundary $\partial X$, endowed with a singular Morse-Smale flow $\varphi$. Assume the chain recurrent set of $\varphi$ is given by $\mathcal{R}(\varphi) = \{ \sigma_1, \dots , \sigma_{m_0}\}$. Then there exists a collection of pairwise disjoint, closed, $(n-1)$-dimensional submanifolds $B_0, B_1, \dots, B_{m_0}$ of $X$ satisfying the following properties:
\begin{enumerate}[label=(\alph*), leftmargin=*, font=\itshape, align=left]
    \item $B_0 = \partial X^-$ and $B_{m_0} = \partial X^+$;
    \item the flow $\varphi$ is transverse to each $B_i$;
    \item for each $i \notin \{0, m_0\}$, the submanifold $B_i$ divides $X$ into two regions whose closures are denoted by $G_i$ and $H_i$. These regions satisfy
    \[
        G_{i-1} \subset G_i \quad \text{and} \quad H_{i+1} \subset H_i,
    \]
    and $G_i$ contains exactly $i$ critical elements of $\varphi$. Moreover, by setting $G_0 \coloneqq B_0$, $G_{m_0} \coloneqq X$, $H_0 \coloneqq X$, and $H_{m_0} \coloneqq B_{m_0}$, we have
    \[
        G_i \cap H_i = B_i \quad \text{and} \quad G_i \cup H_i = X,
    \]
    for all $i \in \{0, \dots, m_0\}$; 
    \item $B_i$ is the entrance boundary of $G_i$ with respect to $\varphi$.
\end{enumerate}
\end{lemma}

\begin{proof}
First, we can establish a partial order on $\mathcal{R}(\varphi)$ using the order given in Definition~\ref{definition_ordem_singularidades}, since the definitions of stable and unstable manifolds extend naturally for spherical-cone singularities, see Remark~\ref{remark_stable_unstable_cone}.

The idea of the proof is to construct a sequence of sets $(B_i)_{i=0}^{m_0}$ with the desired properties by induction. To this end, we start by defining $B_0 \coloneqq \partial X^-$. As our induction hypothesis, assume that we have constructed a sequence of sets up to some $B_{i-1}$, and consequently $G_{i-1}$ and $H_{i-1}$, satisfying
\[
    X = G_{i-1} \cup H_{i-1} \quad \text{and} \quad B_{i-1} = G_{i-1} \cap H_{i-1},
\]
such that $G_{i-1}$ contains exactly $i-1$ critical elements of $\varphi$, respecting the partial order given in Definition~\ref{definition_ordem_singularidades}. 

To construct $B_i$, and consequently $H_i$ and $G_i$, we consider a neighborhood of $B_{i-1}$ of the form
\[
    B_{i-1} \times [-1,1]
\]
(for $i=1$, we instead take $B_0 \times [0,1]$). Here, we identify $B_{i-1}$ with $B_{i-1} \times \{0\}$, hence $B_{i-1} \times [0,1] \subset H_{i-1}$, and the flow $\varphi$ is transverse to  $B_{i-1} \times \{t\}$ for each $t$.

Now, let $p \in \operatorname{M}(\mathcal{C}_k)$ be an arbitrary spherical-cone singularity with nature described by the word
\[
a^{\eta_0}s_1^{\eta_1}\dots s_{n-1}^{\eta_{n-1}}r^{\eta_n},
\]
such that $p \in X \setminus G_{i-1}$ and $\sigma \leq p$ for some critical element $\sigma \in G_{i-1}$, as given by Definition~\ref{definition_ordem_singularidades}. It suffices to consider the case where $p$ is a singular point, as singularities in the regular stratum are already addressed by classical arguments in~\cite{Smale1961} and~\cite[Chapter 4, Lemma 1.2]{palis1982geometric}. 

By Lemma~\ref{lemma_lyapunov_local}, there exist a neighborhood $N$ of $p$ in $X$ and a local Lyapunov function $f \colon N \to \mathbb{R}$. We may choose $\delta > 0$ such that the level set
\[
    \widetilde{W} \coloneqq f^{-1}(\delta) \cap W^s(p)
\]
is contained entirely within $N$. Note that $W^s(p) \setminus \{ p \}$ is a disjoint union of $k-\eta_n$ submanifolds, containing exactly $\eta_{\lambda}$ submanifolds of dimension $\lambda$ for each $\lambda \in \{1, 2, \dots, n-1\}$, see Remark~\ref{remark_stable_unstable_cone}. Thus, we can take $k-\eta_n$ normal bundles $E_{\epsilon, j}$ of $W^s(p) \setminus \{p\}$ in $X$ restricted to $\widetilde{W}$, considering only vectors of norm at most $\epsilon$. Consequently, we can follow exactly the same geometric construction steps as in the proof of~\cite[Lemma 2.1]{Smale1961}, or analogously~\cite[Lemma 2.5]{Montufar}, to construct the set $B_i$ and the corresponding sets $G_i$ and $H_i$.

By hypothesis, $\mathcal{R}(\varphi) = \{\sigma_1, \dots, \sigma_{m_0}\}$ is finite set. Therefore, the inductive process must terminate. This yields a finite sequence of submanifolds $B_1, \dots, B_{m_0}$ satisfying all necessary properties. 
\end{proof}

\begin{lemma}
\label{lemma_constantes_lyapunov}
Let $X$ be a compact $n$-pseudomanifold with spherical-cone singularities, with possibly nonempty boundary $\partial X$, and  let $\varphi$ be a singular Morse-Smale flow on $X$ that contains exactly one critical element $\sigma$ in $X$. Then, for any given $c \in \mathbb{R}$, there exists a Lyapunov function $f \colon X \to \mathbb{R}$ for $\varphi$ such that $f \equiv c-\frac{1}{2}$ on $\partial X^-$, $f \equiv c+\frac{1}{2}$ on $\partial X^+$, and $f(\sigma) = c$.
\end{lemma}

\begin{proof}
As in the previous proof, it suffices to consider the case where $p \in \operatorname{M}(\mathcal{C}_k)$ is a spherical-cone singularity, since critical elements in the regular stratum have already been addressed by classical arguments in~\cite{Smale1961} and~\cite{meyer1968energy}.  Lemma~\ref{lemma_lyapunov_local} guarantees the existence of a neighborhood $N$ of $p$ and a local Lyapunov function $f \colon N \to \mathbb{R}$ for $\varphi$. Given $c \in \mathbb{R}$, by adding a constant if necessary, we may assume that $f(p)=c$. 

Taking $\delta > 0$ as in the proof of Lemma~\ref{Lema_sequencia}, we define the local level sets
\[
    R^{+} \coloneqq f^{-1}(c+\delta) \cap N \quad \text{and} \quad R^{-} \coloneqq f^{-1}(c-\delta) \cap N. 
\]
Using these sets $R^+$ and $R^-$, the construction of the desired global Lyapunov function on $X$ follows exactly the same geometric steps as the proof of~\cite[Lemma 2.2]{Smale1961}, or analogously~\cite[Lemma 2.6]{Montufar}. 
\end{proof}

\begin{theorem}
\label{theorem_lyapunov_global}
Every singular Morse-Smale flow on a compact $n$-pseudomanifold with spherical-cone singularities admits a Lyapunov function.
\end{theorem}

\begin{proof}
Let $X$ be compact $n$-pseudomanifold with spherical-cone singularities endowed with a singular Morse-Smale flow $\varphi$.
First, consider the filtration $(G_i)_{i=0}^{m_0}$ obtained from Lemma~\ref{Lema_sequencia}. Observe that for each $i \in \{1, \dots, m_0\}$, the compact set 
\[
    F_i \coloneqq \overline{G_i \setminus G_{i-1}}
\]
contains exactly one critical element, namely $\sigma_i$. By Lemma~\ref{lemma_constantes_lyapunov}, we can choose a Lyapunov function $f_i \colon F_i \to \mathbb{R}$ for $\varphi$ such that $f_i(\sigma_i) = c_i$. Furthermore, the space
\[
    \overline{X \setminus \bigcup_{i=1}^{m_0} F_i}
\]
is a compact, smooth $n$-manifold containing no critical elements of $\varphi$. Thus, there exists a Lyapunov function $\widetilde{f}$ on this compact manifold that agrees with each $f_i$ on their respective boundaries. Moreover, since the set of critical elements is finite, we may choose the constants $c_i$ to be pairwise distinct. By gluing the functions $f_i$ and $\widetilde{f}$, we obtain a continuous global function $f \colon X \to \mathbb{R}$ satisfying all the properties of Definition~\ref{definition_lyapunov_function}. 

Therefore, $f$ is the desired global Lyapunov function for $\varphi$ on $X$. 
\end{proof}

\subsection{Morsification of Singular Morse-Smale Flows}
\label{subsection_morsification}

The concept of morsification was originally introduced in~\cite{Lima2021} in the context of $2$-dimensional singular sets, where it was used as a tool to define the intersection number between singularities in the singular setting and to extend the Morse complex to GS surfaces. 

In this section, we construct a morsification that  plays a central role in our subsequent results. In particular, we use it in Subsection~\ref{subsection_graph} to characterize the associated Lyapunov graph, in Subsection~\ref{subsection_morsification_characteristic} to relate the Euler-Poincaré characteristic of an $n$-pseudomanifold with spherical-cone singularities to that of the smooth $n$-manifold obtained via morsification, and in Section~\ref{section_IH} to derive results in intersection homology. We begin by presenting the formal definition of morsification in our setting.

\begin{definition}
\label{def:morsificacao}
Let $X$ be an $n$-pseudomanifold with spherical-cone singularities, endowed with a singular Morse-Smale flow $\varphi$. A quadruple $(\widetilde{X}, \widetilde{\varphi}, \mathfrak{h}, \mathfrak{p})$ is called a \textit{morsification} of $(X, \varphi)$ if it satisfies the following conditions:
    \begin{enumerate}[label=(\roman*), leftmargin=*, font=\itshape, align=left]
        \item $\widetilde{X}$ is a smooth $n$-manifold;
        \item $\widetilde{\varphi}$ is a smooth Morse-Smale flow on $\widetilde{X}$; 
        \item $\mathfrak{h}\colon X \to \widetilde{X}$ is a multivalued map such that the restriction $\restr{\mathfrak{h}}{X \setminus \operatorname{Sing}(X)}$ is a homeomorphism;
        \item $\mathfrak{p}\colon \widetilde{X} \to X$ is a projection satisfying $\mathfrak{p} \circ \mathfrak{h} = \operatorname{id}_X$.
    \end{enumerate}
\end{definition}

When no confusion arises, we denote a morsification $(\widetilde{X}, \widetilde{\varphi}, \mathfrak{h}, \mathfrak{p})$ simply by $(\widetilde{X}, \widetilde{\varphi})$. In our setting, such a morsification always exists, as established by the following lemma.

\begin{lemma}
\label{lemma_morsification}
Let $X$ be a compact $n$-pseudomanifold with spherical-cone singularities, endowed with a singular Morse-Smale flow $\varphi$. Then there exists a morsification $(\widetilde{X}, \widetilde{\varphi}, \mathfrak{h}, \mathfrak{p})$ of $(X,\varphi)$.
\end{lemma}

\begin{proof}
Since $X$ is an $n$-pseudomanifold with spherical-cone singularities, it is sufficient to perform the morsification locally within an isolating neighborhood $N$ of each given $p \in \operatorname{Sing}(X)$. That is, we construct a local quadruple $(\widetilde{N}, \widetilde{\varphi}, \mathfrak{h}, \mathfrak{p})$ such that $\partial N = \partial \widetilde{N}$.

Let $p \in \operatorname{Sing}(X)$ be an arbitrary spherical-cone singularity. By definition, $p \in \operatorname{M}(\mathcal{C}_k)$ for some $k \geq 2$. We can choose an isolating neighborhood $N$ of $p$ that, up to homeomorphism, is the wedge sum $N = \bigvee_{i=1}^k \overline{\mathbb{B}}_i$ with basepoint $p$. Also, let $a^{\eta_0} s_1^{\eta_1} s_2^{\eta_2}\ldots s_{n-1}^{\eta_{n-1}} r^{\eta_n}$ be the word that describes the nature of  $p$.

Consider the space $\widetilde{N}$ defined as the disjoint union $\widetilde{N} = \bigsqcup_{i=1}^k \overline{\mathbb{B}}_i$. Within each connected component, we consider its center and denote the set of these points by
\[
    S = \{ \widetilde{p}_{0}^{1}, \dots, \widetilde{p}_{0}^{\eta_0}, \dots, \widetilde{p}_{\lambda}^{1}, \dots, \widetilde{p}_{\lambda}^{\eta_\lambda}, \dots, \widetilde{p}_{n}^{1}, \dots, \widetilde{p}_{n}^{\eta_n} \} \subset \widetilde{N}.
\]
Observe that $N \setminus \{ p \}$ and $\widetilde{N} \setminus S$ are homeomorphic. Let $g\colon N \setminus \{ p \} \to \widetilde{N} \setminus S$ denote this homeomorphism. We define the multivalued map $\mathfrak{h}\colon N \to \widetilde{N}$ by
\[
    \mathfrak{h}(x) \coloneqq
    \begin{cases}
        S, & \text{if } x = p; \\
        g(x), & \text{if } x \neq p.
    \end{cases}
\]
Conversely, we define the projection map $\mathfrak{p}\colon \widetilde{N} \to N$ by
\[
    \mathfrak{p}(y) \coloneqq
    \begin{cases}
        p, & \text{if } y \in S; \\
        g^{-1}(y), & \text{if } y \notin S.
    \end{cases}
\]
By construction, the restriction $\restr{\mathfrak{h}}{N \setminus \{p\}}$ is a homeomorphism, and it holds that $\mathfrak{p} \circ \mathfrak{h} = \operatorname{id}_N$.

As noted in the proof of Lemma~\ref{lemma_lyapunov_local}, the flow $\varphi\colon \mathbb{R} \times N \to N$ can be written as
\[
    \varphi(t,x) =
    \begin{cases}
        p, & \text{if } x = p; \\
        \varphi^a(t,x), & \text{if } x \in \overline{\mathbb{B}}^{a}_i \setminus \{p\}; \\
        \varphi^{s_\lambda}(t,x), & \text{if } x \in \overline{\mathbb{B}}^{s_\lambda}_i \setminus \{p\}; \\
        \varphi^r(t,x), & \text{if } x \in \overline{\mathbb{B}}^{r}_i \setminus \{p\}.
    \end{cases}
\]
Thus,  define the map  $\widetilde{\varphi}\colon \mathbb{R} \times \widetilde{N} \to \widetilde{N}$ as follows
\[
    \widetilde{\varphi}(t,x) \coloneqq
    \begin{cases}
        \varphi^a(t,x), & \text{if } x \in \overline{\mathbb{B}}^{a}_i; \\
        \varphi^{s_\lambda}(t,x), & \text{if } x \in \overline{\mathbb{B}}^{s_\lambda}_i; \\
        \varphi^r(t,x), & \text{if } x \in \overline{\mathbb{B}}^{r}_i.
    \end{cases}
\]
The map  $\widetilde{\varphi}$ is continuous since $\widetilde{N}$ is equipped with the disjoint union topology. Moreover, $\widetilde{\varphi}$ satisfies the standard properties of a flow, and its set of singularities corresponds exactly to $S$. Hence, $\widetilde{\varphi}$ yields a continuous flows on $\widetilde{N}$. Consequently, the quadruple $(\widetilde{N}, \widetilde{\varphi}, \mathfrak{h}, \mathfrak{p})$ constitutes a local morsification of $(N, \varphi)$.

Finally, to construct a global morsification on $X$, it suffices to apply this local morsification to the isolating neighborhood of each singularity in $\operatorname{Sing}(X)$ and glue the resulting spaces to the regular part of the pseudomanifold. If, upon gluing, the resulting global flow fails to satisfy the necessary Morse-Smale transversality conditions, a standard small perturbation is sufficient to achieve the desired flow, see~\cite{Smale1961}. This yields the global morsification $(\widetilde{X}, \widetilde{\varphi}, \mathfrak{h}, \mathfrak{p})$. 
\end{proof}

\begin{remark}
\label{remark_morsificacao}
The morsification established in~\cite{Lima2021} for $2$-dimensional spherical-cone singularities differs from the one constructed in Lemma~\ref{lemma_morsification}. Throughout this paper, any reference to a morsification refers exclusively to the construction given in Lemma~\ref{lemma_morsification}.
\end{remark}

\subsection{Lyapunov Graphs for Singular Morse-Smale Flows} 
\label{subsection_graph}

Before defining the Lyapunov graph, we recall basic graph-theoretic notions.

A \textit{directed graph} is a triple $G = (V, E, \psi)$ where $V$ and $E$ are finite sets and $\psi$ is a map from $E$ to $V \times V$. The elements of $V$ are called \textit{vertices} and the elements of $E$ are called \textit{edges}. If $e$ is an edge and $ \psi(e)=(u, v)$, we say that $e$ is a \textit{directed edge} from $u$ to $v$, an \textit{outgoing edge} of $u$, and an \textit{incoming edge} of $v$. For simplicity, we write $e = (u, v)$ whenever $\psi(e) = (u, v)$. If $v$ is a vertex, then the number of incoming edges of $v$ is called the \textit{indegree} of $v$ and is denoted by $e^+(v)$; the number of outgoing edges of $v$ is called the \textit{outdegree} of $v$ and is denoted by $e^-(v)$; and the sum of the indegree and the outdegree of $v$ is called the \textit{degree} of $v$ and is denoted $e(v)$, that is, 
\[
e(v)\coloneqq e^+(v) + e^-(v).
\]

Throughout this paper, we use the term \textit{graph} to mean a directed graph.

Let $G_1 = (E_1, V_1, \psi_1)$ and $G_2 = (E_2, V_2, \psi_2)$ be graphs. A \textit{graph map}, denoted by $\Phi\colon G_1\to G_2$, is a pair of maps $\Phi_E\colon E_1\to E_2$ and $\Phi_V\colon V_1\to V_2$ that commutes appropriately with the maps $\psi_1$ and $\psi_2$; that is, for any $e\in E_1$ with $\psi_1(e)=(u,v)$, we have 
\[
    \psi_2\circ \Phi_E (e) = (\Phi_V(u), \Phi_V(v)).
\]
The graphs $G_1$ and $G_2$ are \textit{isomorphic} if there exists a graph map $\Phi\colon G_1\to G_2$ such that $\Phi_E$ and $\Phi_V$ are isomorphisms.

\begin{definition}
\label{definition_Lyapunov_graph}
Let $X$ be a pseudomanifold with spherical-cone singularities, and let $\varphi$ be a singular Morse-Smale flow on $X$. If $f \colon X \to \mathbb{R}$ is a Lyapunov function for $\varphi$, the \textit{Lyapunov graph} $\Gamma_f$ is defined as follows:
\begin{enumerate}[leftmargin=*, label=(\roman*)]
    \item  Define the equivalence relation $\sim_f$ on $X$ by declaring $x \sim_f y$ if and only if $x$ and $y$ belong to the same connected component of a level set of $f$.
    
    \item The graph $\Gamma_f$ is the quotient space $X/{\sim_f}$ endowed with a graph structure: a point $[x] \in X/{\sim_f}$ is a \textit{vertex} of $\Gamma_f$ if and only if the connected component of level set $f^{-1}(f(x))$ contains an element $\sigma \in \mathcal{R}(\varphi)$;  all other points of $X/{\sim_f}$ are  \textit{edges} points.
    
    \item Since $f$ is strictly decreasing along regular trajectories, the graph $\Gamma_f$ is naturally oriented.
\end{enumerate}
\end{definition}

It is usual to enrich the Lyapunov graph by labeling the vertices of the Lyapunov graph with data associated to the critical elements $\sigma\in \mathcal{R}(\varphi)$ (e.g., Conley index, type of singularity, etc.).

In the smooth setting, if the Lyapunov function has finitely many critical values, the quotient space from Definition~\ref{definition_Lyapunov_graph} is a finite directed graph; this result follows from~\cite[Proposition 1.4]{franks1985nonsingular}.  A similar result holds in our singular setting, and the proof follows the same steps as that of~\cite[Proposition 1.4]{franks1985nonsingular}; we therefore omit it.

\begin{proposition}
If $f \colon X \to \mathbb{R}$ is a Lyapunov function with respect to a singular Morse-Smale flow $\varphi$ on a closed $n$-pseudomanifold with spherical-cone singularities $X$, then the Lyapunov graph $\Gamma_{f}$ associated to $f$ is a finite directed graph.
\end{proposition}

\begin{remark}
    More generally, for continuous functions $f\colon X_1\to X_2$, where $X_1$ and $X_2$ are arbitrary topological spaces, the same equivalence relation yields the quotient space $R_f$, called the Reeb space. When $X_2=\es^1$ or $X_2=\R$, the quotient space $R_f$ is referred to as the Reeb graph. In more general settings, there exists an upper bound for the topological dimension of a Reeb space and a Reeb graph for a wide class of topological spaces and maps; see~\cite{GELBUKH2025109462}.
\end{remark}

\begin{example}
\label{example_graph_sphere}
Let $(X, \varphi)$ be the $n$-pseudomanifold with spherical-cone singularities endowed with the singular Morse-Smale flow described in Example~\ref{example_spheres}, for $k=8$ and $\eta_0=\eta_n=4$, where $n$ is arbitrary. By Theorem~\ref{theorem_lyapunov_global}, there exists a Lyapunov function $f$ associated with $(X, \varphi)$. Thus, the associated Lyapunov graph $\Gamma_{f}$ is illustrated in Figure~\ref{figura_grafo_8_folhas}, where the vertex $v\in \Gamma_f$ corresponding to the singularity $p\in \operatorname{M}(\mathcal{C}_8)$ of nature $a^4r^4$  and satisfies  $e^-(v)=e^+(v)=4$. \hfill $\diamond$

\begin{figure}[!ht]
\centering
\begin{overpic}[width=3cm]{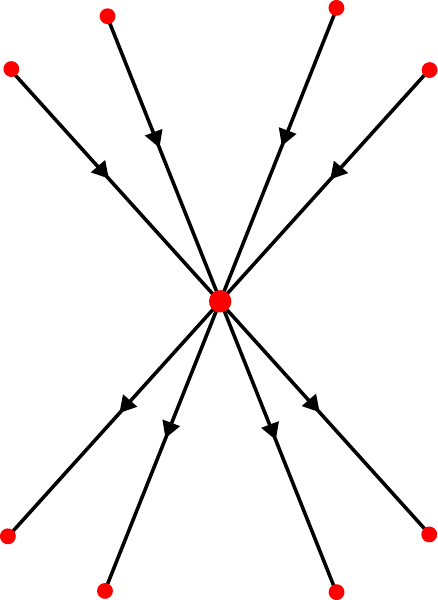}
\put(43, 46.5){$v$}
\put(-10,92){$\Gamma_{f}$}
\end{overpic}
\caption{Lyapunov graph for the flow in Example~\ref{example_spheres}, when  $k=8$ and $\eta_0=\eta_n=4$.}
\label{figura_grafo_8_folhas}
\end{figure}
\end{example}

Let $X$ be a closed $n$-pseudomanifold with spherical-cone singularities, equipped with a singular Morse-Smale flow $\varphi$, and let $f\colon X\to \mathbb{R}$ be a Lyapunov function. Denote by $\pi_{f} \colon X \to \Gamma_{f}$ the quotient map associated with the equivalence relation defining the Lyapunov graph of $f$. Furthermore, let $(\widetilde{X}, \widetilde{\varphi})$ be the morsification given by Lemma~\ref{lemma_morsification}, and let $\widetilde{f}\colon \widetilde{X}\to \mathbb{R}$ be a Lyapunov function for $\widetilde{\varphi}$. Denote by $\pi_{\widetilde{f}}\colon \widetilde{X} \to \Gamma_{\widetilde{f}}$ the corresponding quotient map  defining the Lyapunov graph of $\widetilde{f}$. With this notation in place, we state the following theorem:

\begin{theorem}
\label{theorem_graph_map}
Let $f \colon X \to \mathbb{R}$ be a Lyapunov function with respect to a singular Morse-Smale flow $\varphi$ on a closed $n$-pseudomanifold $X$ with spherical-cone singularities, and let $(\widetilde{X}, \widetilde{\varphi})$ be the morsification given by Lemma~\ref{lemma_morsification}. Then there exists a graph map $\Phi \colon \Gamma_{\widetilde{f}}\to \Gamma_{f}$ such that the following diagram commutes:
\begin{equation} 
\label{diagrama_grafos}
    \begin{tikzcd}
        \widetilde{X} \arrow[r, "\mathfrak{p}"] \arrow[d, "\pi_{\widetilde{f}}"'] 
        & X \arrow[d, "\pi_{f}"] \\
        \Gamma_{\widetilde{f}} \arrow[r, "\Phi"'] 
        & \Gamma_{f}
    \end{tikzcd}
\end{equation}
where $\mathfrak{p} \colon \widetilde{X} \to X$ is the projection given by Lemma~\ref{lemma_morsification}. Moreover, $\Phi$ satisfies the following properties:
\begin{enumerate}[label=(\alph*), font=\itshape, leftmargin=*, align=left]
    \item $\Phi_E$ is an orientation-preserving bijection between the edge sets of $\Gamma_{\widetilde{f}}$ and $\Gamma_{f}$;
    
    \item $\Phi_V$ is a surjective map between the vertex sets that induces the following relation on their cardinalities
    \[
    \#V_{\Gamma_{f}}=\#V_{\Gamma_{\widetilde{f}}} +\sum_{k \geq 2}(1-k)\#\operatorname{M}(\mathcal{C}_k),
    \]
    where $\#V_{\Gamma_{\widetilde{f}}}$ and $\#V_{\Gamma_{f}}$ denote the number of vertices of the graphs $\Gamma_{\widetilde{f}}$ and $\Gamma_{f}$, respectively, and $\#\operatorname{M}(\mathcal{C}_k)$ is the cardinality of the finite set $\operatorname{M}(\mathcal{C}_k)$;
    
    \item If $p\in \operatorname{M}(\mathcal{C}_k)$, then
    \[
    e^{-}(v_p) = \sum_{\widetilde{p} \in \mathfrak{p}^{-1}(p)} e^{-}(v_{\widetilde{p} }) \quad \text{and} \quad  e^{+}(v_p) = \sum_{\widetilde{p} \in \mathfrak{p}^{-1}(p)} e^{+}(v_{\widetilde{p} }). 
    \]
\end{enumerate}
\end{theorem}

\begin{proof}
Let $p\in \operatorname{M}(\mathcal{C}_k)\subset \operatorname{Sing}(X)$ be an arbitrary spherical-cone singularity of nature described by the word
\[
a^{\eta_0}s_1^{\eta_1}\dots s_{n-1}^{\eta_{n-1}}r^{\eta_n}.
\]
By the construction in Lemma~\ref{lemma_morsification}, we obtain a projection $\mathfrak{p}\colon \widetilde{X} \to X$ such that
\[
\mathfrak{p}^{-1}(p) = \{ \widetilde{p}^1_0, \dots , \widetilde{p}^{\eta_0}_0, \dots , \widetilde{p}^1_{\lambda}, \dots, \widetilde{p}^{\eta_\lambda}_{\lambda}, \dots, \widetilde{p}^1_n, \dots , \widetilde{p}^{\eta_n}_n \}.
\]
Since each $\widetilde{p}\in \mathfrak{p}^{-1}(p)$ is a singularity of $\widetilde{\varphi}$ and $p$ is a singularity of $\varphi$, it follows from the definition of Lyapunov graphs that $\pi_{\widetilde{f}}(\widetilde{p}) = v_{\widetilde{p}}$ and $\pi_{f}(p) = v_p$, where $v_{\widetilde{p}}$ and $v_p$ are vertices of $\Gamma_{\widetilde{f}}$ and $\Gamma_f$, respectively.

We define the vertex map $\Phi_V\colon V_{\Gamma_{\widetilde{f}}}\to V_{\Gamma_f}$ as follows.  For each  $p \in \operatorname{Sing}(X)$ and each  $\widetilde{p}\in \mathfrak{p}^{-1}(p)$, we set $\Phi_V(\pi_{\widetilde{f}}(\widetilde{p}))=v_p$, where $v_p$ is the vertex of $\Gamma_f$ associated to $p$. Since  no singularities or periodic orbits in the regular part of $\widetilde{X}$ are collapsed by $\mathfrak{p}$, we define $\Phi_V$ to act  as the identity on all remaining vertices. This concludes the definition of $\Phi_V$.

To define the edge map $\Phi_E$, it suffices to note that the restriction
\[
\mathfrak{p}\colon \widetilde{X}\setminus \mathcal{R}(\widetilde{\varphi}) \to X \setminus \mathcal{R}(\varphi)
\]
is a diffeomorphism, where $\mathcal{R}(\widetilde{\varphi})$ and $\mathcal{R}(\varphi)$ denote the sets of critical elements of $\widetilde{\varphi}$ and $\varphi$, respectively. Thus, we define $\Phi_E$ as the bijection between the edge sets that preserves orientation. This proves part $(a)$ of the theorem.

To ensure the commutativity of the diagram in \eqref{diagrama_grafos}, we need only to check it for points $\widetilde{p}\in \mathfrak{p}^{-1}(p)$, where $p\in \operatorname{Sing}(X)$. Note that
\[
\pi_f(\mathfrak{p}(\widetilde{p})) = \pi_f(p) = v_p = \Phi_V(\pi_{\widetilde{f}}(\widetilde{p})) = \Phi(\pi_{\widetilde{f}}(\widetilde{p})).
\]
Therefore, $\pi_f \circ \mathfrak{p} = \Phi \circ \pi_{\widetilde{f}}$, and the diagram commutes.

To prove item $(b)$, note that for any $p\in \operatorname{M}(\mathcal{C}_k)$, the map $\Phi_V$ collapses $k$ vertices of $\Gamma_{\widetilde{f}}$ to a single vertex of $\Gamma_f$. Since this observation holds for every such singularity, the following equality holds
\[
\#V_{\Gamma_f} = \#V_{\Gamma_{\widetilde{f}}} + \sum_{k \geq 2}(1-k)\#\operatorname{M}(\mathcal{C}_k).
\]
From $(a)$ and $(b)$, it follows immediately the equalities in item $(c)$.

Therefore, we have proved the existence of the graph map $\Phi\colon \Gamma_{\widetilde{f}}\to \Gamma_f$ satisfying the desired properties of the theorem.
\end{proof}

For intricate singular Morse–Smale flows on pseudomanifolds, the Lyapunov graph $\Gamma_f$ can be computed via  a morsification together with the graph map $\Phi \colon \Gamma_{\widetilde{f}} \to \Gamma_f$. We follow this approach in the next example.

\begin{example}
Consider the pseudomanifold with spherical-cone singularities $(X, \varphi)$ and the singular Morse-Smale flow $\varphi$ described in Example~\ref{example_torus} for the case where $n=2$, $m_1=2$, and $m_2=1$. Figure~\ref{figura_grafo_pojecao} shows a morsification $(\widetilde{X}, \widetilde{\varphi})$ for $(X, \varphi)$ and the corresponding Lyapunov graphs  $\Gamma_{\tilde{f}}$ and $\Gamma_{f}$. 
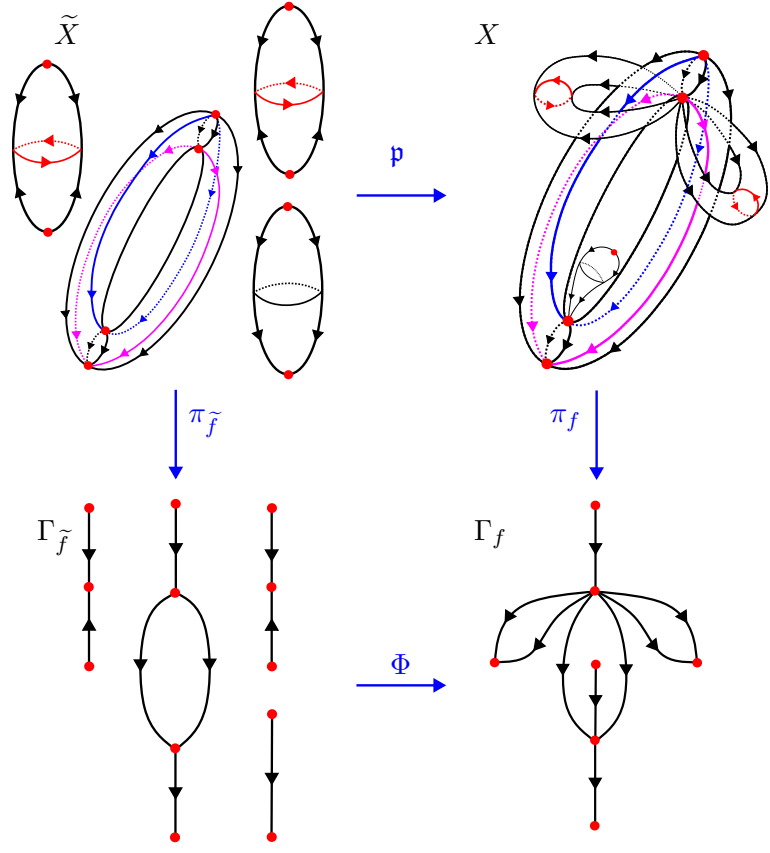
\begin{figure}[!ht]
\centering
\begin{overpic}[width=10cm]{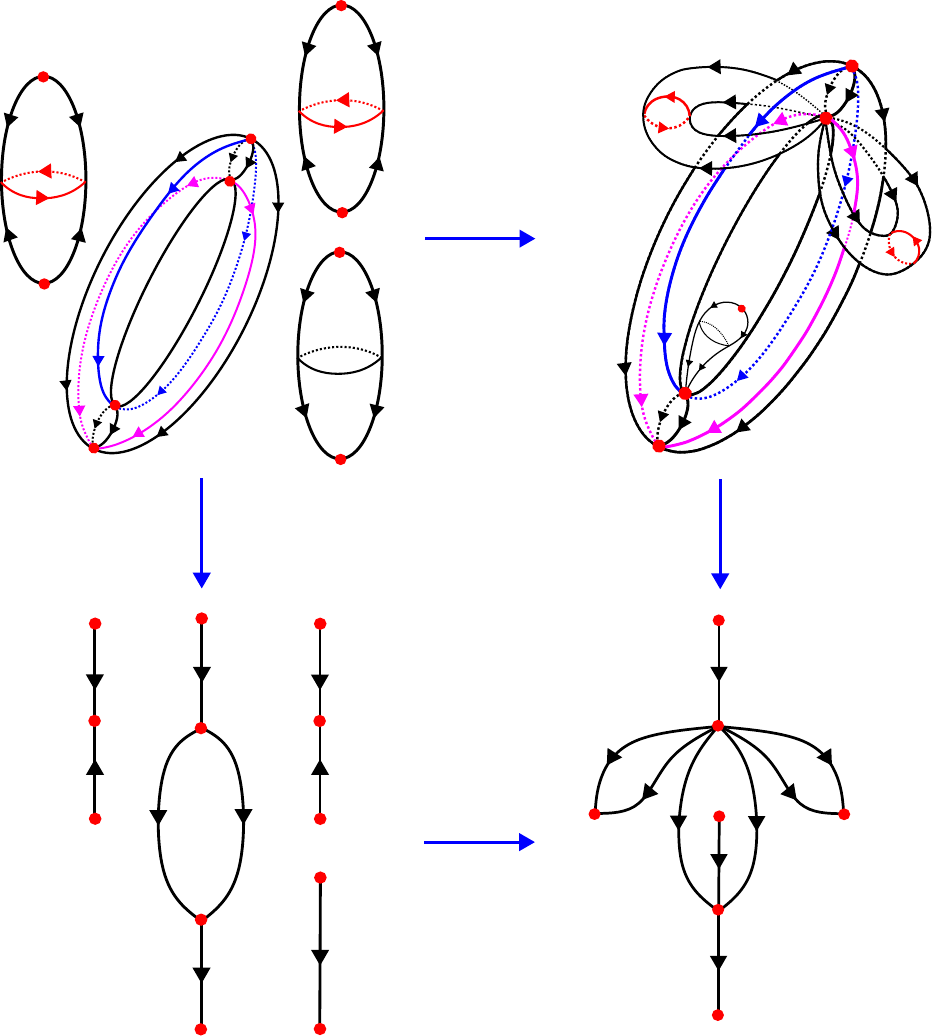}
\put(55,36){$\Gamma_{f}$}
\put(3, 36){$\Gamma_{\widetilde{f}}$}
\put(5,95){$\widetilde{X}$}
\put(55,95){$X$}
\put(45,81){\textcolor{blue}{$\mathfrak{p}$}}
\put(21,50){\textcolor{blue}{$\pi_{\widetilde{f}}$}}
\put(64,50){\textcolor{blue}{$\pi_{f}$}}
\put(45,20){\textcolor{blue}{$\Phi$}}
\end{overpic}
\caption{Obtaining the graph $\Gamma_{f}$ through the $\Phi$.}
\label{figura_grafo_pojecao}
\end{figure}
\hfill $\diamond$
\end{example}

Theorem~\ref{theorem_graph_map}, together with the local behavior of singular Morse-Smale flows, can be used to obtain estimates for the indegree and outdegree of vertices associated with spherical-cone singularities. More precisely, we have the following result.

\begin{theorem}
\label{theorem_degree}
Let $X$ be an $n$-pseudomanifold with spherical-cone singularities endowed with a singular Morse-Smale flow  $\varphi$, and let $\Gamma_{f}$ be a Lyapunov graph associated with $\varphi$. For each vertex $v$ of $\Gamma_f$ corresponding to a spherical-cone singularity $p\in \operatorname{M}(\mathcal{C}_k)$ of nature $a^{\eta_0} s_1^{\eta_{1}} \cdots s_{n-1}^{\eta_{n-1}} r^{\eta_n}$, the following local conditions hold:
\begin{equation}\label{eq_desigualdade_degree}
\begin{aligned}
    (k-\eta_0) - \eta_1 & \ \leq \ e^{-}(v)  \ \leq \ (k-\eta_0) + \eta_1, \\
    (k-\eta_n) - \eta_{n-1} & \ \leq \ e^{+}(v) \ \leq \ (k-\eta_n) + \eta_{n-1}.
\end{aligned}
\end{equation}
\end{theorem}

\begin{proof}
Let $p \in \operatorname{M}(\mathcal{C}_k)$ be a spherical-cone singularity of nature describe by the word $a^{\eta_0} s_1^{\eta_1} \dots s_{n-1}^{\eta_{n-1}} r^{\eta_n}$. Consider the projection $\mathfrak{p} \colon \widetilde{X} \to X$ given in Lemma~\ref{lemma_morsification}, and the fiber over $p$
\[
S \coloneqq \mathfrak{p}^{-1}(p) = \{\widetilde{p}^1_0, \dots, \widetilde{p}^{\eta_0}_0, \dots, \widetilde{p}^1_{\lambda}, \dots, \widetilde{p}^{\eta_\lambda}_{\lambda}, \dots, \widetilde{p}^1_n, \dots, \widetilde{p}^{\eta_n}_n\},
\]
where each $\widetilde{p}^i_{\lambda}$ is a singularity of index $\lambda$ of $\widetilde{\varphi}$. By \cite[Corollary 3.3]{cruz1999}, the indegrees of the vertices of the Lyapunov graph associated with $(\widetilde{X}, \widetilde{\varphi})$ satisfy the following relations
\[
e^-(v_{\widetilde{p}^i_0}) = 0, \quad 0 \leq e^-(v_{\widetilde{p}^i_1}) \leq 2,  \quad e^-(v_{\widetilde{p}^i_{\lambda}}) = 1, \quad  \text{for all } i \text{ and }  \lambda \in \{2, \dots, n\}.
\]
By summing these relations over the corresponding sets of critical points in $S$, we obtain
\[
\sum_{i=1}^{\eta_0} e^-(v_{\widetilde{p}^i_0}) = 0, \quad 0 \leq \sum_{i=1}^{\eta_1} e^-(v_{\widetilde{p}^i_1}) \leq 2\eta_1, \quad \text{and} \quad \sum_{i=1}^{\eta_{\lambda}} e^-(v_{\widetilde{p}^i_{\lambda}}) = \eta_\lambda.
\]
Combining these expressions for all singularities in $S$, it yields
\[
\eta_2 + \dots + \eta_n \ \leq \ \sum_{\widetilde{p} \in S} e^-(v_{\widetilde{p}})  \ \leq \  2\eta_1 + \eta_2 + \dots + \eta_n.
\]
Since the sum of all indices $\eta_j$ equals $k$, we can rewrite these bounds in terms of $k$ as follows
\[
(k-\eta_0) - \eta_1 \ \leq \ \sum_{\widetilde{p} \in S} e^{-}(v_{\widetilde{p}}) \ \leq \ (k-\eta_0) + \eta_1.
\]
By item (c) of Theorem~\ref{theorem_graph_map}, we have $e^{-}(v) = \sum_{\widetilde{p} \in S} e^{-}(v_{\widetilde{p}})$, which leads to our first set of inequalities
\[
(k-\eta_0) - \eta_1 \ \leq \ e^{-}(v) \ \leq \ (k-\eta_0) + \eta_1.
\]

Similarly, applying \cite[Corollary 3.3]{cruz1999} to the outdegrees, we find
$$ e^+(v_{\widetilde{p}^i_n}) = 0, \quad 0 \leq e^+(v_{\widetilde{p}^i_{n-1}}) \leq 2, \quad  e^+(v_{\widetilde{p}^i_{\lambda}}) = 1, \quad \text{for all} \ \ \lambda \in \{0, \dots, n-2\}.$$
Summing these relations over $S$ provides
\[
\sum_{i=1}^{\eta_n} e^+(v_{\widetilde{p}^i_n}) = 0, \quad 0 \leq \sum_{i=1}^{\eta_{n-1}} e^+(v_{\widetilde{p}^i_{n-1}}) \leq 2\eta_{n-1}, \quad \text{and} \quad \sum_{i=1}^{\eta_{\lambda}} e^+(v_{\widetilde{p}^i_{\lambda}}) = \eta_\lambda.
\]
By summing all outdegrees in $S$, it leads to the inequality:
\[
\eta_0 + \dots + \eta_{n-2}  \ \leq \ \sum_{\widetilde{p} \in S} e^+(v_{\widetilde{p}}) \ \leq \ 2\eta_{n-1} + \eta_0 + \dots + \eta_{n-2}.
\]
Rewriting this expression in terms of $k$, we obtain
\[
(k-\eta_n) - \eta_{n-1} \ \leq \ \sum_{\widetilde{p} \in S} e^+(v_{\widetilde{p}}) \ \leq \ (k-\eta_n) + \eta_{n-1}.
\]
Finally, using item (c) of Theorem~\ref{theorem_graph_map} again, we conclude that
\[
(k-\eta_n) - \eta_{n-1}  \ \leq \ e^+(v) \ \leq \ (k-\eta_n) + \eta_{n-1}.
\]
This completes the proof.
\end{proof}

The next two corollaries follow directly from the inequalities in~\eqref{eq_desigualdade_degree}. 
Their proofs consist in evaluating these bounds under the specific hypotheses of each statement and using the equality $e(v) = e^{-}(v) + e^{+}(v)$.

\begin{corollary}
\label{corollary_degree_specific_nature}
Let $X$ be an $n$-pseudomanifold with spherical-cone singularities endowed with a singular Morse-Smale flow  $\varphi$, and let $\Gamma_{f}$ be a Lyapunov graph associated with $\varphi$. Then for any vertex $v$ of $\Gamma_f$ corresponding to a spherical-cone singularity $p \in \operatorname{M}(\mathcal{C}_k)$, the following local conditions hold:
\begin{enumerate}[label=(\alph*), font=\itshape, leftmargin=*, align=left]
    \item If $p$ is of nature $a^k$, then $e^-(v)=0$ and $e^+(v)=k$.
    \item If $p$ is of nature $s_1^k$, then $0\leq e^-(v) \leq 2k$ and $e^+(v)=k$.
    \item If $p$ is of nature $s_{\lambda}^k$ with $\lambda\in \{2,\dots, n-2\}$, then $e^-(v) = k$ and $e^+(v) = k$.
    \item If $p$ is of nature $s_{n-1}^k$, then $e^-(v) = k$ and $0\leq e^+(v) \leq 2k$.
    \item If $p$ is of nature $r^k$, then $e^-(v)=k$ and $e^+(v)=0$.
    \item If $p$ is of nature $a^{\eta_0} s_1^{\eta_{1}} \cdots s_{n-1}^{\eta_{n-1}} r^{\eta_n}$ with $\eta_1=\eta_{n-1}=0$, then $e^-(v)=k-\eta_0$ and $e^+(v)=k-\eta_n$.
\end{enumerate}
\end{corollary}

\begin{corollary}
\label{corollary_degree}
Let $X$ be an $n$-pseudomanifold with spherical-cone singularities, endowed with a singular Morse-Smale flow $\varphi$. A Lyapunov graph $\Gamma_{f}$ associated with $\varphi$ satisfies the following local conditions for a vertex $v$ associated with any spherical-cone singularity $p \in \operatorname{M}(\mathcal{C}_k)$: 
\begin{enumerate}[label=(\alph*), font=\itshape, leftmargin=*, align=left]
    \item If $p$ is of nature $a^{\eta_0} s_1^{\eta_{1}} \cdots s_{n-1}^{\eta_{n-1}} r^{\eta_n}$, then
    \[
    2k-(\eta_0+\eta_1+\eta_{n-1}+\eta_n)  \ \leq \ e(v) \ \leq \ 2k-(\eta_0-\eta_1-\eta_{n-1}+\eta_n). 
    \]
    \item If $p$ is of nature $a^{\eta_0} s_1^{\eta_{1}} \cdots s_{n-1}^{\eta_{n-1}} r^{\eta_n}$ with $\eta_1=\eta_{n-1}=0$, then 
    \(
    e(v)=2k-(\eta_0+\eta_n). 
    \)
\end{enumerate}
\end{corollary}

\section{Euler-Poincaré Characteristic for Pseudomanifolds with Spherical-Cone Singularities}
\label{section_characteristic}

The study of the Euler-Poincaré characteristic dates back to Euler in the 18th century, who established the celebrated formula
\[
\chi(K)=V-E+F,
\]
where $K\subset \mathbb{R}^3$ is a polyhedron, $V$ denotes the number of vertices, $E$ the number of edges, and $F$ the number of faces. In the 19th century, Poincaré generalized this notion to higher-dimensional spaces, showing that
\[
\chi(T) = \sum_{i=0}^n (-1)^i \beta_i(T),
\]
where $T$ is an $n$-dimensional triangulable topological space and $\beta_i(T)$ denotes its$i$-th Betti number. 

The Poincaré-Hopf Theorem represents the first instance where the Euler-Poincaré characteristic is described in terms of the indices of singularities of a vector field. It was first proved by Henri Poincaré in dimension two \cite{poincare1885memoire}, and later extended to higher dimensions by Heinz Hopf \cite{hopf1927vektorfelder}. The theorem can be stated as follows.

Let $M$ be a compact manifold with boundary $\partial M$, and let $V$ be a continuous vector field tangent to $M$ with isolated singularities. Denote by $p_i \in \operatorname{Sing}(V)$ the singularities of $V$ and by $I(V, p_i)$ their respective Poincaré-Hopf indices. If $V$  points outward along $\partial M$, then 
\[
    \chi(M) = \sum_{p_i \in \operatorname{Sing}(V)} I(V, p_i).
\]
If $V$ points inward along $\partial M$, then
\[
    \chi(M) - \chi(\partial M) = \sum_{p_i \in \operatorname{Sing}(V)} I(V, p_i).
\]
A fundamental consequence of the Poincaré-Hopf Theorem is that the sum of the indices of a vector field tangent to a compact manifold, with isolated singularities, does not depend on the chosen vector field.

More conceptually, the theorem shows that the Euler–Poincaré characteristic measures the obstruction to the existence of a nowhere-vanishing tangent vector field on a compact oriented manifold. In this sense, it can be interpreted as one of the earliest characteristic classes. For further details, see \cite[Chapter~2]{brasselet2021introduction}.

Conley theory once again proves to be a powerful tool for computing the Euler–Poincaré characteristic of smooth manifolds in terms of the ranks of the homology Conley index of the critical elements of a Morse–Smale flow on the space.

In this section,  Conley theory is used to derive alternative formulas for the Euler-Poincaré characteristic of an $n$-pseudomanifold with spherical-cone singularities in terms of dynamical data arising from a singular Morse-Smale flow on $X$. 

\subsection{Alternative Formula for the Euler-Poincaré Characteristic}
\label{subsection_alternative_characteristic}

In this subsection, we employ the tools developed in Section~\ref{section_conley} to establish the alternative formula presented in Theorem~\ref{theorem_characteristic_conley}. This extends the results of~\cite{Lima2025, Montufar} by showing that, whereas those previous works are restricted to dimension~$2$, Theorem~\ref{theorem_characteristic_conley} holds in arbitrary dimensions.

\begin{theorem}
\label{theorem_characteristic_conley}
Let $X$ be a closed $n$-pseudomanifold with spherical-cone singularities endowed with a singular Morse–Smale flow $\varphi$. Then
\[
    \chi(X)=\sum_{p \in \operatorname{Sing}(\varphi)} \sum_{j=0}^{n}(-1)^j h_j(p), 
\]
where $h_{\ast}(p)=\left( h_0(p), \dots, h_n(p) \right)$ denotes the numerical Conley indices of the singularity $p$.
\end{theorem}

\begin{proof}
Let $\mathcal{R}(\varphi)=\{ \sigma_1, \ldots, \sigma_{m_0}\}$, where each $\sigma_i$ is a critical element of $\varphi$, i.e. a singularity or a periodic orbit. Consider the sequence of sets $(G_i)_{i=0}^{m_0}$ given by Lemma~\ref{Lema_sequencia}, satisfying
\(
    G_0 \subset G_1\subset \ldots \subset G_{m_0}=X.
\)
For each $1\leq i\leq m_0$, $(G_i, G_{i-1})$ is an index pair for $\sigma_i$. The long exact sequence of the pair $(G_i,G_{i-1})$ is given by
\begin{equation*}
\begin{gathered}
\begin{tikzcd}[column sep=1.5em, row sep=1.5em, font=\normalsize]
\cdots \arrow{r}{\pi_{j}}  & H_{j}(G_i, G_{i-1}) \arrow{r}{\partial_j} & H_{j-1}(G_{i-1}) \arrow{r}{ f_{j-1}} &  H_{j-1}(G_{i})  \arrow{r}{\pi_{j-1}} & H_{j-1}(G_i,G_{i-1}) \arrow{r}{\partial_{j-1}} & \cdots
\end{tikzcd}
\end{gathered}
\end{equation*}
where, $f_{j-1}$ is the homomorphism induced by the inclusion map $f\colon  G_{i-1} \to G_i$, $\pi_{j}$ is induced by the projection $\pi\colon  (G_i, \varnothing) \to (G_i, G_{i-1})$, and $\partial_j$ is the connecting homomorphism.

From the exactness of this sequence and the rank–nullity formula\footnote{In this case, all homology groups are finitely generated abelian groups; therefore, the rank version of the rank–nullity theorem applies.}, we have
\begin{align}
\label{equation_proof_euler_1}
    \operatorname{rank}( \Ima (\pi_j))&=\operatorname{rank} (\Ker (\partial_j))\nonumber\\
    &= \operatorname{rank}( H_j(G_i,G_{i-1})) - \operatorname{rank}( \Ima (\partial_j) )\nonumber\\
    &= \operatorname{rank}( H_j(G_i,G_{i-1})) - \operatorname{rank}( \Ker ( f_{j-1}) )
\end{align}
and
\begin{align}
\label{equation_proof_euler_2}
    \operatorname{rank}( \Ima (\pi_{j-1}))&= -\operatorname{rank} (\Ker (\pi_{j-1})) + \operatorname{rank} (H_{j-1}(G_i))\nonumber\\
    &= -\operatorname{rank} (\Ima (f_{j-1})) + \operatorname{rank} (H_{j-1}(G_i)).
\end{align}
Summing \eqref{equation_proof_euler_1} and \eqref{equation_proof_euler_2}, we obtain 
\begin{align*}
    \operatorname{rank}( \Ima (\pi_j)) + \operatorname{rank}( \Ima (\pi_{j-1})) =& \operatorname{rank}( H_j(G_i,G_{i-1})) - \operatorname{rank}( \Ker ( f_{j-1}) )  -\operatorname{rank} (\Ima (f_{j-1})) + \operatorname{rank} (H_{j-1}(G_i))\\
    =& \operatorname{rank}( H_j(G_i,G_{i-1})) + \operatorname{rank} (H_{j-1}(G_i)) -\left[ \operatorname{rank}( \Ker ( f_{j-1}) ) + \operatorname{rank} (\Ima (f_{j-1})) \right] \\
    =& \operatorname{rank}(H_j(G_i,G_{i-1})) + \operatorname{rank} (H_{j-1}(G_i))-  \operatorname{rank} (H_{j-1}(G_{j-1})).
\end{align*}

Since $CH_{\ast}(\sigma_i)\cong H_{\ast} (G_i, G_{i-1})$,  it follows that
\( 
    \operatorname{rank} (H_j(G_i, G_{i-1}))= h_j( \sigma_i).
\)
Thus,
\[
    \operatorname{rank}( \Ima (\pi_j)) + \operatorname{rank}( \Ima (\pi_{j-1})) = h_j(\sigma_i) + \beta_{j-1}(G_i) - \beta_{j-1}(G_{i-1}).
\]
Fixing $i$ and taking the alternating sum over $j$, we obtain
\[
\sum_{j}(-1)^j\Big( \operatorname{rank}( \Ima (\pi_j)) + \operatorname{rank}( \Ima (\pi_{j-1}))\Big) = \sum_{j}(-1)^j h_j(\sigma_i)+ \sum_{j}(-1)^j\Big( \beta_{j-1}(G_i) - \beta_{j-1}(G_{i-1})\Big).   
\]
The left-hand side of the equation is a telescoping sum that vanishes. Consequently,
\[
    \sum_{j=0}^n(-1)^j h_j(\sigma_i) + \sum_{j=1}^{n+1}(-1)^j\Big( \beta_{j-1}(G_i) - \beta_{j-1}(G_{i-1})\Big) =0.
\]
Now, summing the last equation over $i$, we obtain
\begin{align*}
    0 &= \sum_{i=1}^{m_0}\sum_{j=0}^n(-1)^j h_j(\sigma_i) + \sum_{i=1}^{m_0}\sum_{j=1}^{n+1}(-1)^j\left( \beta_{j-1}(G_i) - \beta_{j-1}(G_{i-1})\right)\\
    &= \sum_{i=1}^{m_0}\sum_{j=0}^n(-1)^j h_j(\sigma_i) + \sum_{j=1}^{n+1}(-1)^j\left( \beta_{j-1}(G_{m_0}) \right).
\end{align*}
Since $G_{m_0}=X$, it follows the equality
$$ \sum_{j=0}^{n}(-1)^j\left( \beta_{j}(X) \right) \ = \  \sum_{i=1}^{m_0}\sum_{j=0}^n(-1)^j h_j(\sigma_i).$$

Therefore, 
\begin{align*}
    \chi(X) \ =& \ \sum_{i=1}^{m_0}\sum_{j=0}^n(-1)^j h_j(\sigma_i)\ = \sum_{\sigma \in \mathcal{R}(\varphi)}\sum_{j=0}^n(-1)^j h_j(\sigma).
\end{align*}
Note that if $\sigma=\gamma$ is a periodic orbit of index $\lambda$, for $\lambda=0,\dots, n$,  Example~\ref{example_conley_smooth} implies that 
\[
    \sum_{i=0}^n h_i(\gamma)= (-1)^{\lambda} + (-1)^{\lambda+1}=0,
\] 
i.e., its contribution to $\chi(X)$ vanishes.
Therefore, we can rewrite the result only for the singularities of $\varphi$, hence
\[
    \chi (X)\ = \sum_{p\in \operatorname{Sing}(\varphi)}\sum_{j=0}^n(-1)^j h_j(p).
\]
\end{proof}

\begin{remark}
\label{remark_caractersitca_com_bordo}
Theorem~\ref{theorem_characteristic_conley} also holds when $X$ is compact; in this case, letting $\partial X^- \subset \partial M$ denote the exit set of the flow $\varphi$, the following equality holds
\[
    \chi(X, \partial X^-) = \sum_{p\in \operatorname{Sing}(\varphi)}\sum_{j=0}^n(-1)^j h_j(p),
\]
where $\chi(X, \partial X^-)$ denotes the relative Euler characteristic of the pair $(X, \partial X^-)$. This is analogous to the case of smooth manifolds; see, for example, \cite[Proposition 2.1]{cruz1999}.
\end{remark}

\subsection{Euler–Poincaré Characteristic via Morsification}
\label{subsection_morsification_characteristic}
The morsification process alters the topology of $X$ in a controlled manner. The following theorem quantifies this change by relating the Euler–Poincaré characteristic of a pseudomanifold  $X$ to that of its morsification $\widetilde{X}$. In particular, the difference between these invariants is determined entirely by the spherical-cone singularities of $X$.

\begin{theorem}
\label{theorem_characteristic_morsification}
Let $X$ be a closed $n$-pseudomanifold with spherical-cone singularities endowed with a singular Morse-Smale flow $\varphi$. Then 
\[
    \chi(\widetilde{X}) = \chi(X) + \sum_{k \geq 2} (k-1)  \#\operatorname{M}(\mathcal{C}_k),
\]
where $\widetilde{X}$ denotes the morsification of $X$ and $\#\operatorname{M}(\mathcal{C}_k)$ is the cardinality of the finite set $\operatorname{M}(\mathcal{C}_k)$.
\end{theorem}

\begin{proof}
By Theorem~\ref{theorem_characteristic_conley},  the Euler-Poincaré characteristic of $X$ can be expressed in terms of the numerical Conley indices of the singularities of $\varphi$. 
We first analyze how the morsification affects the numerical Conley index of each spherical-cone singularity. To this end, let $(\widetilde{X}, \widetilde{\varphi}, \mathfrak{h}, \mathfrak{p})$ be a morsification of $(X,\varphi)$, as presented in Lemma~\ref{lemma_morsification}.
\begin{enumerate}[font=\itshape, leftmargin=*, align=left]
    \item Let $p \in \operatorname{M}(\mathcal{C}_k)$ be a singularity with nature $a^k$, where $k \geq 2$ is an integer. According to  Lemma~\ref{lemma_morsification},
    \(
        \mathfrak{h}(p)=\{ \widetilde{p}^1_0, \dots, \widetilde{p}^k_0 \}.
    \)
    By Corollary~\ref{corollary_indice}, we have $h_0(p)=h_0(\widetilde{p}^i_0)=1$ for all $i \in \{1, \ldots, k\}$. Thus,
    \[
    h_0(p)=1=k-(k-1) =\sum_{i=1}^kh_0(\widetilde{p}^i_0) -(k-1)=\sum_{q\in \mathfrak{h}(p)}h_0(q) - (k-1).
    \]
    Therefore,
    \begin{align}
    \label{eq_cone_atrator_zero}
        h_0(p)=\sum_{q\in \mathfrak{h}(p)}h_0(q) -(k-1).
    \end{align}
    Furthermore, for $1 \leq \lambda \leq n$, by Corollary~\ref{corollary_indice}, we have $h_\lambda(p)=\sum_{i=1}^k h_\lambda(\widetilde{p}^i_0)=0$. Hence,
    \begin{align}
    \label{eq_cone_atrator_lamda}
        h_\lambda(p)=\sum_{q \in \mathfrak{h}(p)}h_\lambda(q).
    \end{align}
    
    \item Let $p \in \operatorname{M}(\mathcal{C}_k)$ be a singularity with nature $a^{\eta_0(p)} s_1^{\eta_1(p)} \ldots s_{n-1}^{\eta_{n-1}(p)} r^{\eta_n(p) }$ such that $0 \leq \eta_0(p) < k$. By the construction in Lemma~\ref{lemma_morsification}, 
    \[
        \mathfrak{h}(p)=\{ \widetilde{p}^1_0, \dots, \widetilde{p}^{\eta_0(p)}_0, \dots, \widetilde{p}^1_\lambda, \dots, \widetilde{p}^{\eta_\lambda(p)}_{\lambda}, \dots, \widetilde{p}^1_n, \dots, \widetilde{p}^{\eta_n(p)}_n \}.
    \]
    If $2 \leq \lambda \leq n$, then  $h_\lambda(p)=\sum_{i=1}^{\eta_\lambda(p) }h_\lambda(\widetilde{p}^i_\lambda)$, by Corollary~\ref{corollary_indice}. Hence, 
    \begin{align}
    \label{eq_cone_lambda}
    h_\lambda(p)=\sum_{q \in \mathfrak{h}(p)}h_\lambda(q).
    \end{align}
    If $\lambda=1$, by Corollary~\ref{corollary_indice}, we have
    \( \displaystyle
        h_1(p)=\eta_1(p) + (k- \eta_0(p) -1)= \sum_{i=1}^{\eta_1(p)}h_1(\widetilde{p}^i_1) + (k-\eta_0(p)-1).
    \)
    Then,
    \begin{align}
    \label{eq_cone_um}
        h_1(p)=\sum_{q\in \mathfrak{h}(p)}h_1(q) + (k-\eta_0(p)-1).
    \end{align}
    If $\lambda=0$,  by Corollary~\ref{corollary_indice}, we have
    \( \displaystyle
        h_0(p)= 0=\eta_0(p)-\eta_0(p) = \sum_{i=1}^{\eta_0(p)}h_0(\widetilde{p}^i_0) - \eta_0(p).
    \)
    Hence,
    \begin{align}
    \label{eq_cone_zero}
    h_0(p)=\sum_{q\in \mathfrak{h}(p)}h_0(q) - \eta_0(p).
    \end{align}
\end{enumerate}

Now, observe that
\[
\operatorname{Sing}(\varphi)= \operatorname{Sing}(X)\cup \left( \bigcup_{k = 0 }^{n}\operatorname{Crit}_k (\varphi) \right)= \left(  \bigcup_{k \geq 2 }\operatorname{M}(\mathcal{C}_k) \right) \cup \left( \bigcup_{k =0}^n \operatorname{Crit}_k (\varphi) \right).
\]
We define the  sets
\( 
    A_k \coloneqq \{ p \in \operatorname{M}(\mathcal{C}_k) \mid p \text{ is of nature } a^k \}\) and \( B_k \coloneqq \operatorname{M}(\mathcal{C}_k) \setminus A_k,
\)
along with their unions
\[
    A \coloneqq \bigcup_{k \geq 2} A_k, \qquad B \coloneqq \bigcup_{k \geq 2} B_k, \qquad \text{and} \qquad C \coloneqq \bigcup_{k \geq 0}\operatorname{Crit}_k (\varphi).
\]
Hence, $\operatorname{Sing}(\varphi)=A \cup B \cup C$. From the definitions of the above sets and Theorem~\ref{theorem_characteristic_conley}, we obtain
\begin{align}
\chi(X)&=\sum_{p \in \operatorname{Sing}(\varphi)}\sum_{j=0}^n (-1)^j h_j(p) \nonumber \\
\label{eq_new_01}
&=\sum_{p \in A}\sum_{j=0}^n (-1)^j h_j(p) + \sum_{p \in B}\sum_{j=0}^n (-1)^j h_j(p)  + \sum_{p \in C}\sum_{j=0}^n (-1)^j h_j(p).
\end{align}

We now expand each term in Equation~\eqref{eq_new_01}. First, for the set $A$, using the definitions of $A_k$ and $A$ together with the relations obtained in~\eqref{eq_cone_atrator_zero} and~\eqref{eq_cone_atrator_lamda}, we obtain the following:
\begin{align*}
    \sum_{p\in A}\sum_{j=0}^n (-1)^j h_j(p) &= \sum_{k\geq 2}\sum_{p\in A_k}\sum_{j=0}^n (-1)^j h_j(p)\\
    &= \sum_{k\geq 2}\sum_{p\in A_k}\left( h_0(p) +  \sum_{j=1}^n (-1)^j h_j(p)\right)\\
    &= \sum_{k\geq 2}\sum_{p\in A_k}\left( \sum_{q\in \mathfrak{h}(p)}\left( h_0(q) -(k-1) \right)  +  \sum_{q\in \mathfrak{h}(p)}\sum_{j=1}^n (-1)^j h_j(q)\right)\\
    &= \sum_{k\geq 2}\sum_{p\in A_k}\sum_{q\in \mathfrak{h}(p)}\sum_{j=1}^n (-1)^j h_j(q)  -\sum_{k\geq 2} \sum_{p\in A_k} (k-1). 
\end{align*}

Then, 
\begin{equation}
\label{eq_new_02}
    \sum_{p\in A}\sum_{j=0}^n (-1)^j h_j(p) =\sum_{p\in A}\sum_{q\in \mathfrak{h}(p)}\sum_{j=1}^n (-1)^j h_j(q)  -\sum_{k\geq 2} \sum_{p\in A_k} (k-1).
\end{equation}

Next, expanding the term corresponding to $B$ in Equation~\eqref{eq_new_01}, and using the definitions of the sets $B_k$ and $B$ alongside relations~\eqref{eq_cone_lambda}, \eqref{eq_cone_um}, and \eqref{eq_cone_zero}, we obtain
\begin{align*}
    \sum_{p\in B}\sum_{j=0}^n (-1)^jh_j(p)&= \sum_{k\geq 2}\sum_{p\in B_k} \left( h_0(p) - h_1(p) +  \sum_{j=2}^n h_j(p) \right)\\
    &= \sum_{k\geq 2}\sum_{p\in B_k}\left( \left[ \sum_{q\in \mathfrak{h}(p)}h_0(q) -\eta_0(p)\right] -\left[ \sum_{q\in \mathfrak{h}(p)}h_1(q) + (k - \eta_0(p)-1 )  \right] \right. \\
    &\qquad \left. + \sum_{q\in \mathfrak{h}(p)}\sum_{j=2}^n  (-1)^j h_j(q) \right)\\
    &=\sum_{k\geq 2}\sum_{p\in B_k}\sum_{q\in \mathfrak{h}(p)}\sum_{j=0}^n (-1)^j h_j(q)  -\sum_{k\geq 2}\sum_{p \in B_k}(\eta_0(p) + k -\eta_0(p)-1)\\
    &=\sum_{p\in B}\sum_{q\in \mathfrak{h}(p)}\sum_{j=0}^n (-1)^j h_j(q)   -\sum_{k\geq 2}\sum_{p \in B_k}(k-1). 
\end{align*}

Therefore, 
\begin{equation}
\label{eq_new_03}
    \sum_{p\in B}\sum_{j=0}^n (-1)^jh_j(p) = \sum_{p\in B}\sum_{q\in \mathfrak{h}(p)}\sum_{j=0}^n (-1)^j h_j(q)   -\sum_{k\geq 2}\sum_{p \in B_k}(k-1).
\end{equation}

Applying \eqref{eq_new_02} and \eqref{eq_new_03} to \eqref{eq_new_01} yields
\begin{align*}
    \chi(X) &= \sum_{p\in A}\sum_{q\in \mathfrak{h}(p)}\sum_{j=1}^n (-1)^j h_j(q) - \sum_{k\geq 2} \sum_{p\in A_k} (k-1)\\
    & \quad + \sum_{p\in B}\sum_{q\in \mathfrak{h}(p)}\sum_{j=0}^n (-1)^j h_j(q)   -\sum_{k\geq 2}\sum_{p \in B_k}(k-1) + \sum_{p \in C}\sum_{j=0}^n (-1)^j h_j(p)\\
    &=  \sum_{p\in A \cup B}\sum_{q\in \mathfrak{h}(p)}\sum_{j=0}^n (-1)^j h_j(q)   -\sum_{k\geq 2}\sum_{p \in A_k \cup B_k}(k-1) + \sum_{p \in C}\sum_{j=0}^n (-1)^j h_j(p)\\
    &=  \sum_{p\in \operatorname{Sing}(X)}\sum_{q\in \mathfrak{h}(p)}\sum_{j=0}^n (-1)^j h_j(q)   -\sum_{k\geq 2}\sum_{p \in \operatorname{M}(\mathcal{C}_k)}(k-1) + \sum_{p \in C}\sum_{j=0}^n (-1)^j h_j(p)\\
    &=  \sum_{p\in \operatorname{Sing}(X)}\sum_{q\in \mathfrak{h}(p)}\sum_{j=0}^n (-1)^j h_j(q)   -\sum_{k\geq 2}\# \operatorname{M}(\mathcal{C}_k) (k-1) + \sum_{p \in C}\sum_{j=0}^n (-1)^j h_j(p). 
\end{align*}

Therefore, 
\begin{equation}
\label{eq_new_04}
\chi(X)=\left( \sum_{p\in \operatorname{Sing}(X)}\sum_{q\in \mathfrak{h}(p)}\sum_{j=0}^n (-1)^j h_j(q) + \sum_{p \in C}\sum_{j=0}^n (-1)^j h_j(p) \right) -\sum_{k\geq 2}\# \operatorname{M}(\mathcal{C}_k) (k-1).  
\end{equation}

Note that
\[
    \operatorname{Sing}(\widetilde{\varphi})=\bigcup_{k \geq 0 } \operatorname{Crit}_k(\widetilde{\varphi})= \left( \bigcup_{p \in \operatorname{Sing}(X)}\mathfrak{h}(p) \right)\cup C.
\]
Applying this relationship between the sets and Theorem~\ref{theorem_characteristic_conley} to $\widetilde{X}$, we obtain 
\begin{align}
    \chi(\widetilde{X})& \ = \sum_{p \in \operatorname{Sing}(\widetilde{\varphi}) }\sum_{j=0}^n (-1)^j h_j(p) \ =   \sum_{p \in \operatorname{Sing}(X)} \sum_{q\in \mathfrak{h}(p) }\sum_{j=0}^n (-1)^j h_j(p) +  \sum_{p \in C}\sum_{j=0}^n (-1)^j h_j(p).    \label{eq_new_05} 
\end{align}

Substituting \eqref{eq_new_05} into \eqref{eq_new_04}, we obtain
\[
    \chi(X)=\chi(\widetilde{X}) -\sum_{k\geq 2}\# \operatorname{M}(\mathcal{C}_k) (k-1). 
\]

Therefore,
\[
    \chi(\widetilde{X})=\chi(X) + \sum_{k \geq 2 } (k-1)\cdot \# \operatorname{M}(\mathcal{C}_k). 
\]
\end{proof}

As a consequence of the equality in Theorem~\ref{theorem_characteristic_morsification}, we deduce the following result.

\begin{corollary}
\label{corollary_cara}
Let $X$ be a closed $n$-pseudomanifold with spherical-cone singularities endowed with a singular Morse-Smale flow $\varphi$. Then
\[
    \chi(X) = \sum_{i=0}^n (-1)^i c_i + \sum_{k \geq 2 } (1-k) \# \operatorname{M}(\mathcal{C}_k), 
\]
where $c_i  \ \coloneqq \displaystyle \!\! \sum_{p \in \operatorname{Sing}(\varphi)} \!\!\eta_{i}(p)$ for all $i=0, 1, \dots, n$, and $\#\operatorname{M}(\mathcal{C}_k)$ is the cardinality of the set $\operatorname{M}(\mathcal{C}_k)$. 
\end{corollary}

\begin{proof}
 Let $(\widetilde{X}, \widetilde{\varphi}, \mathfrak{h}, \mathfrak{p})$ be a morsification of $(X,\varphi)$, as presented in Lemma~\ref{lemma_morsification}.
By Theorem~\ref{theorem_characteristic_morsification}, the following equality holds
\begin{align}
\label{eq_corolario_1}
    \chi(\widetilde{X})=\chi (X) + \sum_{k\geq 2}(k-1)\# \operatorname{M}(\mathcal{C}_k).
\end{align}
Since $\widetilde{X}$ is a closed smooth manifold, the results of~\cite{smale1960morse} yield the following equality
\begin{align}
\label{eq_corolario_2}
    \chi(\widetilde{X}) = \sum_{i=0}^n (-1)^i \widetilde{c}_i,
\end{align}
where $\widetilde{c}_i$ denotes the number of singular points of index $i$ of the Morse-Smale flow $\widetilde{\varphi}$. By~\eqref{eq_corolario_1} and~\eqref{eq_corolario_2},
\[
    \sum_{i=0}^n(-1)^i \widetilde{c}_i=\chi (X)+\sum_{k\geq 2}(k-1)\# \operatorname{M}(\mathcal{C}_k).
\]
Hence, 
\begin{align}
\label{eq_corolario_3}
    \chi(X)= \sum_{i=0}^n(-1)^i \widetilde{c}_i + \sum_{k\geq 2}(1-k)\# \operatorname{M}(\mathcal{C}_k).
\end{align}
On the other hand, 
if $p\in \operatorname{M}(\mathcal{C}_k)$ is a singularity of nature $a^{\eta_0(p)} s_1^{\eta_{1}(p)} \dots s_{n-1}^{\eta_{n-1}(p)} r^{\eta_n(p)}$, the points associated to $p$ on the morsification are
\[
    \mathfrak{h}(p)=\left\{ \widetilde{p}^1_0, \dots , \widetilde{p}^{\eta_0(p) }_0, \dots, \widetilde{p}^{1}_{\lambda}, \dots, \widetilde{p}^{\eta_\lambda (p)}_{\lambda}, \dots, \widetilde{p}^{1}_{n}, \dots, \widetilde{p}^{\eta_n(p)}_n \right\}, 
\]
where each $\widetilde{p}^i_{\lambda}$ is a nondegenerate critical point of index $\lambda$ of $\widetilde{\varphi}$, and
$\eta_{\lambda}(p)=\# \left\{ \widetilde{p}^1_{\lambda}, \dots , \widetilde{p}^{\eta_{\lambda}(p)}_{\lambda}  \right\}$, for all $\lambda \in\{0,\cdots,n\}$. It follows that
\[
    \widetilde{c}_i=\sum_{p\in \operatorname{Sing}(\varphi)}\eta_i(p), \qquad \forall \ i\in\{0,1,\dots, n\}.
\]
Then,
\[
    \sum_{i=0}^n(-1)^i \widetilde{c}_i = \sum_{i=0}^n(-1)^i \sum_{p\in \operatorname{Sing}(\varphi)}\eta_i(p).
\]

Therefore, upon substituting this expression into \eqref{eq_corolario_3}, we obtain
\[
 \chi(X) = \sum_{i=0}^n (-1)^i c_i + \sum_{k \geq 2 } (1-k) \# \operatorname{M}(\mathcal{C}_k).
\]
\end{proof}

The following examples illustrate the use of previous results relating the Euler-Poincaré characteristic to the system's dynamics, demonstrating how topological properties can be derived from dynamical data.

\begin{example}
Consider the closed $n$-pseudomanifold $X$ presented in Example~\ref{example_spheres}, equipped with the singular Morse-Smale flow introduced in that same example. By Theorem~\ref{theorem_characteristic_conley},
\[
    \chi (X) = (k-\eta_0) + (-1)^n(k-\eta_n) + \left( -(k-\eta_0 -1) +(-1)^n \eta_n \right),
\]
where the first term corresponds to the alternating sum of the Conley indices of all attracting singularities in the regular part, and the second term corresponds to the alternating sum of the Conley indices of all repelling singularities in the regular part, with the factor $(-1)^n$ arising from the dimension. The last term is the alternating sum of $h_*(p)$, where, again, the factor $(-1)^n$ is due to the dimension $n$. It is straightforward to see that $ \chi (X) = (-1)^n k +1$, that is,  
\begin{align}
\label{equacao_esferas_01}
    \chi(X) & = 
    \begin{cases} 
        k+1, \quad \text{if } n \text{ is even};\\
        1-k, \quad \text{if } n \text{ is odd}.
    \end{cases}
\end{align}

By Theorem~\ref{theorem_characteristic_morsification} and the previous equality, we have
$$
    \chi(\widetilde{X}) \ = \ \chi (X) + \sum_{k\geq 2} (k-1)\#\operatorname{M}(\mathcal{C}_k) \ = \ (-1)^n k + 1 +(k-1)\ = \ (-1)^n k + k.
$$
Thus,  
\begin{align}
\label{equacao_esferas_02}
    \chi(\widetilde{X}) &= 
    \begin{cases} 
        2k, \quad \text{if } n \text{ is even};\\
        0, \quad \text{if } n \text{ is odd}.
    \end{cases}
\end{align}
Finally, we recover the Euler-Poincaré characteristic of $X$ using an alternative approach given by the formula in  Corollary~\ref{corollary_cara}:
$$
    \chi(X) \ = \ \sum_{i=0}^n (-1)^ic_i + \sum_{k\geq 2}(1-k) \# \operatorname{M}(\mathcal{C}_k) \ =\ (k+(-1)^nk) +(1-k) \ = \ (-1)^nk +1. 
$$

Figure~\ref{figura_buque_atrator_repulsor_8_folhas} depicts the pair $(X,\varphi)$ for the particular case where $n=2$, $k=8$, $\eta_0=4$, and $\eta_2=4$.

\begin{figure}[!ht]
\centering
\begin{overpic}[width=5.5cm]{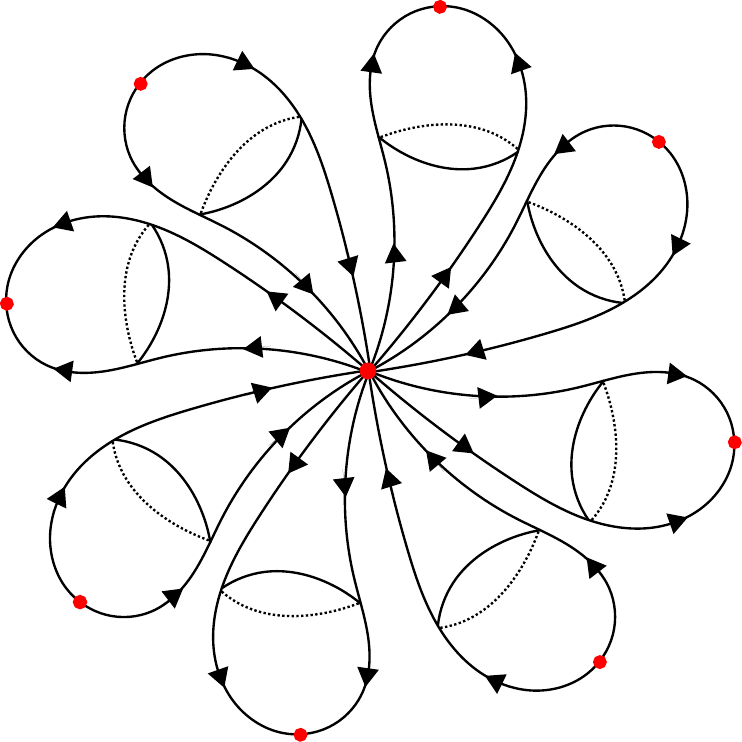}
\put(2,90){$X$}
\end{overpic}
\caption{Spherical-cone singularity $p \in \operatorname{M}(\mathcal{C}_8)$ of nature $a^4r^4$, in dimension $2$.}
\label{figura_buque_atrator_repulsor_8_folhas}
\end{figure}

In this case, equations~\eqref{equacao_esferas_01} and~\eqref{equacao_esferas_02} provide 
\(
    \chi(X) = k +1 = 9
\) and 
\(
    \chi (\widetilde{X}) = 2k = 16.
\)
\hfill $\diamond$
\end{example}

\begin{example}
Consider the $n$-pseudomanifold with spherical-cone singularities 
\(
X=\left(T^n\sqcup M_1 \sqcup M_2 \right)/\sim
\)
given in Example~\ref{example_torus}, where $n$ is an arbitrary even integer greater than $2$. As discussed previously,  $\operatorname{Sing}(X)=\{p_1, p_2\}$, where $p_1\in \operatorname{M}(\mathcal{C}_{2m_1+1})$ has nature $s_{\lambda}r^{2m_1}$ and $p\in \operatorname{M}(\mathcal{C}_{m_2+1})$ has nature $a^{m_2}s_{\lambda}$. We can  compute $\chi(X)$ using Theorem~\ref{theorem_characteristic_conley}:
\begin{align*}
    \chi(X) =& \sum_{p\in \operatorname{Sing}(\varphi)} \sum_{j=0}^n (-1)^jh_j(p)\\
    =& \left( \sum_{j=0}^n (-1)^j c_j -2(-1)^{\lambda} \right) + (-1)^n m_2 +\sum_{j=0}^n (-1)^jh_j(p_1) + \sum_{j=0}^n(-1)^j h_j(p_2)\\
    =& \ \chi(T^n) -2(-1)^{\lambda} +m_2 +\sum_{j=0}^n (-1)^jh_j(p_1) + \sum_{j=0}^n (-1)^jh_j(p_2).
 \end{align*}
Here, the first term in the second equality comes from the remaining flow singularities in the regular part of $T^n\setminus \{ p_1, p_2\}$, since $2$ singularities of the set $c_{\lambda}$ were mapped to the cone points; the second term comes from the $m_2$ repelling singularities in the regular part of $M_2$ that remain after the quotient. To cover all singularities in $\operatorname{Sing}(\varphi)$, it only remains to include the alternating sum of the Conley indices of the singularities $p_1$ and $p_2$, which yields the second equality. In the last equality, we used the fact 
\[
\chi(T^n) =\sum_{j=0}^n(-1)^jc_j
\]
and that $n$ is even. Now, using that $\chi(T^n)=0$ and the numerical Conley indices of $p_1$ and $p_2$ compute in Example~\ref{example_torus}, we obtain 
\begin{align*}
\chi(X)=& \ \chi(T^n) -2(-1)^{\lambda} +m_2 +\sum_{j=0}^n(-1)^j h_j(p_1) + \sum_{j=0}^n(-1)^j h_j(p_2)\\
=& \ 0 -2(-1)^{\lambda} + m_2 + (-2m_1 +(-1)^{\lambda} + 2m_1) + (-1)^{\lambda}\\
=& \ m_2. 
\end{align*}
Therefore, $\chi(X)=m_2$. Now, by Theorem~\ref{theorem_characteristic_morsification}, we have 
\begin{align*}
    \chi(\widetilde{X}) =& \ \chi(X) + \sum_{k\geq 2}(k-1)\#\operatorname{M}(\mathcal{C}_k)\\
    =& \ m_2 + ((2m_1+ 1)-1) + ((m_2+1)-1)\\
    =& \ 2(m_1+m_2). 
\end{align*}
Hence, $\chi(\widetilde{X})=2(m_1+m_2)$, where $\widetilde{X}$ is the smooth $n$-manifold obtained from $X$ by Lemma~\ref{lemma_morsification}. 

A particular case of this example for $n=2$, $m_1=2$, and $m_2=1$ can be seen in Figure~\ref{fig_example_002}. In this case, 
$\chi(X)=m_2=1$ and $\chi(\widetilde{X})=2(m_1+m_2)=6.$  
\hfill $\diamond$

\begin{figure}[!ht]
\centering
\begin{overpic}[width=5cm]{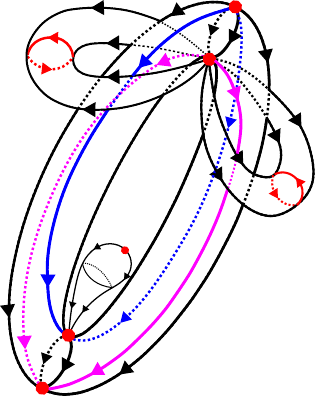}
\put(44,73){$p_1$}
\put(30,40){$p_2$}
\end{overpic}
\caption{Closed $2$-pseudomanifold with spherical-cone singularity.}
\label{fig_example_002}
\end{figure}
\end{example}

\section{Intersection Homology for Pseudomanifolds with Spherical-Cone Singularities}
\label{section_IH}

In this section, we present important results concerning intersection homology of pseudomanifolds equipped with a singular Morse-Smale flow. Subsection~\ref{subsection_IH_background}, 
provides the necessary background on intersection homology, including the basic definitions and results used in our arguments.

We begin with a result that extends a classical theorem in intersection homology. We emphasize that it is purely topological in nature and does not require the pseudomanifold to be endowed with a dynamical system.

\begin{theorem}
\label{teo_topologico}
Let $X$ be an oriented closed $n$-pseudomanifold with spherical-cone singularities. Let $\overline{p}$ be any perversity on $X$. Then the intersection homology groups are given by
            \begin{equation}
            \label{eq_teorema}
                IH^{\overline{p}}_{i}(X)\cong
                \begin{cases}
                    H_{i}(X), & \text{if } i > n-1 -\overline{p}(n), \\
                    \Ima \left(H_i(X\setminus \operatorname{Sing} (X))\longrightarrow H_i(X) \right), & \text{if } i = n-1 -\overline{p}(n),\\
                    H_{i}(X \setminus \operatorname{Sing} (X)), & \text{if } i < n-1 -\overline{p}(n).
                \end{cases}
            \end{equation}
Furthermore, for any integer $i$ such that $1 < i < n-1$ and $i \neq n-1-\overline{p}(n)$, the following isomorphisms hold
\begin{equation}
\label{eq_teorema_2}
    IH^{\overline{p}}_i(X) \cong H_i(X) \cong H_i(X \setminus \operatorname{Sing} (X)).
\end{equation}
\end{theorem}

\begin{proof}
Let $\operatorname{Sing}(X)=\{p_1,\ldots, p_{m_0} \}$ be the singular set of $X$.
For each singularity $p_j \in \operatorname{Sing}(X)$, consider the open cone $c(L_{p_j})$ over its link $L_{p_j}$, and define the following open subsets of $X$
\[
    U_1 = X \setminus \operatorname{Sing}(X) \ \ \text{and} \ \ U_2=\bigcup_{j=1}^{m_0} c(L_{p_j}).
\]
These open subsets cover $X$, that is, $U_1 \cup U_2 = X$. Furthermore, there is a homotopy equivalence
\[
    U_1 \cap U_2  \simeq \bigcup_{j=1}^{m_0}L_{p_j}.
\]
Since $L_{p_j}$ is a smooth manifold, then for any perversity $\overline{p}$
\begin{align}
\label{eq_2}
    IH^{\overline{p}}_i(U_1 \cap U_2) &  \ \cong \ \bigoplus_{j=1}^{m_0}IH^{\overline{p}}_i(L_{p_j}) \ \cong \ \bigoplus_{j=1}^{m_0} H_i(L_{p_j}).
\end{align}
Analogously,
\begin{align}
\label{eq_3}
    H_i(U_1\cap U_2) & \cong \bigoplus_{j=1}^{m_0}H_i(L_{p_j}). 
\end{align}
By the definition of $U_2=\bigcup_{j=1}^{m_0}c(L_{p_j})$, we have the following isomorphisms
\begin{align}
\label{eq_4}
    IH^{\overline{p}}_i(U_2)& \cong \bigoplus_{j=1}^{m_0}IH^{\overline{p}}_i(c(L_{p_j}))
\end{align}
and
\begin{align}
\label{eq_5}
    H_i(U_2)\cong \bigoplus_{j=1}^{m_0} H_i(c(L_{p_j})). 
\end{align}

Since $U_1$ and $U_2$ are open subsets whose union is $X$, and given that the singularities are isolated, we can apply the Mayer-Vietoris sequence in intersection homology. For any perversity $\overline{p}$, this sequence is given by
\begin{equation*}
\begin{gathered}
\begin{tikzcd}[column sep=1.5em, row sep=1.5em, font=\normalsize] 
\cdots \arrow{r} & IH^{\overline{p}}_{i}( U_1\cap U_2) \arrow{r}{f=(i_{\ast},j_{\ast})}                              & IH^{\overline{p}}_{i}( U_1 )\oplus IH^{\overline{p}}_{i}(U_2) \arrow{r}{g=k_{\ast}-l_{\ast}}& IH^{\overline{p}}_{i}(X) \arrow{lldd}{\partial_{\ast}} \\
                                           &                                                               &                                                               &                                        \\
& IH^{\overline{p}}_{i-1}(U_1 \cap U_2) \arrow{r}& IH^{\overline{p}}_{i-1}(U_1 )\oplus IH^{\overline{p}}_{i-1}(U_2) \arrow{r}& IH^{\overline{p}}_{i-1}(X) \arrow{r} & \cdots 
\end{tikzcd}
\end{gathered}
\end{equation*}
where $i_{\ast}$, $j_{\ast}$, $k_{\ast}$, and $l_{\ast}$ are the maps induced by the inclusions $i\colon U_1 \cap U_2\to U_1$, $j\colon U_1  \cap U_2 \to U_2$, $k\colon U_1 \to X$, and $l\colon U_2\to X$. Using equations~\eqref{eq_2} and~\eqref{eq_4}, and the fact that $U_1$ is smooth, we obtain
\begin{equation}
\label{eq_6}
\begin{tikzcd}[baseline=(current bounding box.center), column sep=0.6em, row sep=1em, font=\normalsize]
\cdots \arrow{r} & \displaystyle\bigoplus_{j=1}^{m_0}H_i(L_{p_j}) \arrow{r}{f}                               & H_{i}(U_1)\oplus \displaystyle\bigoplus_{j=1}^{m_0}IH^{\overline{p}}_i(c(L_{p_j})) \arrow{r}{g} & IH^{\overline{p}}_{i}(X) \arrow{lldd}{\partial_{\ast}} \\
                                         &                                                               &                                                                                               &                                        \\
& \displaystyle\bigoplus_{j=1}^{m_0}H_{i-1}(L_{p_j}) \arrow{r} & H_{i-1}(U_1)\oplus \displaystyle\bigoplus_{j=1}^{m_0}IH^{\overline{p}}_{i-1}(c(L_{p_j})) \arrow{r}& IH^{\overline{p}}_{i-1}(X) \arrow{r} & \cdots 
\end{tikzcd}
\end{equation}
Analogously, the Mayer-Vietoris sequence for the pair $(U_1, U_2)$ in singular homology is given by
\begin{equation}
\label{eq_7}
\begin{tikzcd}[baseline=(current bounding box.center), column sep=0.8em, row sep=1em, font=\normalsize] 
\cdots \arrow{r} &\displaystyle\bigoplus_{j=1}^{m_0}H_i(L_{p_j}) \arrow{r}{f}                               & H_{i}(U_1)\oplus \displaystyle\bigoplus_{j=1}^{m_0} H_i(c(L_{p_j})) \arrow{r}{g}& H_{i}(X) \arrow{lldd}{\partial_{\ast}} \\
                                         &                                                               &                                                                                               &                                        \\
& \displaystyle\bigoplus_{j=1}^{m_0}H_{i-1}(L_{p_j}) \arrow{r} & H_{i-1}(U_1)\oplus \displaystyle\bigoplus_{j=1}^{m_0} H_{i-1}(c(L_{p_j})) \arrow{r} & H_{i-1}(X) \arrow{r} & \cdots   
\end{tikzcd}
\end{equation}
For each $j \in \{1, \dots, m_0\}$, the link $L_{p_j}$ is a compact smooth $(n-1)$-dimensional manifold. Therefore, according to \cite[Proposition 4.7.2]{kirwan2006}, we have
\begin{align*}
    IH^{\overline{p}}_i(c(L_{p_j}))=
    \begin{cases}
    IH^{\overline{p}}_i(L_{p_j}),& \text{ if } i<n-1-\overline{p}((n-1)+1);\\
    0, & \text{ if } i \geq n-1-\overline{p}((n-1)+1).
    \end{cases}
\end{align*}
Then,
\begin{align}
\label{eq_8}
    IH^{\overline{p}}_i(c(L_{p_j}))=
    \begin{cases}
    H_i(L_{p_j}),& \text{ if } i<n-1-\overline{p}(n);\\
    0, & \text{ if } i \geq n-1-\overline{p}(n).
    \end{cases}
\end{align}

Up to this point in the proof, we have carried out a general analysis that holds for an arbitrary perversity without additional hypotheses. We now proceed to analyze the assumptions of each part of the theorem to complete the proof.

Suppose that $i < n-1-\overline{p}(n)$. Combining~\eqref{eq_6} and~\eqref{eq_8}, we obtain the long exact sequence
\begin{equation*}
\begin{gathered}
\begin{tikzcd}[column sep=1.2em, row sep=1.2em, font=\normalsize] 
\cdots \arrow{r} & \displaystyle\bigoplus_{j=1}^{m_0}H_i(L_{p_j}) \arrow{r}{f}                              & H_{i}( U_1 )\oplus \displaystyle\bigoplus_{j=1}^{m_0}H_i(L_{p_j}) \arrow{r}{g} & IH^{\overline{p}}_{i}(X) \arrow{lldd}{\partial_{\ast}} \\
                                           &                                                               &                                                                 &                                        \\
& \displaystyle\bigoplus_{j=1}^{m_0}H_{i-1}(L_{p_j}) \arrow{r} & H_{i-1}(U_1)\oplus \displaystyle\bigoplus_{j=1}^{m_0}H_{i-1}(L_{p_j}) \arrow{r}& IH^{\overline{p}}_{i-1}(X) \arrow{r} & \cdots 
\end{tikzcd}
\end{gathered}
\end{equation*}

Using the fact that $f=(i_{\ast}, j_{\ast})$, we deduce that $\Ker f = 0$. Consequently, the First Isomorphism Theorem implies that $\Ima f \cong \bigoplus_{j=1}^{m_0}H_i(L_{p_j})$. By exactness, we have $\Ker g = \Ima f \cong \bigoplus_{j=1}^{m_0}H_i(L_{p_j})$. Thus, applying the First Isomorphism Theorem to $g$, we obtain
\begin{align*}
    \Ima g &\cong \frac{ H_i(U_1) \oplus \bigoplus_{j=1}^{m_0}H_i(L_{p_j}) }{\Ker g} \\
           &\cong \frac{ H_i(U_1) \oplus \bigoplus_{j=1}^{m_0}H_i(L_{p_j}) }{\bigoplus_{j=1}^{m_0}H_i(L_{p_j})} \\
           &\cong H_i(U_1).
\end{align*}
On the other hand, by the exactness of the sequence and the fact that $f$ is injective, we have $\Ker \partial_{\ast} = \Ima g$ and $\Ima \partial_{\ast} = \Ker f = 0$. This means that $\partial_{\ast}$ is the zero map, which yields $\Ima g = IH_{i}^{\overline{p}}(X)$. Therefore, if $i < n-1-\overline{p}(n)$, then $IH^{\overline{p}}_i(X) \cong H_i (X \setminus \operatorname{Sing}(X))$.

Now, assuming that $i > n-1-\overline{p}(n)$, this implies that $i > i-1 \geq n-1-\overline{p}(n) \geq 1$ because, for any perversity, we have $\overline{p}(n) \leq \overline{t}(n) = n-2$. Therefore, $H_i(c(L_{p_j})) \cong 0$, since $c(L_{p_j})$ is contractible. Thus, by \eqref{eq_6}, \eqref{eq_7}, and \eqref{eq_8}, we obtain the following commutative diagram, where the maps $\omega_k$ are induced by the natural chain inclusions $IC_{\ast}^{\overline{p}} \hookrightarrow C_{\ast}$:

\begin{equation*}
\begin{gathered}
\begin{tikzcd}[column sep=1.5em, row sep=1.5em, font=\normalsize] 
\displaystyle\bigoplus_{j=1}^{m_0}  H_{i}(L_{p_j}) \arrow{dd}{\omega_1}\arrow{r} & H_{i}(U_1) \arrow{dd}{\omega_2}\arrow{r}& IH^{\overline{p}}_i(X) \arrow{dd}{\omega_3} \arrow{r}& \displaystyle\bigoplus_{j=1}^{m_0} H_{i-1}(L_{p_j}) \arrow{dd}{\omega_4}\arrow{r}&  H_{i-1}(U_1)\arrow{dd}{\omega_5}\\
                                                             &                                             &                                             &                                                                &                                   \\
\displaystyle\bigoplus_{j=1}^{m_0} H_{i}(L_{p_j}) \arrow{r}                            & H_{i}(U_1) \arrow{r}                            & H_i(X) \arrow{r}                            & \displaystyle\bigoplus_{j=1}^{m_0} H_{i-1}(L_{p_j})  \arrow{r}                            &  H_{i-1}(U_1)         
\end{tikzcd}
\end{gathered}
\end{equation*}
Since $\omega_1$, $\omega_2$, $\omega_4$, and $\omega_5$ are isomorphisms, the Five Lemma implies that $\omega_3$ is also an isomorphism. Therefore, if $i > n-1-\overline{p}(n)$, we have $IH^{\overline{p}}_i(X) \cong H_i(X)$.

Finally, suppose that $i = n-1-\overline{p}(n)$. By \eqref{eq_8}, we have $IH^{\overline{p}}_i(c(L_{p_j})) = 0$ and $IH^{\overline{p}}_{i-1}(c(L_{p_j})) = H_{i-1}(L_{p_j})$. Furthermore, $H_{i}(c(L_{p_j})) = 0$. Using the exact sequences in \eqref{eq_6} and \eqref{eq_7}, we obtain the commutative diagram
\begin{equation*}
\begin{gathered}
\begin{tikzcd}[column sep=1em, row sep=1.5em, font=\small ] 
\displaystyle\bigoplus_{j=1}^{m_0}  H_{i}(L_{p_j}) \arrow{dd}{\omega_1}\arrow{r} & H_{i}(U_1) \arrow{dd}{\omega_2}\arrow{r}{g}& IH^{\overline{p}}_i(X) \arrow{dd}{\omega_3} \arrow{r}{\partial_{\ast}}& \displaystyle\bigoplus_{j=1}^{m_0} H_{i-1}(L_{p_j}) \arrow{dd}{\omega_4}\arrow{r}{f}&  H_{i-1}(U_1)\displaystyle\bigoplus_{j=1}^{m_0} H_{i-1}(L_{p_j}) \arrow{dd}{\omega_5}\\
                                                             &                                             &                                             &                                                                &                                   \\
\displaystyle\bigoplus_{j=1}^{m_0} H_{i}(L_{p_j}) \arrow{r}                            & H_{i}(U_1) \arrow{r}{g'}                            & H_i(X) \arrow{r}                            & \displaystyle\bigoplus_{j=1}^{m_0} H_{i-1}(L_{p_j})  \arrow{r}                            &  H_{i-1}(U_1) \displaystyle\bigoplus_{j=1}^{m_0}H_{i-1}(c(L_{p_j}))         
\end{tikzcd}
\end{gathered}
\end{equation*}
Since $\Ker f = \{0\}$, by exactness, we have $\Ima \partial_{\ast} = \Ker f = \{0\}$; therefore, $\partial_{\ast}$ is the zero map. Thus, by exactness
\[
IH^{\overline{p}}_i(X) = \Ker \partial_{\ast} = \Ima g.
\]

We know that $\omega_2$ is an isomorphism, and the diagram is commutative, so
\[
g'\circ \omega_2=\omega_3\circ g \ \Rightarrow \ g'=\omega_3\circ g\circ \omega_2^{-1}.
\]
Hence, $\Ima g'= \Ima (\omega_3\circ g)$. By the exactness of the rows and the commutativity of the diagram, the fact that $\omega_1$ and $\omega_2$ are isomorphisms guarantees that the restriction $\restr{\omega_3}{\Ima g}$ is injective. Therefore, $\Ima g' \cong \Ima g$, and we conclude that if $i=n-1-\overline{p}(n)$, then
\begin{align*}
    IH^{\overline{p}}_i(X) & \cong  \Ima \left(H_i(U_1)\longrightarrow H_i(X)\right)\\
    &= \Ima \left(H_i(X\setminus \operatorname{Sing}(X))\longrightarrow H_i(X)\right) .
\end{align*}

This concludes the proof of the first part of the theorem, yielding the isomorphism~\eqref{eq_teorema}.

To prove the second part of the theorem, namely that for $1 < i < n-1$ with $i \neq n-1-\overline{p}(n)$ we have the isomorphisms
\[
    IH^{\overline{p}}_i(X) \cong H_i(X) \cong H_i(X \setminus \operatorname{Sing}(X)),
\]
it suffices to show that if $1 < i < n-1-\overline{p}(n)$ then $IH^{\overline{p}}_i(X) \cong H_i(X)$, and if $n-1-\overline{p}(n) < i < n-1$ then $IH^{\overline{p}}_i(X) \cong H_i(X \setminus \operatorname{Sing}(X))$, since we have already established the isomorphism~\eqref{eq_teorema}.

First, suppose that $i \in \mathbb{Z}$ satisfies $1 < i < n-1-\overline{p}(n)$, then we necessarily have $0 < i-1 < i < n-1-\overline{p}(n)$ and $\overline{p}(n) < \overline{t}(n) = n-2$. By \eqref{eq_8}, we have $IH^{\overline{p}}_i(c(L_{p_j})) \cong H_i(L_{p_j})$ and $IH^{\overline{p}}_{i- 1}(c(L_{p_j})) \cong H_{i-1}(L_{p_j})$. Furthermore, since $0 < i-1 < i < n-1 -\overline{p}(n) < n-1$, we have
\[
    H_i(L_{p_j})=H_{i-1}(L_{p_j})=H_i(c(L_{p_j}))=H_{i-1}(c(L_{p_j}))\cong 0.
\]

From this observation and the exact sequences in \eqref{eq_6} and \eqref{eq_7}, we obtain the commutative diagram:
\begin{equation*}
\begin{gathered}
\begin{tikzcd}[column sep=1.1em, row sep=1.5em, font=\normalsize] 
\displaystyle\bigoplus_{j=1}^{m_0} H_{i}(L_{p_j}) \arrow{dd}{\omega_1}\arrow{r} &  H_{i}(U_1) \arrow{dd}{\omega_2}\arrow{r}&  IH^{\overline{p}}_i(X) \arrow{dd}{\omega_3} \arrow{r}& \displaystyle\bigoplus_{j=1}^{m_0}  H_{i-1}(L_{p_j}) \arrow{dd}{\omega_4}\arrow{r}&  H_{i-1}(U_1)  \arrow{dd}{\omega_5}\\
                                                             &                                             &                                             &                                                                &                                   \\
\displaystyle\bigoplus_{j=1}^{m_0} H_{i}(L_{p_j}) \arrow{r}                            &  H_{i}(U_1) \arrow{r}                            &  H_i(X) \arrow{r}                            &  \displaystyle\bigoplus_{j=1}^{m_0} H_{i-1}(L_{p_j})  \arrow{r}                            &  H_{i-1}(U_1)                        
\end{tikzcd}
\end{gathered}
\end{equation*}

By the Five Lemma, we conclude that $IH^{\overline{p}}_i(X) \cong H_i(X)$ for $1 < i < n-1-\overline{p}(n)$. Combining this with the isomorphism~\eqref{eq_teorema} yields
\begin{equation}
\label{eq_01}
    IH^{\overline{p}}_i(X) \cong H_i(X) \cong H_i(X \setminus \operatorname{Sing} (X)), \quad \text{whenever} \quad 1 < i < n-1-\overline{p}(n).
\end{equation}

Finally, suppose that $i \in \mathbb{Z}$ and $n-1-\overline{p}(n) < i < n-1$. Then we necessarily have $n-1-\overline{p}(n) \leq i-1 < i$. By \eqref{eq_8} and \eqref{eq_7}, we have the following exact sequence
\begin{equation*}
\begin{gathered}
\begin{tikzcd}[column sep=1em, row sep=1.5em, font=\small] 
\displaystyle\bigoplus_{j=1}^{m_0}H_i(L_{p_j}) \arrow{r}{f}  & H_{i}(U_1) \arrow{r}{g} & IH^{\overline{p}}_{i}(X) \arrow{r}{\partial_{\ast}} & \displaystyle\bigoplus_{j=1}^{m_0}H_{i-1}(L_{p_j}) \arrow{r} & H_{i-1}(U_1) \arrow{r}& IH^{\overline{p}}_{i-1}(X)                
\end{tikzcd}
\end{gathered}
\end{equation*}
Furthermore, if $n-1-\overline{p}(n) < i < n-1$, then $1 \leq n-1-\overline{p}(n) \leq i-1 < i < n-1$, and we have 
\begin{align*}
    H_i(L_{p_j})&=H_i(\cup \es^{n-1}) \cong \bigoplus H_{i}(\es^{n-1})\cong 0; \\
    H_{i-1}(L_{p_j})&=H_{i-1}(\cup \es^{n-1}) \cong \bigoplus H_{i-1}(\es^{n-1}) \cong 0.
\end{align*}
Thus, we obtain the exact sequence:
\begin{equation*}
\begin{gathered}
\begin{tikzcd}[column sep=1em, row sep=1.5em, font=\small] 
0 \arrow{r}{f}  & H_{i}(U_1) \arrow{r}{g} & IH^{\overline{p}}_{i}(X) \arrow{r}{\partial_*} & 0  \arrow{r} & H_{i-1}(U_1) \arrow{r}& IH^{\overline{p}}_{i-1}(X) .
\end{tikzcd}
\end{gathered}
\end{equation*}
By exactness, we have
\begin{align}
\label{eq_9}
    \Ima  g= \Ker \partial_{\ast} \cong IH^{\overline{p}}_i(X).
\end{align}
Again, by exactness, we have
\begin{align}
\label{eq_10}
    \{0\}=\Ima f = \Ker g.
\end{align}
From \eqref{eq_9} and \eqref{eq_10}, we see that $g$ is an isomorphism; hence, $IH^{\overline{p}}_i(X) \cong H_i(U_1)$. Therefore, if $n-1-\overline{p}(n) < i < n-1$, then $IH^{\overline{p}}_i(X) \cong H_i(X\setminus \operatorname{Sing}(X))$. Combining this with the isomorphism \eqref{eq_teorema} yields
\begin{equation}
\label{eq_02}
    IH^{\overline{p}}_i(X) \cong H_i(X) \cong H_i(X \setminus \operatorname{Sing} (X)), \quad \text{whenever} \quad n-1-\overline{p}(n) < i < n-1.
\end{equation}

We conclude the proof by observing that the isomorphisms \eqref{eq_01} and \eqref{eq_02} together imply
\[
    IH^{\overline{p}}_i(X) \cong H_i(X) \cong H_i(X \setminus \operatorname{Sing} (X)), \quad \text{for} \quad 1 < i < n-1 \quad \text{with} \quad i \neq n-1-\overline{p}(n).
\]
\end{proof}

\begin{remark}
Observe that the second part of Theorem~\ref{teo_topologico}, namely the isomorphism given in \eqref{eq_teorema_2}, only applies when the dimension is sufficiently large. Indeed, for dimensions $n=2$ and $n=3$, there is no integer $i$ satisfying 
\[
    1 < i < n-1 \quad \text{and} \quad i \neq n-1-\overline{p}(n).
\]
However, for $n \geq 4$, this range of degrees becomes nonempty, rendering the isomorphisms applicable. In light of this observation, for $n=2$ and $n=3$, Theorem~\ref{teo_topologico} reduces solely to the isomorphism \eqref{eq_teorema}. Consequently, it recovers Theorem~2 established in \cite{LimaTenorio_2}. Thus, the present theorem provides a genuine extension of that previous result to arbitrary dimensions.
\end{remark}

\begin{example}
Still regarding Theorem~\ref{teo_topologico}, the hypothesis of spherical-cone singularities is essential to obtain the isomorphisms given in \eqref{eq_teorema_2}. In general, the isomorphisms need not hold when the singular part of $X$ contains cone singularities other than spherical-cone singularities. In this example, we provide such a counterexample.

Consider the suspension of the $4$-dimensional torus, $\Sigma T^4$, where $T^4 = \es^1\times \es^1\times \es^1\times \es^1$. The space $\Sigma T^4$ is a $5$-dimensional pseudomanifold with two isolated singularities, $\{p, q\}$. We know that 
\[
    H_i(\Sigma T^4\setminus \{ p,q \}) \cong H_i(T^4),
\]
since $\Sigma T^4\setminus \{p,q\}$ is homotopy equivalent to $T^4$. Furthermore, it is a standard fact that $\widetilde{H}_i(\Sigma T^4)\cong \widetilde{H}_{i-1}(T^4)$ for all $i$, where $\widetilde{H}_{\ast}( - )$ denotes reduced singular homology. Therefore, we obtain the homology groups presented in Table~\ref{tab_counterexample}. 

\begin{table}[htpb]
\centering
\renewcommand{\arraystretch}{1.3}
\begin{tabular}{|c|c|c|}
\hline
 \ \ $i$ \ \ & $H_i(\Sigma T^4 \setminus \{p, q\})$ & $H_i(\Sigma T^4)$ \\
\hline
$0$ & $\mathbb{Z}$   & $\mathbb{Z}$   \\
\hline
$1$ & $\mathbb{Z}^4$ & $0$            \\
\hline
$2$ & $\mathbb{Z}^6$ & $\mathbb{Z}^4$ \\
\hline
$3$ & $\mathbb{Z}^4$ & $\mathbb{Z}^6$ \\
\hline
$4$ & $\mathbb{Z}$   & $\mathbb{Z}^4$ \\
\hline
$5$ & $0$            & $\mathbb{Z}$   \\
\hline
\end{tabular}
\vspace{0.2cm}
\caption{Comparison of singular homology groups for the regular stratum and the total space of $\Sigma T^4$.}
\label{tab_counterexample}
\end{table}

Thus, we clearly have $H_2(\Sigma T^4 \setminus \{ p, q\}) \ncong H_2(\Sigma T^4 )$ and $H_3(\Sigma T^4 \setminus \{ p, q\}) \ncong H_3(\Sigma T^4 )$. Consequently, the isomorphism \eqref{eq_teorema_2} does not hold, even though for the zero perversity $\overline{p}(5) = \overline{0}(5) = 0$ and the top perversity $\overline{p}(5) = \overline{t}(5) = 3$, the integers $i=2,3$ satisfy the required hypotheses:
\[
    1 < i < 4 \quad \text{and} \quad i \neq 4-\overline{p}(5).
\]
This failure occurs because the links of $p$ and $q$ are $4$-dimensional tori, rather than a disjoint union of $4$-dimensional spheres.
\end{example}

\begin{example}
The isomorphism given in \eqref{eq_teorema} within the first part of Theorem~\ref{teo_topologico} is a classical result in intersection homology and holds for cone singularities different from spherical-cone singularities. Consider $X$ as the suspension of the torus, illustrated in Figure~\ref{s:suspentore}. The singular points $p$ and $q$ are conical singularities whose links are homeomorphic to a torus; in particular, they are not spherical-cone singularities.

\begin{figure}[H]
\centering
\begin{tikzpicture}[scale=0.2]

\draw (0,0) ellipse (7 and 4);
\draw (-3,0.7) arc (180:350:3 and 2);
\draw (2.8,0) arc (23:158:3 and 2);

\draw [red,thick](0,0) ellipse (5.5 and 3);
\draw [red,dashed,thick][rotate=200] (2.2,3.3) arc (90:270: 0.7 and 1.5);
\draw [red,thick][rotate=200] (2.2,0.3) arc (270:450: 0.7 and 1.5);

\node at (-2.25,-2)[red] [rotate=290]  {$<$};
\node at (-2.25,-1.5)[left,red]  {$b$};,
\node at (-6.3,0)[red]  {$a$};    
\node at (-5.5,0)[red]   [rotate=90]  {$<$};

\draw (-1.5,10) -- (-5,10) -- (-9,7) -- (13,7) --(17,10)-- (1.5,10);
\draw (-2.5,8.5) arc (180:360:2.5 and 0.8);
\draw (10.5,8.5) ellipse( 2.5 and 0.8);
\draw (9.25,8.8) arc (170:370:1.3 and 0.4);
\draw (11.4,8.5) arc (20:158:1 and 0.4);
\node at (6,8.5) {$=$};

\node at (0,12)[above]{$p$};
\draw (0,12) -- (-2.5,8.5); 
\draw [dotted,very thick] (-2.5,8.5) -- (-3.55,7);
\draw (-3.55,7) -- (-6.5,3);
\draw (0,12)  -- (-2,8);  
\draw [dotted,very thick](-2,8) -- (-2.5,7);
\draw (-2.5,7)-- (-4,4);
\draw (0,12)  -- (-1.5,7.8); 
\draw [dotted,very thick](-1.5,7.8) -- (-1.8,7);
\draw (-1.8,7)-- (-2.5,5);
\draw (0,12)  -- (-0.75,7.7); 
\draw  [dotted,very thick](-0.75,7.7) -- (-0.9,7);
\draw (-0.9,7)-- (-1.25,5.25);
 \draw (0,12)  -- (0,7.7); 
 \draw  [dotted,very thick](0,7.7) -- (0,7);
 \draw (0,7) -- (0,5);
 \draw (0,12)  --(0.9,7.75); 
\draw  [dotted,very thick]( (0.9,7.75) -- (1.05,7);
\draw(1.05,7)--(1.5,5);
\draw (0,12)  -- (1.7,7.8);
\draw  [dotted,very thick](1.7,7.8) --(2,7);
\draw (2,7)-- (3,4.5);
\draw (0,12)  -- (2.15,8.2);
\draw [dotted,very thick](2.15,8.2) -- (2.8,7);
\draw (2.8,7)-- (4.5,4);
\draw (0,12)  -- (2.5,8.5);
\draw [dotted,very thick](2.5,8.5) -- (3.65,7);
\draw(3.65,7)--(6.5,3);

\node at (8.5,0) {$T$};

\node at (0,-12)[below]{$q$};
\draw (0,-12) -- (-6.5,-3);
\draw (0,-12) -- (-4,-4);
\draw (0,-12) -- (-2.5,-5);
\draw (0,-12) -- (-1.25,-5.25);
\draw (0,-12) -- (0,-5.5);
\draw (0,-12) -- (1.5,-5);
\draw (0,-12) -- (3,-4.5);
\draw (0,-12) -- (4.5,-4);
\draw (0,-12) -- (6.5,-3);
\end{tikzpicture}
\caption{Suspension of the torus. The singular points $p$ and $q$ are conical singularities with toroidal links.}\label{s:suspentore}
\end{figure}
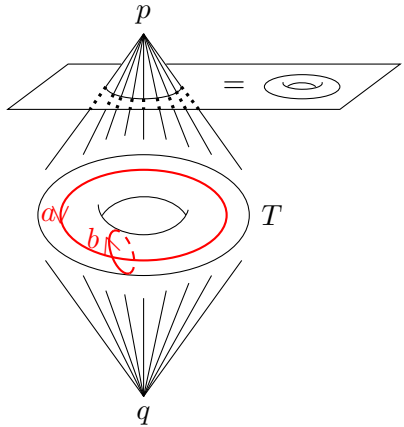

In dimension $3$, there are only two perversities: $\overline{0}(3)=0$ and $\overline{t}(3)=1$. Thus, the isomorphism given in \eqref{eq_teorema} splits into two cases:

\[
IH^{\overline{0}}_{i}(X)\cong
\begin{cases}
    H_{i}(X), & \text{if } i=3, \\
    \Ima \left(H_i(X\setminus \operatorname{Sing} (X))\longrightarrow H_i(X) \right), & \text{if } i = 2,\\
    H_{i}(X \setminus \operatorname{Sing} (X)), & \text{if } i=1,0,
\end{cases}
\]

and 

\[
IH^{\overline{t}}_{i}(X)\cong
\begin{cases}
    H_{i}(X), & \text{if } i=3,2; \\
    \Ima \left(H_i(X\setminus \operatorname{Sing} (X))\longrightarrow H_i(X) \right), & \text{if } i = 1;\\
    H_{i}(X \setminus \operatorname{Sing} (X)), & \text{if } i=0.
\end{cases}
\]

Using the previously calculated singular homologies alongside these isomorphisms, we can construct Table~\ref{tablesuspension}, which displays the singular homology and intersection homology groups of $X$. 

\begin{table}[htpb]
\centering
\renewcommand{\arraystretch}{1.5}
\begin{tabular}{|c|c|c|c|c|}
\hline
\ \  $i$ \ \ & $H_i(X)$ & $H_i(X\setminus \{p,q\})$ & $IH^{\overline{0}}_i(X)$ & $IH^{\overline{t}}_i(X)$ \\
\hline
$0$ & $\mathbb{Z}_{\{x\}}$ & $\mathbb{Z}_{\{x\}}$ & $\mathbb{Z}_{\{x\}}$ & $\mathbb{Z}_{\{x\}}$ \\
\hline
$1$ & $0$ & $\mathbb{Z}_{a} \oplus \mathbb{Z}_{b}$ & $\mathbb{Z}_{a} \oplus \mathbb{Z}_{b}$ & $0$ \\
\hline
$2$ & $\mathbb{Z}_{\Sigma (a)} \oplus \mathbb{Z}_{\Sigma (b)}$ & $\mathbb{Z}_{T}$ & $0$ & $\mathbb{Z}_{\Sigma (a)} \oplus \mathbb{Z}_{\Sigma (b)}$ \\
\hline
$3$ & $\mathbb{Z}_{[X]}$ & $0$ & $\mathbb{Z}_{[X]}$ & $\mathbb{Z}_{[X]}$ \\
\hline
\end{tabular}
\vspace{0.2cm}
\caption{Comparison of the groups $H_*(X)$, $IH_*^{\overline{p}}(X)$, and $H_*(X\setminus\{A,B\})$ for the suspension of the torus $X$.}
\label{tablesuspension}
\end{table}
\end{example}

In the remainder of this section, we develop a dynamical approach to the computation of the intersection homology of closed pseudomanifolds with spherical-cone singularities. Our construction is inspired by the classical relationship between Morse theory and singular homology in the smooth setting.

Recall that if $M$ is a closed smooth manifold and $f\colon M\to \mathbb{R}$ is a Morse function whose gradient flow is Morse-Smale, then the singular homology of $M$ can be recovered from a chain complex generated by the critical points of $f$. This construction, known as the Morse complex, encodes topological information through the dynamics of the associated Morse-Smale flow.
For the convenience of the reader, we briefly recall the main ingredients of this construction. Further details can be found in~\cite{Banyaga2004, Weber2006}.

Let $f\colon M \to \mathbb{R}$ be a Morse function on a smooth $n$-manifold $M$. Fix an arbitrary orientation of the unstable manifold $W^{u}(x)$ for each critical point $x \in \operatorname{Crit}(f)$, and denote by $Or$ the collection of these choices. The \textit{Morse group} $C^M_{\ast}(f)$ is defined as the graded free abelian group generated by the critical points of $f$, graded by their Morse index; i.e.
\[
    C^M_{k}(f) = C^M_{k}(M,f,g, Or; \mathbb{Z}) \coloneqq  \bigoplus_{x \in \operatorname{Crit}_k(f)} \mathbb{Z} \langle x\rangle,
\]
where $\langle x\rangle$ denotes the critical point $x$ together with the chosen orientation of $W^{u}(x)$, and $g$ is a Riemannian metric on $M$. The \textit{k-th Morse boundary map} 
\( 
    \partial_{k}^M \colon C^M_{k}(f) \to C^M_{k-1}(f),
\)
is defined on a generator $\langle x\rangle \in C^M_{k}(f)$ as
\[
    \partial^M_k\langle x\rangle \coloneqq  \sum_{y \in \operatorname{Crit}_{k-1}(f)} n(x,y) \langle y\rangle,  
\]
where $n(x,y)$ is the signed count of connecting trajectories from $x$ to $y$; see~\cite{Weber2006} for a precise definition. Extending linearly yields a chain complex $(C^M_\ast(f),\partial_\ast^M)$, called the \emph{Morse-Smale-Witten complex}, or simply the \emph{Morse complex}.
 The corresponding \textit{Morse homology groups} with integer coefficients are given by
\[
    H^M_{k}(M,f,g,Or;\mathbb{Z}) \coloneqq \frac{\ker(\partial_{k}^M)}{\operatorname{im}(\partial_{k+1}^M)}, \quad \text{for all } k\in \mathbb{Z}.    
\]

A fundamental theorem of Morse homology states that these groups are independent of the choices of Morse function, metric, and orientations, and are naturally isomorphic to the singular homology of $M$:
\[
    H^M_{\ast}(M;\mathbb{Z}):= H^M_{k}(M,f,g,Or;\mathbb{Z}) \cong H_{\ast} (M;\mathbb{Z}).
\]
Although we present the construction with integer coefficients, it extends naturally to coefficients in an arbitrary ring.

Our goal is to use the morsification introduced in Section~\ref{subsection_morsification} to adapt this viewpoint to pseudomanifolds with spherical-cone singularities. The following result enables us to make use of dynamical data associated with a singular Morse-Smale flow without closed orbits on an $n$-pseudomanifold with spherical-cone singularities $X$ to compute its intersection homology.

\begin{theorem}
\label{theorem_IH_with_dyna}
Let $X$ be an orientable closed $n$-pseudomanifold with spherical-cone singularities, and let $\varphi$ be a singular Morse-Smale flow without periodic orbits on $X$. Then, for any perversity $\overline{p}$ on $X$, the following isomorphism holds
\[
    IH^{\overline{p}}_i(X;\Z) \ \cong \ H^M_i(\widetilde{X}, \widetilde{\varphi}; \Z),
\]
where $(\widetilde{X}, \widetilde{\varphi})$ is the morsification of $(X, \varphi)$ given by Lemma~\ref{lemma_morsification}, and $H^M_i(\widetilde{X}, \widetilde{\varphi}; \Z )$ is the Morse homology group of $(\widetilde{X}, \widetilde{\varphi})$.
\end{theorem}

\begin{proof}
    Let $(\widetilde{X}, \widetilde{\varphi})$ be the morsification of $(X, \varphi)$ given by Lemma~\ref{lemma_morsification}. Since $\widetilde{X}$ is an oriented smooth manifold of dimension $n$, we have
    \[
        H_n(\widetilde{X}, \widetilde{X}\setminus \{ x \})\cong \Z,
    \]
    for all $x \in \widetilde{X}$. Thus, $\widetilde{X}$ can be regarded as a normal pseudomanifold.

    Let  $\pi \colon \widetilde{X} \to X$ be the projection associated to the morsification $(\widetilde{X},\widetilde{\varphi})$, as in Lemma~\ref{lemma_morsification}. Given  $p \in \operatorname{Sing}(X)$ such that $p \in \operatorname{M}(\mathcal{C}_k)$, we have 
    \( 
        \pi(  \{ \widetilde{p}_1, \dots , \widetilde{p}_k \})=p\) 
        and \(  \pi^{-1}(p) = \{ \widetilde{p}_1, \dots , \widetilde{p}_k \}.
    \)
  Hence,
    \begin{align*}
        \bigoplus_{q\in \pi^{-1}(p)}H_n(\widetilde{X}, \widetilde{X}\setminus \{q\})\ = \ &\bigoplus_{i=1}^k H_n(\widetilde{X}, \widetilde{X}\setminus \{\widetilde{p}_i\})
        \ \cong \ \bigoplus_{i=1}^k \Z. 
    \end{align*}
    On the other hand, for such point  $p \in \operatorname{M}(\mathcal{C}_k)$, it holds that
    \[
        H_n(X, X\setminus \{p\})\cong \bigoplus_{i=1}^k \Z.
    \]
    Outside of $\operatorname{Sing}(X)$, the map $\pi\colon \widetilde{X}\setminus \pi^{-1}(\operatorname{Sing}(X))\to X\setminus \operatorname{Sing}(X)$ is a diffeomorphism, yielding the following isomorphism
    \[
        H_n(\widetilde{X}, \widetilde{X}\setminus \{ \pi^{-1}(x) \} )\ \cong \ \Z  \ \cong \ H_n(X, X\setminus \{x\}),
    \]
    for all $x \in X \setminus \operatorname{Sing}(X)$. Therefore, in any case, we have the following isomorphism
    \[
        \pi_{\ast}\colon  \bigoplus_{q\in \pi^{-1}(p)}H_n(\widetilde{X}, \widetilde{X}\setminus \{q\}) \longrightarrow H_n(X, X\setminus \{p\}).
    \]
    Thus, $\pi \colon \widetilde{X}\to X$ is a normalization of $X$. It follows from the theorem in Section~4.2 of~\cite{goresky1980} that
    \[
        IH^{\overline{p}}_{\ast} (X; \Z)\cong IH^{\overline{p}}_{\ast} (\widetilde{X}; \Z),
    \]
    for every perversity $\overline{p}$. Since $\widetilde{X}$ is smooth,  $IH^{\overline{p}}_{\ast} (\widetilde{X}; \Z)\cong H_{\ast}(\widetilde{X}; \Z)$. Applying the Morse Homology Theorem, we obtain
    \[
        H_{\ast}(\widetilde{X}; \Z)\cong H^M_{\ast} (\widetilde{X}, \widetilde{\varphi}; \Z),
    \]  
    where $H^M_{\ast} (\widetilde{X}, \widetilde{\varphi}; \Z)$ denotes the Morse homology of $(\widetilde{X}, \widetilde{\varphi})$. Therefore, we obtain the following isomorphism for every perversity $\overline{p}$
    \[
        IH^{\overline{p}}_{\ast} (X; \Z)\cong H^M_{\ast} (\widetilde{X}, \widetilde{\varphi}; \Z).
    \]
\end{proof}

\begin{remark}
Most of the definitions and results used in this article, and hence the majority of the obtained results, remain valid for an arbitrary coefficient group $G$ instead of instead of $\Z$. This opens the door to more general statements. For instance, it would be interesting to investigate the nonorientable case, using $\Z_2$-coefficients.
\end{remark}

To conclude this section, we present a simple example illustrating how dynamical behavior can be used to explicitly compute the intersection homology of a pseudomanifold with spherical-cone singularities via Theorem~\ref{theorem_IH_with_dyna}.

\begin{example} Given a natural number $n\geq 2$, consider the $n$-dimensional torus $\mathbb{T}^n$ endowed with the Morse flow presented  in Example~\ref{example_torus}. We pointed out that in Example~\ref{example_torus}, we assumed $n$ to be even solely to guarantee the existence of a flow on the sphere with exactly two repelling singularities and one attracting periodic orbit; we are not using this structure here. Consider $M_1$ (resp., $M_2$) be the disjoint union of $m_1$ (resp., $m_2$) $n$-spheres endowed with a North-South pole flow. Finally, consider the $n$-pseudomanifold with spherical-cone singularities 
\[
    X = \left( T^n \sqcup M_1 \sqcup M_2 \right) / \sim, 
\]
where the equivalence relation $\sim$ is defined by identifying the $m_1$ repelling singularities of $M_1$ and a saddle of index $\lambda$ on $T^n$ to a single singularity $p_1$, and identifying the $m_2$ attracting singularities of $M_2$ and another saddle of index $\lambda$  on $T^n$ to a single singularity $p_2$. This space can be naturally endowed with a singular Morse-Smale flow $\varphi$ as described in Example~\ref{example_torus}. In this case, we have two spherical-cone singularities, namely $p_1\in \operatorname{M}(\mathcal{C}_{m_1+1})$ of nature $s_{\lambda}r^{m_1}$, and $p_2\in \operatorname{M}(\mathcal{C}_{m_2+1})$ of nature $a^{m_2}s_{\lambda}$.  

The morsification of the pair $(X,\varphi)$ is given simply by $\widetilde{X} = T^n\sqcup M_1 \sqcup M_2$. Note that the associated Morse-Smale flow $\widetilde{\varphi}$ is given by the flows associated to the negative gradient vector field of the function $f \colon T^n \to \R$ on $T^n$ (see Example~\ref{example_torus}), and to the negative gradient of the height function  $g_j \colon M_j \to \R$ on $M_j$,  for $j=1,2$. Thus, the Morse chain complex  associated to $g_j$ is given by
\begin{align*}
C^M_i(M_j, g_j; \Z) =
\begin{cases}
    \bigoplus_{k=1}^{m_j} \Z, & \text{if } i=0, n;\\
    0, & \text{otherwise},
\end{cases}
\end{align*}
and the boundary operators vanish, i.e. $\partial^M_i = 0$ for all $i=0, \dots, n$. Hence, the Morse homology of $M_j$ is equal to
\begin{align*}
H^M_i(M_j, g_j; \Z) =
\begin{cases}
    \bigoplus_{k=1}^{m_j} \Z, & \text{if } i=0, n;\\
    0, & \text{otherwise}.
\end{cases}
\end{align*}
For the Morse function $f$ on the torus $T^n$, we have 
\[
    C^M_i(T^n, f; \Z) = \bigoplus_{k=1}^{c_i} \Z \cong \Z^{c_i},
\]
where $c_i = \binom{n}{i}$. As established in Example~\ref{example_torus}, the gradient dynamics of the separable cosine function on $T^n$ yield pairs of connecting orbits with opposite orientations that cancel out algebraically. This implies that the boundary operators on $T^n$ also vanish, i.e. $\partial^M_i = 0$ for all $i=0, \dots, n$. Thus, 
\[
    H^M_i(T^n, f; \Z) \cong \bigoplus_{k=1}^{c_i} \Z.
\]
Therefore, by Theorem~\ref{theorem_IH_with_dyna}, we obtain the following isomorphisms for all $i$ and for any perversity $\overline{p}$ on $X$
\begin{align*}
    IH^{\overline{p}}_i(X; \Z) & \ \cong \  H^M_i(\widetilde{X}, \widetilde{\varphi}; \Z) \ \cong \ H^M_i(T^n, f; \Z) \oplus H^M_i(M_1, g_1; \Z) \oplus H^M_i(M_2, g_2; \Z).
\end{align*}
Consequently, 
\[
    IH^{\overline{p}}_n(X; \Z) \ \cong \ IH^{\overline{p}}_0(X; \Z)\ \cong \ \bigoplus_{k=1}^{m_1+m_2+1} \Z,
\]
and for all $i \in \{ 1, \dots, n-1 \}$, we obtain
\[
    IH^{\overline{p}}_i(X; \Z)\ \cong \ \bigoplus_{k=1}^{c_i} \Z.
\]

In the case that $n=2$, $m_1=1$ and $m_2=2$, we can visualize the corresponding space $X$ in Figure~\ref{fig_example_003}.
\begin{figure}[!ht]
\centering
\begin{overpic}[width=5cm]{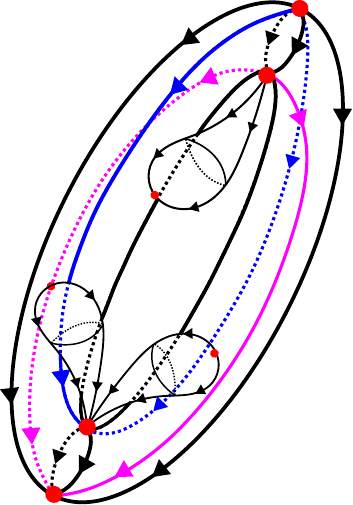}
\put(47,89){$p_1$}
\put(9.5,15.5){$p_2$}
\end{overpic}
\caption{Closed $2$-pseudomanifold with spherical-cone singularity and no periodic orbits.}
\label{fig_example_003}
\end{figure}
In this setting, we restrict ourselves to the zero perversity since, in dimension $2$, the zero perversity is the only one available. Thus
\[
IH^{\overline{0}}_i(X; \Z)= 
\begin{cases}
    \Z \oplus \Z \oplus \Z \oplus \Z, & \text{if }i=2;\\
    \Z \oplus \Z, & \text{if }i=1;\\
    \Z \oplus \Z \oplus \Z \oplus \Z, & \text{if }i=0.
\end{cases}
\]
\hfill $\diamond$
\end{example}

\section{Final Remarks}
\label{section_conclusion}

This work can be viewed as a first attempt to develop a Conley-theoretic framework for singular dynamical systems on higher-dimensional stratified manifolds. Previous related constructions in this direction were restricted to the two-dimensional case, while the present paper initiates the study in arbitrary dimensions.

Within this setting,  we have developed a dynamical framework for studying topological invariants of singular spaces, in which singular Morse–Smale dynamics on pseudomanifolds with spherical-cone singularities is used to encode and recover homological information. In particular, we showed how intersection homology and the Euler–Poincaré characteristic can be expressed in terms of dynamical data, extending classical constructions from smooth Morse theory to a singular setting.

The results presented here  may serve as a starting point toward a broader Morse–Conley type theory for singular spaces. Several directions remain open. In particular, it would be natural to extend the present framework to more general classes of singular spaces, especially those with higher-dimensional singular strata and more general cone singularities, as well as to investigate more general dynamical systems beyond the singular Morse–Smale setting, where additional complexity and possible chaotic behavior arise.

From a more combinatorial perspective, the relationship between singular flows and their associated Lyapunov graphs also suggests further questions. Theorems such as Theorem~\ref{theorem_graph_map} and Theorem~\ref{theorem_degree} only begin to describe this interaction. Furthermore, problems concerning the realization of abstract Lyapunov graphs can be formulated and investigated in the near future.

Finally, we note that most of the constructions and results developed in this paper extend naturally to coefficients in a general abelian group. This opens the possibility of further generalizations, including the study of nonorientable pseudomanifolds via $\mathbb{Z}_2$ coefficients.

\addtocontents{toc}{\protect\setcounter{tocdepth}{-1}}

\subsection*{Author contributions}
All authors contributed equally to the conception, development, and writing of the manuscript. All authors read and approved the final version of the paper.

\subsection*{Conflict of interest}
The authors declare no potential conflict of interests.

\addtocontents{toc}{\protect\setcounter{tocdepth}{1}}

\bibliographystyle{amsplain} 
\bibliography{bibliography}   
\end{document}